\documentclass[12pt,a4paper]{article}
\usepackage{amssymb,amsmath,amsthm,epsfig,verbatim,pstricks,url}
\usepackage{tikz,pgfplots,tikz-3dplot}
\usepackage{setspace}
\usetikzlibrary{fit}
\newcommand\addvmargin[1]{
\node[fit=(current bounding box),inner ysep=#1,inner xsep=0]{};
}
\usepackage{ifpdf}  
\newtheorem{thm}{\sc Theorem.}[section]
\newtheorem{lem}[thm]{\sc Lemma.}
\newtheorem{rem}[thm]{\sc Remark.}
\renewcommand{\theequation}{\arabic{section}.\arabic{equation}}
\newenvironment{AMS}%
{{\upshape\bfseries AMS subject classifications. }\ignorespaces}{}
\newenvironment{keywords}{{\upshape\bfseries Key words. }\ignorespaces}{}

\newcommand{\bRgeq}{{\mathbb R}_{\geq 0}}

\newcommand{\RZ}{{\mathbb R} \slash {\mathbb Z}}
\newcommand{\bR}{{\mathbb R}}
\newcommand{\bN}{{\mathbb N}}
\newcommand{\bZ}{\mathbb{Z}}

\newcommand{\spa}{\operatorname{span}}

\newcommand{\Gauss}{{\mathcal{K}}}
\newcommand{\doctorkappa}{\mathfrak{K}}
\newcommand{\ratio}{{\mathfrak r}}

\newcommand{\dH}[1]{\;{\rm d}{\mathcal{H}}^{#1}} 

\newcommand{\drho}{\;{\rm d}\rho}

\newcommand{\Vh}{\underline{V}^h}

\newcommand{\Vhpartial}{\underline{V}^h_\partial}
\newcommand{\Vhpartialzero}{\underline{V}^h_{\partial_0}}
\newcommand{\Vpartial}{\underline{V}_\partial}
\newcommand{\Vpartialzero}{\underline{V}_{\partial_0}}
\newcommand{\Wpartialzero}{L^2(I)}
\newcommand{\vecWpartialzero}{[L^2(I)]^2}
\newcommand{\Whpartialzero}{W^h_{(\partial_0)}}
\newcommand{\vecWhpartialzero}{\underline{W}_{(\partial_0)}^h}

\newcommand{\dd}[1]{\frac{\rm d}{{\rm d}#1}}
\newcommand{\ddt}{\dd{t}}

\newcommand{\ek}{e}
\newcommand{\ttau}{\Delta t}
\newcommand{\sliprho}{\widehat\varrho_{\partial\mathcal{S}}}
\newcommand{\BGNmckappa}{\mathcal{A}}
\newcommand{\GDmckappa}{\mathcal{B}}
\newcommand{\BGNmc}{\mathcal{C}}
\newcommand{\GDmc}{\mathcal{D}}

\newcommand{\normal}{{\rm n}}
\def\epsilon{\varepsilon} 

\newcommand{\mint}{\textstyle\mints\displaystyle}
\newcommand{\mints}{\int\!\!\!\!\!{\rm-}}

\newcommand{\errorXx}{\|\Gamma - \Gamma^h\|_{L^\infty}}
\begin{document}
\title{
Variational discretization of \\ axisymmetric curvature flows
}
\author{John W. Barrett\footnotemark[2] \and 
        Harald Garcke\footnotemark[3]\ \and 
        Robert N\"urnberg\footnotemark[2]}

\renewcommand{\thefootnote}{\fnsymbol{footnote}}
\footnotetext[2]{Department of Mathematics, 
Imperial College London, London, SW7 2AZ, UK}
\footnotetext[3]{Fakult{\"a}t f{\"u}r Mathematik, Universit{\"a}t Regensburg, 
93040 Regensburg, Germany}

\date{}

\maketitle

\begin{abstract}
We present natural axisymmetric variants of schemes
for curvature flows introduced earlier by the present
authors and analyze them in detail. 
Although numerical methods for geometric flows have
been used frequently in axisymmetric settings, numerical analysis
results so far are rare.
In this paper, we present stability, equidistribution, existence and uniqueness
results for the introduced approximations. 
Numerical computations show that these schemes are
very efficient in computing numerical solutions of geometric flows
as only a spatially one-dimensional problem has to be solved. The good
mesh properties of the schemes also allow them to compute in very complex
axisymmetric geometries.
\end{abstract} 


\renewcommand{\thefootnote}{\arabic{footnote}}

\setcounter{equation}{0}
\section{Introduction} \label{sec:intro}
Numerical approximations of curvature flows such as the mean curvature
flow and the Gauss curvature flow have been studied intensively 
during the last 30 years. In many situations the axisymmetry of these
geometric flows can be used to reduce the dimension of the governing 
equations, and so numerical methods have been used frequently in such 
axisymmetric settings. However, results on the numerical analysis
of such schemes so far are rare. 
In this paper we present parametric finite element
approximations for axisymmetric curvature flows, and carefully analyse their
properties.

In general, in 
curvature driven evolution equations the normal velocity
of a hypersurface in 
$\bR^3$ is given by an expression involving the mean and/or the Gauss
curvature of the surface. Evolving surfaces are of interest in geometry,
and they can appear in application areas such as materials science, for example
as grain boundaries. In addition, evolution laws involving the curvature 
of the surface arise in situations, where surface quantities are coupled
to the surrounding volume by additional fields, which for example arises
in the evolution of phase boundaries or in two-phase flow. In any case
solving the evolution law for the surface with a stable discretization
of curvature is a corner stone of a reliable and efficient numerical
method.

Approaches to solve surface evolution equations numerically involve
different descriptions of the evolving surface. Traditionally level set
methods, phase field methods or parametric front tracking methods have
been used.
For example, parametric finite element approximations of curvature flows
have been considered in \cite{Dziuk91,gflows3d,ElliottF17,KovacsLL18arxiv}.
We refer to the review paper \cite{DeckelnickDE05}, and the references
therein, for further information on numerical methods for general geometric
evolution equations.

In this paper we aim to numerically compute a family of hypersurfaces
 $(\mathcal{S}(t))_{t \geq 0} \subset \bR^3$, which we
later assume to be axisymmetric, and which fulfills a geometric
evolution law involving its principal curvatures. We will focus on
the mean curvature flow, which for $\mathcal{S}(t)$ is given by the 
evolution law
\begin{equation} \label{eq:mcfS}
\mathcal{V}_{\mathcal{S}} = k_m\qquad\text{ on } \mathcal{S}(t)\,,
\end{equation}
and which is the $L^2$--gradient flow for $\mathcal{H}^2(\mathcal{S}(t))$, 
since
\begin{equation*} 
\ddt\, \mathcal{H}^2(\mathcal{S}(t))
= - \int_{\mathcal{S}(t)} k_m\,\mathcal{V}_{\mathcal{S}} \dH{2}
= - \int_{\mathcal{S}(t)} (\mathcal{V}_{\mathcal{S}})^2 \dH{2} 
\end{equation*}
for surfaces without boundary.
Here $\mathcal{V}_{\mathcal{S}}$ denotes the normal velocity of 
$\mathcal{S}(t)$ in the direction of the normal $\vec\normal_{\mathcal{S}}$.
Moreover, $k_m$ is the mean curvature of $\mathcal{S}(t)$, i.e.\ the sum of the
principal curvatures of $\mathcal{S}(t)$, see \cite{Mantegazza11} 
for an introduction to the mean curvature flow.

We also consider the nonlinear mean curvature flow
\begin{equation} \label{eq:nlmcf}
\mathcal{V}_{\mathcal{S}} = f(k_m)
\quad\text{on }\ \mathcal{S}(t)\,,
\end{equation}
where $f:(a,b)\rightarrow\mathbb{R}$ with $-\infty\leq a<b\leq\infty$,
is a strictly monotonically increasing continuous function, as well as the
volume preserving variant
\begin{equation} \label{eq:nlmcf2}
\mathcal{V}_{\mathcal{S}} = f(k_m)
- \frac{\int_{\mathcal{S}} f(k_m) \dH{2}}{\int_{\mathcal{S}} 1 \dH{2}}
\quad\text{on }\ \mathcal{S}(t)\,.
\end{equation}
Possible choices for $f$ are
\begin{subequations}
\begin{equation}\label{eq:fbeta}
f(r) =|r|^{\beta-1}r, \quad \beta\in\mathbb{R}_{>0}\,,
\end{equation}
or
\begin{equation}\label{eq:fimcf}
f(r) = - r^{-1} 
\end{equation}
for the inverse mean curvature flow. These two choices have applications 
for example in image processing or in general relativity, see
\cite{MikulaS01,HuiskenI01} and the references therein. 
Of course, (\ref{eq:nlmcf}) with
\begin{equation} \label{eq:fmcf}
f(r) = r
\end{equation}
\end{subequations}
collapses to (\ref{eq:mcfS}). 

If $\Omega(t)$ denotes the region enclosed by $\mathcal{S}(t)$, i.e.\
$\mathcal{S}(t) = \partial\Omega(t)$, then the flow (\ref{eq:nlmcf2}) is such
that
\begin{equation} \label{eq:dVdt}
\ddt\,\mathcal{L}^3(\Omega(t)) = 
\int_{\mathcal{S}(t)} \mathcal{V}_{\mathcal{S}} \dH{2}
 = 0\,,
\end{equation}
where here we assume that 
$\vec\normal_{\mathcal{S}}$ is the outer normal
to $\Omega(t)$ on $\mathcal{S}(t)$. This justifies the expression volume
preserving flow. These flows are of interest in geometry and we refer to
\cite{Huisken87,Athanassenas97,Cabezas-RivasS10,Hartley16} 
for more information.

More generally, we can also consider flows of the form
\begin{equation} \label{eq:Fmg}
\mathcal{V}_{\mathcal{S}} = F(k_m, k_g) = F(k_1 + k_2, k_1\,k_2)
\quad\text{on }\ \mathcal{S}(t)\,,
\end{equation}
where $k_g = k_1\,k_2$ denotes the Gauss curvature of $\mathcal{S}(t)$,
with $k_1$ and $k_2$ the two principal curvatures. 
Of course, (\ref{eq:Fmg}) with $F(r,s) = f(r)$ reduces to (\ref{eq:nlmcf}). On
the other hand, the choice $F(r,s) = -s$, for closed surfaces, leads to 
the Gauss curvature flow
\begin{equation} \label{eq:Gaussflow}
\mathcal{V}_{\mathcal{S}} = - k_g \quad\text{on }\ \mathcal{S}(t)\,,
\end{equation}
see e.g.\ \cite[(1.14)]{willmore}, where in (\ref{eq:Gaussflow}) we again
assume that $\vec\normal_{\mathcal{S}}$ is the outer normal
to $\Omega(t)$ on $\mathcal{S}(t)$. Such flows have found considerable
interest in geometry recently and we refer to
\cite{Gerhardt90,McCoyMW14,McCoyMW15,Urbas90} for more information.
One reason why the Gauss curvature flow is of particular interest, 
is because this flow allows to study the fate of the rolling stones, 
see \cite{Andrews99}.  

In this paper, we consider the case that $\mathcal{S}(t)$ is 
an axisymmetric surface, that is rotationally symmetric with respect to the
$x_2$--axis. We further assume that $\mathcal{S}(t)$ is made up of a single 
connected component, with or without boundary. Clearly, in the latter case 
the boundary $\partial\mathcal{S}(t)$ of $\mathcal{S}(t)$ consists of either
one or two circles that each lie within a hyperplane that is parallel to the
$x_1-x_3$--plane. For the evolving family of surfaces we allow for the
following types of boundary conditions. A boundary circle may assumed to be
fixed, it may be allowed to move vertically along the boundary of a fixed
infinite cylinder that is aligned with the axis of rotation, 
or it may be allowed to expand and shrink within a 
hyperplane that is parallel to the $x_1-x_3$--plane. 
Depending on the postulated free energy, certain
angle conditions will arise where $\mathcal{S}(t)$ meets the external 
boundary. If the free energy is just surface area,
$\mathcal{H}^2(\mathcal{S}(t))$, then a $90^\circ$ degree contact angle
condition arises. We refer to Section~\ref{sec:1} below for further details,
in particular with regard to more general contact angles.

The dimensionally reduced formulation has several severe advantages both 
analytically as well as numerically. In analysis it has been used for
example to study the onset of singularities, see 
\cite{Huisken90,DziukK91,Matioc07,LeCrone14,McCoyMW15} and other
singularity formation mechanisms, see \cite{BasaSS94,BernoffBW98}. 
Numerically it leads to equations
which are far easier to solve and at the same time problems with the
mesh topology do not occur. Therefore, axisymmetric settings have been
frequently used for numerical computations of surface evolutions.
For example, graph formulations for axisymmetric geometric evolution laws
have been considered in \cite{ColemanFM96,DeckelnickDE03,DeckelnickS10}, 
while a finite difference approximation of a parametric description for
the evolution of general axisymmetric surfaces has been studied in
\cite{MayerS02}. Hence the latter is closely related to the presented
work, although we stress that it does not contain any numerical analysis.
Moreover, also more complex problems such as for example two phase flows or
biomembranes, in which also curvature effects play a role, have been
treated in an axially symmetric setting. We refer to 
\cite{GanesanT08,VeerapaneniGBZ09,HuKL14,CoxL15,Zhao17preprint}, and we expect
that our approach will have an impact on such more complex evolutions as
well.
In terms of the numerical analysis for
the approximation of axisymmetric surface evolutions only very few results have
appeared in the literature so far, see e.g.\ 
\cite{DeckelnickDE03,DeckelnickS10} in the context of a graph formulation
for the higher order curvature flows surface diffusion and Willmore flow, 
respectively.
To the best of our knowledge, our paper contains the first stability
results for fully discrete approximations of axisymmetric mean curvature flow.
In addition, we consider the numerical analysis of approximations
for axisymmetric higher order flows, such as surface diffusion and 
Willmore flow, in the recently appeared article \cite{axisd}.

The present authors in the last ten years introduced parametric finite
element methods
for geometric evolution equations which have the property that the mesh
generically behaves well during the evolution. We also refer to the recent
work \cite{ElliottF17} for a method which also leads to good meshes.
This is an advantage compared to earlier front tracking approaches in which 
often the meshes degenerated during the evolution such that the computations
had to be stopped. In a series of papers, 
\cite{triplej,triplejMC,gflows3d,willmore,ejam3d}, we were able to analyze 
mesh properties and showed stability results. In particular, in two
dimensions a semi-discrete version of the method led to
equidistribution of mesh points.
In this paper we introduce a parametric finite element method for
the axisymmetric formulations of the surface evolution equations
discussed above relying on ideas of our earlier work. However, a lot
of new techniques have to be introduced stemming partly from the fact
that close to the axis of rotation the equations, depending on the formulation,
become either singular or degenerate, and 
partly because one has to decide how to deal with the equidistribution 
property.
We will discuss several ways to handle these issues and will show 
stability, equidistribution, existence and uniqueness
results for the new schemes.

This paper is organised as follows. In Section~\ref{sec:1} we introduce several
weak formulations which will be crucial for the parametric finite
element approximations introduced later.
In Section~\ref{sec:sd} we derive semidiscrete, i.e. continuous in time
discrete in space discretizations, and discuss stability and
equidistribution properties.
Section~\ref{sec:fd} is devoted to fully discrete schemes for which
existence results are shown for linear as well as nonlinear variants, as
well as uniqueness results for linear schemes. In addition, we show stability
for a fully discrete, mildly nonlinear discretization. Finally, we
present several numerical results demonstrating that the majority of
the schemes led to efficient, reliable results for mean curvature flow
as well as for fully nonlinear curvature flows including its mass
preserving variants.

\setcounter{equation}{0}
\section{Weak formulations} \label{sec:1}
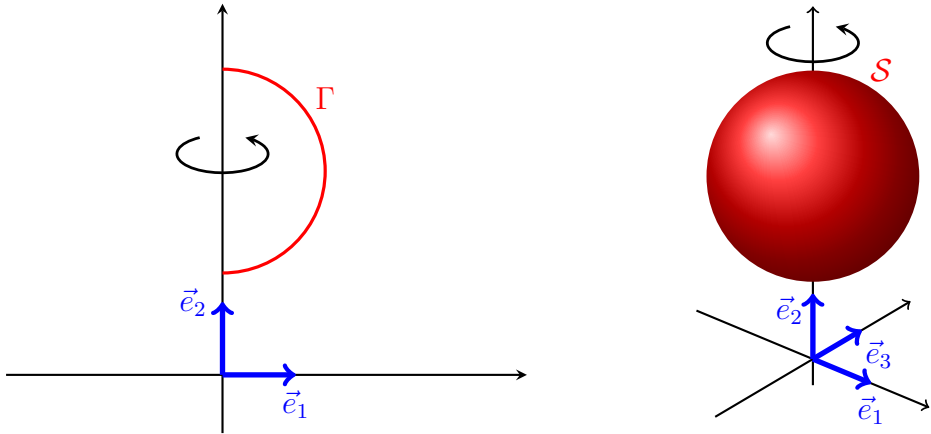
\begin{figure}
\center
\newcommand{\AxisRotator}[1][rotate=0]{%
    \tikz [x=0.25cm,y=0.60cm,line width=.2ex,-stealth,#1] \draw (0,0) arc (-150:150:1 and 1);%
}
\begin{tikzpicture}[every plot/.append style={very thick}, scale = 1]
\begin{axis}[axis equal,axis line style=thick,axis lines=center, xtick style ={draw=none}, 
ytick style ={draw=none}, xticklabels = {}, 
yticklabels = {}, 
xmin=-0.2, xmax = 0.8, ymin = -0.4, ymax = 2.55]
after end axis/.code={  
   \node at (axis cs:0.0,1.5) {\AxisRotator[rotate=-90]};
   \draw[blue,->,line width=2pt] (axis cs:0,0) -- (axis cs:0.5,0);
   \draw[blue,->,line width=2pt] (axis cs:0,0) -- (axis cs:0,0.5);
   \node[blue] at (axis cs:0.5,-0.2){$\vec\ek_1$};
   \node[blue] at (axis cs:-0.2,0.5){$\vec\ek_2$};
   \draw[red,very thick] (axis cs: 0,0.7) arc[radius = 70, start angle= -90, end angle= 90];
   \node[red] at (axis cs:0.7,1.9){$\Gamma$};
}
\end{axis}
\end{tikzpicture} \qquad \qquad
\tdplotsetmaincoords{120}{50}
\begin{tikzpicture}[scale=2, tdplot_main_coords,axis/.style={->},thick]
\draw[axis] (-1, 0, 0) -- (1, 0, 0);
\draw[axis] (0, -1, 0) -- (0, 1, 0);
\draw[axis] (0, 0, -0.2) -- (0, 0, 2.7);
\draw[blue,->,line width=2pt] (0,0,0) -- (0,0.5,0) node [below] {$\vec\ek_1$};
\draw[blue,->,line width=2pt] (0,0,0) -- (0,0.0,0.5);
\draw[blue,->,line width=2pt] (0,0,0) -- (0.5,0.0,0);
\node[blue] at (0.2,0.4,0.1){$\vec\ek_3$};
\node[blue] at (0,-0.2,0.3){$\vec\ek_2$};
\node[red] at (0.7,0,1.9){$\mathcal{S}$};
\node at (0.0,0.0,2.4) {\AxisRotator[rotate=-90]};

\tdplottransformmainscreen{0}{0}{1.4}
\shade[tdplot_screen_coords, ball color = red] (\tdplotresx,\tdplotresy) circle (0.7);
\end{tikzpicture}
\caption{Sketch of $\Gamma$ and $\mathcal{S}$, as well as 
the unit vectors $\vec\ek_1$, $\vec\ek_2$ and $\vec\ek_3$.}
\label{fig:sketch}
\end{figure}

Let $\RZ$ be the periodic interval $[0,1]$, and set
\[
I = \RZ\,, \text{ with } \partial I = \emptyset\,,\quad \text{or}\quad
I = (0,1)\,, \text{ with } \partial I = \{0,1\}\,.
\]
We consider the axisymmetric situation, where 
$\vec x(t) : \overline I \to \bR^2$ 
is a parameterization of $\Gamma(t)$. 
Throughout $\Gamma(t)$ represents the generating curve of a
surface $\mathcal{S}(t)$ 
that is axisymmetric with respect to the $x_2$--axis, see
Figure~\ref{fig:sketch}. In particular, on defining
\begin{equation*} 
\vec\Pi_3^3(r, z, \theta) = 
(r \cos\theta, z, r \sin\theta)^T 
\quad\text{for}\quad r\in \bRgeq\,,\ z \in \bR\,,\ \theta \in [0,2\,\pi]
\end{equation*}
and
\begin{equation*} 
\Pi_2^3(r, z) = \{\vec\Pi_3^3(r, z, \theta) : \theta \in [0,2\,\pi)\}\,,
\end{equation*}
we have that
\begin{equation} \label{eq:SGamma}
\mathcal{S}(t)  = 
\bigcup_{(r,z)^T \in \Gamma(t)} \Pi_2^3(r, z)
= \bigcup_{\rho \in \overline I} \Pi_2^3(\vec x(\rho,t))\,.
\end{equation}
Here we allow $\Gamma(t)$ to be either a closed curve, parameterized over
$\RZ$, which corresponds to $\mathcal{S}(t)$ being a genus-1 surface
without boundary.
Or $\Gamma(t)$ may be an open curve, parameterized over $[0,1]$.
Then $\Gamma(t)$ has two endpoints, and each endpoint can either correspond to
an interior point of $\mathcal{S}(t)$, or to a boundary circle of
$\mathcal{S}(t)$. Endpoints of $\Gamma(t)$ that correspond to an interior point
of the surface $\mathcal{S}(t)$ are attached to the $x_2$--axis, 
on which they can freely move up and down. For example, if both endpoints
of $\Gamma(t)$ are attached to the $x_2$--axis,
then $\mathcal{S}(t)$ is a genus-0 surface without boundary.
If only one end of $\Gamma(t)$ is attached to the $x_2$--axis, 
then $\mathcal{S}(t)$ is an open surface with boundary, where the boundary
consists of a single connected component.
If no endpoint of $\Gamma(t)$ is attached to the $x_2$--axis, 
then $\mathcal{S}(t)$ is an open surface with boundary, where the boundary
consists of two connected components. 

In particular, we always assume that, for all $t \in [0,T]$,
\begin{subequations}
\begin{align} 
\vec x(\rho,t) \,.\,\vec\ek_1 & > 0 \quad 
\forall\ \rho \in \overline I\setminus \partial_0 I\,,\label{eq:xpos} \\
\vec x(\rho,t) \,.\,\vec\ek_1 &= 0 \quad 
\forall\ \rho \in \partial_0 I\,,\label{eq:axibc} \\
\vec x_t(\rho,t) \,.\,\vec\ek_i &= 0 \quad 
\forall\ \rho \in \partial_i I \,, \ i =1,2\,,  \label{eq:freeslipbc} \\
\vec x_t(\rho,t) &= \vec 0 \quad 
\forall\ \rho \in \partial_D I \,, \label{eq:noslipbc}
\end{align}
\end{subequations}
where $\partial_D I \cup \bigcup_{i=0}^2 \partial_i I = \partial I$ 
is a disjoint partitioning of $\partial I$, with $\partial_0 I$
denoting the subset of boundary points 
of $I$ that correspond to endpoints of $\Gamma(t)$ attached to the 
$x_2$--axis. Moreover, $\partial_D I \cup \bigcup_{i=1}^2 \partial_i I$ 
denotes the subset of boundary points of $I$ that model components of the
boundary of $\mathcal{S}(t)$. Here endpoints in $\partial_D I$ correspond
to fixed boundary circles of $\mathcal{S}(t)$, 
that lie within a hyperplane parallel to
the $x_1-x_3$--plane $\bR \times \{0\} \times \bR$.
Endpoints in $\partial_1 I$ correspond to boundary circles of $\mathcal{S}(t)$
that can move freely along the boundary of an infinite cylinder
that is aligned with the axis of rotation.
Endpoints in $\partial_2 I$ correspond to boundary circles of $\mathcal{S}(t)$
that can expand/shrink freely within a hyperplane parallel to
the $x_1-x_3$--plane $\bR \times \{0\} \times \bR$.
See Table~\ref{tab:diagram} for a visualization of the different types of 
boundary nodes. 
\begin{table}
\center
\caption{The different types of boundary nodes enforced by 
(\ref{eq:axibc})--(\ref{eq:noslipbc}).}
\begin{tabular}{ccc}
\hline
$\partial I$ & $\partial \Gamma$ & $\partial\mathcal{S}$ \\ \hline
$\partial_0 I$ &
\begin{tikzpicture}[scale=0.5,baseline=40]
\begin{axis}[axis equal,axis line style=thick,axis lines=center, 
xtick style ={draw=none}, ytick style ={draw=none}, xticklabels = {}, 
yticklabels = {}, xmin=-0.1, xmax = 2, ymin = -2, ymax = 2]
\addplot[mark=*,color=blue,mark size=6pt] coordinates {(0,1)};
\draw[<->,line width=3pt,color=red] (axis cs:0.3,0.5) -- (axis cs:0.3,1.5);
\node at (axis cs:0.5,-0.3){\Large$\vec\ek_1$};
\node at (axis cs:-0.3,0.5){\Large$\vec\ek_2$};
\end{axis}
\addvmargin{1mm}
\end{tikzpicture} 
& N/A \\ 
$\partial_D I$ &
\begin{tikzpicture}[scale=0.5,baseline=40]
\begin{axis}[axis equal,axis line style=thick,axis lines=center, 
xtick style ={draw=none}, ytick style ={draw=none}, xticklabels = {}, 
yticklabels = {}, xmin=-0.1, xmax = 2, ymin = -2, ymax = 2]
\addplot[mark=*,color=blue,mark size=6pt] coordinates {(2,1)};
\node at (axis cs:0.5,-0.3){\Large$\vec\ek_1$};
\node at (axis cs:-0.3,0.5){\Large$\vec\ek_2$};
\end{axis}
\addvmargin{1mm}
\end{tikzpicture} 
& 
\begin{tikzpicture}[baseline=0]
\draw[color=blue,thick] (0,0) circle [x radius=2cm, y radius=1cm];
\addvmargin{1mm}
\end{tikzpicture} 
\\ 
$\partial_1 I$ &
\begin{tikzpicture}[scale=0.5,baseline=40]
\begin{axis}[axis equal,axis line style=thick,axis lines=center, 
xtick style ={draw=none}, ytick style ={draw=none}, xticklabels = {}, 
yticklabels = {}, xmin=-0.1, xmax = 2, ymin = -2, ymax = 2]
\addplot[mark=*,color=blue,mark size=6pt] coordinates {(2,1)};
\draw[<->,line width=3pt,color=red] (axis cs:2.3,0.5) -- (axis cs:2.3,1.5);
\draw[thick,color=blue] (axis cs:2,-2) -- (axis cs:2,2);
\node at (axis cs:0.5,-0.3){\Large$\vec\ek_1$};
\node at (axis cs:-0.3,0.5){\Large$\vec\ek_2$};
\end{axis}
\addvmargin{1mm}
\end{tikzpicture} 
& 
\begin{tikzpicture}[baseline=0]
\draw[color=blue,thick] (0,0) circle [x radius=2cm, y radius=1cm];
\draw[<->,color=red,line width=2pt] (2.2,-0.5) -- (2.2,0.5);
\draw[color=blue,thin] (2,-1) -- (2,1);
\draw[color=blue,thin] (-2,-1) -- (-2,1);
\addvmargin{1mm}
\end{tikzpicture} 
\\ 
$\partial_2 I$ &
\begin{tikzpicture}[scale=0.5,baseline=40]
\begin{axis}[axis equal,axis line style=thick,axis lines=center, 
xtick style ={draw=none}, ytick style ={draw=none}, xticklabels = {}, 
yticklabels = {}, xmin=-0.1, xmax = 2, ymin = -2, ymax = 2]
\addplot[mark=*,color=blue,mark size=6pt] coordinates {(2,1)};
\draw[<->,line width=3pt,color=red] (axis cs:1.5,0.7) -- (axis cs:2.5,0.7);
\draw[thick,color=blue] (axis cs:0,1) -- (axis cs:4,1);
\node at (axis cs:0.5,-0.3){\Large$\vec\ek_1$};
\node at (axis cs:-0.3,0.5){\Large$\vec\ek_2$};
\end{axis}
\addvmargin{1mm}
\end{tikzpicture} 
& 
\begin{tikzpicture}[baseline=0]
\draw[color=blue,thick] (0,0) circle [x radius=2cm, y radius=1cm];
\draw[color=blue,thin] (0,0) circle [x radius=1.5cm, y radius=0.75cm];
\draw[color=blue,thin] (0,0) circle [x radius=2.5cm, y radius=1.25cm];
\draw[<->,color=red,line width=2pt] (1.5,0) -- (2.5,0);
\draw[<->,color=red,line width=2pt] (-2.5,0) -- (-1.5,0);
\addvmargin{1mm}
\end{tikzpicture} \\ \hline
\end{tabular}
\label{tab:diagram}
\end{table}%

On assuming that
\begin{equation} \label{eq:xrho}
|\vec x_\rho| \geq c_0 > 0 \qquad \forall\ \rho \in \overline I\,,
\end{equation}
we introduce the arclength $s$ of the curve, i.e.\ $\partial_s =
|\vec{x}_\rho|^{-1}\,\partial_\rho$, and set
\begin{equation} \label{eq:tau}
\vec\tau(\rho,t) = \vec x_s(\rho,t) = 
\frac{\vec x_\rho(\rho,t)}{|\vec x_\rho(\rho,t)|} \qquad \mbox{and}
\qquad \vec\nu(\rho,t) = -[\vec\tau(\rho,t)]^\perp\,,
\end{equation}
where $(\cdot)^\perp$ denotes a clockwise rotation by $\frac{\pi}{2}$.

On recalling (\ref{eq:SGamma}), we observe that the normal
$\vec\normal_{\mathcal{S}}$ on $\mathcal{S}(t)$ is given by
\begin{equation} \label{eq:nuS}
\vec\normal_{\mathcal{S}}(\vec\Pi_3^3(\vec x(\rho,t),\theta)) = 
\begin{pmatrix}
(\vec\nu(\rho,t)\,.\,\vec\ek_1)\,\cos\theta \\
\vec\nu(\rho,t)\,.\,\vec\ek_2 \\
(\vec\nu(\rho,t)\,.\,\vec\ek_1)\,\sin\theta 
\end{pmatrix}
 \quad\text{for}\quad
\rho \in \overline I\,,\ t \in [0,T]\,,\ \theta \in [0,2\,\pi)\,.
\end{equation}
Similarly, the normal velocity $\mathcal{V}_{\mathcal{S}}$ of $\mathcal{S}(t)$ 
in the direction $\vec\normal_{\mathcal{S}}$ is given by
\begin{equation} \label{eq:calVS}
\mathcal{V}_{\mathcal{S}} = \vec x_t(\rho,t)\,.\,\vec\nu(\rho,t) \quad\text{on }
\Pi_2^3(\vec x(\rho,t)) \subset \mathcal{S}(t)\,,
\quad \forall\ \rho \in \overline I\,,\ t \in [0,T]\,.
\end{equation}

For the curvature $\varkappa$ of $\Gamma(t)$ it holds that
\begin{equation} \label{eq:varkappa}
\varkappa\,\vec\nu = \vec\varkappa = \vec\tau_s =
\frac1{|\vec x_\rho|} \left[ \frac{\vec x_\rho}{|\vec x_\rho|} \right]_\rho
.
\end{equation}

An important role in this paper is played by the surface area of the 
surface $\mathcal{S}(t)$, which is equal to
\begin{equation} \label{eq:A}
\mathcal{H}^2(\mathcal{S}(t)) = A(\vec x(t)) = 
2\,\pi\,\int_I \vec x(\rho,t)\,.\,\vec\ek_1\,|\vec x_\rho(\rho,t)|
\drho\,.
\end{equation}
Often the surface area, $A(\vec x(t))$, 
will play the role of the free energy in our paper. 
But for an open surface $\mathcal{S}(t)$, with boundary 
$\partial\mathcal{S}(t)$, 
we consider contact energy contributions which are discussed in
\cite{Finn86}, see also \cite[(2.21)]{ejam3d}. 
In the axisymmetric setting the relevant energy is given by
\begin{equation} \label{eq:E}
E(\vec x(t)) = A(\vec x(t))
+ 2\,\pi\,\sum_{p\in\partial_1 I} 
\sliprho^{(p)}\,(\vec x(p,t)\,.\,\vec\ek_1)\,\vec x(p,t)\,.\,\vec\ek_2
+ \pi\,\sum_{p\in\partial_2 I} 
\sliprho^{(p)}\,(\vec x(p,t)\,.\,\vec\ek_1)^2\,,
\end{equation}
where we recall from (\ref{eq:freeslipbc}) that, for $i=1,2$, either 
$\partial_i I = \emptyset$, $\{0\}$, $\{1\}$ or $\{0,1\}$.
In the above $\sliprho^{(p)} \in \bR$, for $p\in \{0,1\}$, are given constants.
Here $\sliprho^{(p)}$, for $p \in \partial_1 I$, 
denotes the change in contact energy
density in the direction of $-\vec\ek_2$, that the two phases separated by
the interface $\mathcal{S}(t)$ have with the infinite cylinder at
the boundary circle of $\mathcal{S}(t)$ represented by $\vec x(p,t)$. 
Similarly, $\sliprho^{(p)}$, for $p \in \partial_2 I$,
denotes the change in contact energy
density in the direction of $-\vec\ek_1$, that the two phases separated by
the interface $\mathcal{S}(t)$ have with the hyperplane 
$\bR\times\{0\}\times\bR$ 
at the boundary circle of $\mathcal{S}(t)$ represented by
$\vec x(p,t)$. 
These changes in contact energy lead to the contact angle conditions
\begin{subequations}
\begin{align} 
(-1)^p\,\vec\tau(p,t)\,.\,\vec\ek_2 &= \sliprho^{(p)} 
\qquad p \in \partial_1 I\,,\label{eq:mcbc1} \\
(-1)^p\,\vec\tau(p,t)\,.\,\vec\ek_1 &= \sliprho^{(p)} 
\qquad p \in \partial_2 I\,,\label{eq:mcbc2}
\end{align}
\end{subequations}
for all $t \in (0,T]$.
In most cases, the contact energies are assumed to be
the same, so that $\sliprho^{(0)}=\sliprho^{(1)}=0$, which leads to
$90^\circ$ contact angle conditions in (\ref{eq:mcbc1},b), and means that
(\ref{eq:E}) collapses to (\ref{eq:A}).
See \cite{ejam3d} for more details on contact angles and contact energies.
We note that a necessary condition to
admit a solution to (\ref{eq:mcbc1}) or to (\ref{eq:mcbc2}) is that 
\begin{equation} \label{eq:rhobound}
|\sliprho^{(p)}| \leq 1 \qquad p \in \{0,1\}\,.
\end{equation}
In addition, we observe that the energy (\ref{eq:E}) is not bounded from below
if $\sliprho^{(p)} \not=0 $ for $p\in\partial_1 I$ or if
$\sliprho^{(p)}<0$ for $p\in\partial_2 I$.

For later use we note that
\begin{align}
\ddt\, E(\vec x(t)) & = 2\,\pi\,\int_I \left[\vec x_t\,.\,\vec\ek_1
+ \vec x\,.\,\vec\ek_1\,\frac{(\vec x_t)_\rho\,.\,\vec x_\rho}{|\vec x_\rho|^2}
\right] |\vec x_\rho| \drho \nonumber \\ & \qquad
+ 2\,\pi\,\sum_{p\in\partial_1 I} 
\sliprho^{(p)}\left[
(\vec x_t(p,t)\,.\,\vec\ek_1)\,\vec x(p,t)\,.\,\vec\ek_2
+ (\vec x(p,t)\,.\,\vec\ek_1)\,\vec x_t(p,t)\,.\,\vec\ek_2
\right]
\nonumber \\ & \qquad
+ 2\,\pi\,\sum_{p\in\partial_2 I} 
\sliprho^{(p)}\,(\vec x(p,t)\,.\,\vec\ek_1)\,\vec x_t(p,t)\,.\,\vec\ek_1
\,.
\label{eq:dEdt}
\end{align}
Moreover, we recall that expressions for the mean curvature and the 
Gauss curvature of $\mathcal{S}(t)$ are given by
\begin{equation} \label{eq:meanGaussS}
\varkappa_{\mathcal{S}} = 
\varkappa - 
\frac{\vec\nu\,.\,\vec\ek_1}{\vec x\,.\,\vec\ek_1}
\quad\text{and}\quad
\Gauss_{\mathcal{S}} = -
\varkappa\,\frac{\vec\nu\,.\,\vec\ek_1}{\vec x\,.\,\vec\ek_1}
\quad\text{on }\ \overline{I}\,,
\end{equation}
respectively; see e.g.\ \cite[(6)]{CoxL15}. More precisely, if $k_m$ and $k_g$
denote the mean and Gauss curvatures of $\mathcal{S}(t)$, then
\begin{equation*} 
k_m = \varkappa_{\mathcal{S}}(\rho,t) 
\ \text{ and }\
k_g = \Gauss_{\mathcal{S}}(\rho,t) 
\quad\text{on }
\Pi_2^3(\vec x(\rho,t)) \subset \mathcal{S}(t)\,,
\quad \forall\ \rho \in \overline I\,,\ t \in [0,T]\,.
\end{equation*}
In the literature, the two terms making up $\varkappa_{\mathcal{S}}$
in (\ref{eq:meanGaussS}) 
are often referred to as in-plane and azimuthal curvatures,
respectively, with their sum being equal to the mean curvature.
We note that combining (\ref{eq:meanGaussS}) and (\ref{eq:varkappa}) yields 
that
\begin{equation} \label{eq:kappaS}
\varkappa_{\mathcal{S}}\,\vec\nu = 
\vec x_{ss}
- \frac{\vec\nu\,.\,\vec\ek_1}{\vec x\,.\,\vec\ek_1}\,\vec\nu = 
\frac1{|\vec x_\rho|} \left[ \frac{\vec x_\rho}{|\vec x_\rho|} \right]_\rho
- \frac{\vec\nu\,.\,\vec\ek_1}{\vec x\,.\,\vec\ek_1}\,\vec\nu\,.
\end{equation}
It follows from (\ref{eq:kappaS}) and (\ref{eq:tau}) that
\begin{equation} \label{eq:veckappaS}
\vec x\,.\,\vec\ek_1\,\varkappa_{\mathcal{S}}\,\vec\nu = 
(\vec x\,.\,\vec\ek_1)\,\vec x_{ss} - (\vec\nu\,.\,\vec\ek_1)\,\vec\nu 
= ((\vec x\,.\,\vec\ek_1)\,\vec x_s)_s - \vec\ek_1 \,.
\end{equation}
A weak formulation of (\ref{eq:veckappaS}) will form the basis of our stable
approximations for mean curvature flow and surface diffusion.
Clearly, for a smooth surface with bounded mean curvature it follows from
(\ref{eq:kappaS}) that 
\begin{equation} \label{eq:bcnu}
\vec\nu(\rho,t) \,.\,\vec\ek_1 = 0
\qquad \forall\ \rho \in \partial_0 I\,,\quad \forall\ t\in[0,T]\,,
\end{equation}
which, on recalling (\ref{eq:tau}), is clearly equivalent to
\begin{equation} \label{eq:bc}
\vec x_\rho(\rho,t) \,.\,\vec\ek_2 = 0
\qquad \forall\ \rho \in \partial_0 I\,,\quad \forall\ t\in[0,T]\,.
\end{equation}
A precise derivation of (\ref{eq:bc}) in the context of a weak formulation
of (\ref{eq:kappaS}) will be given in the Appendix~\ref{sec:A1}.

\subsection{Mean curvature flow}
In terms of the axisymmetric description of $\mathcal{S}(t)$, the evolution law
(\ref{eq:mcfS}) can be written as
\begin{subequations}
\begin{equation} \label{eq:xt}
\vec x_t\,.\,\vec\nu = \varkappa_{\mathcal{S}} = \varkappa - 
\frac{\vec\nu\,.\,\vec\ek_1}{\vec x\,.\,\vec\ek_1}
\qquad\text{on } I\,,
\end{equation}
with, on recalling (\ref{eq:axibc}--d),
\begin{align} \label{eq:xtbc}
\vec x_t(\rho,t)\,.\,\vec\ek_1 & = 0 \quad
\forall\ \rho\in\partial_0 I\,, \quad 
\vec x_t(\rho,t)\,.\,\vec\ek_i = 0 \quad
\forall\ \rho\in\partial_i I\,,\ i = 1,2\,, \nonumber \\  
\vec x_t(\rho,t) & = \vec 0 \quad \forall\ \rho\in\partial_D I\,, 
\qquad \forall\ t\in[0,T]\,,
\end{align}
\end{subequations}
as well as (\ref{eq:bc}) and (\ref{eq:mcbc1},b).

Let 
\begin{align*} 
\Vpartialzero & = \{ \vec\eta \in [H^1(I)]^2 : \vec\eta(\rho)\,.\,\vec\ek_1 = 0
\quad \forall\ \rho \in \partial_0 I\}\,, \nonumber \\ 
\Vpartial & = \{\eta\in \Vpartialzero : \vec\eta(\rho)\,.\,\vec\ek_i = 0
\quad \forall\ \rho \in \partial_i I\,,\ i=1,2, \quad
\vec\eta(\rho) = \vec 0 \quad \forall\ \rho \in \partial_D I\}\,.
\end{align*}
Then we consider the following weak formulation of (\ref{eq:xt},b),
on recalling (\ref{eq:varkappa}).

$(\BGNmckappa)$:
Let $\vec x(0) \in \Vpartialzero$. For $t \in (0,T]$
find $\vec x(t) \in [H^1(I)]^2$, with
$\vec x_t(t) \in \Vpartial$, and $\varkappa(t)\in L^2(I)$ 
such that
\begin{subequations}
\begin{align}
& \int_I \vec x_t\,.\,\vec\nu\,\chi\,|\vec x_\rho|\drho
= \int_I \left(\varkappa - \frac{\vec\nu\,.\,\vec\ek_1}{\vec x\,.\,\vec\ek_1}
\right) \chi\,|\vec x_\rho| \drho \quad \forall\ \chi \in L^2(I)\,,
\label{eq:xtweak} \\
& \int_I \varkappa\,\vec\nu\,.\,\vec\eta\, |\vec x_\rho| \drho
+ \int_I (\vec x_\rho\,.\,\vec\eta_\rho)\,|\vec x_\rho|^{-1} \drho = 
- \sum_{i=1}^2 
\sum_{p \in \partial_i I} \sliprho^{(p)}\,\vec\eta(p)\,.\,\vec\ek_{3-i}
\quad \forall\ \vec\eta \in \Vpartial\,.
\label{eq:varkappaweak}
\end{align}
\end{subequations}
We note that (\ref{eq:varkappaweak}) weakly imposes (\ref{eq:bc})
and (\ref{eq:mcbc1},b).
We observe that (\ref{eq:xt}) degenerates for $\vec x\,.\,\vec\ek_1 = 0$,
i.e.\ when $\rho \in \partial_0 I$. Hence this degeneracy is balanced by
the condition (\ref{eq:bcnu}). In fact, 
on recalling (\ref{eq:varkappa}) it holds that
\begin{align}
\lim_{\rho\to \rho_0} 
\frac{\vec\nu(\rho,t)\,.\,\vec\ek_1}{\vec x(\rho,t)\,.\,\vec\ek_1} &
= \lim_{\rho\to \rho_0} 
\frac{\vec\nu_\rho(\rho,t)\,.\,\vec\ek_1}{\vec x_\rho(\rho,t)\,.\,\vec\ek_1}
= \vec\nu_s(\rho_0,t)\,.\,\vec\tau(\rho_0,t) 
= -\varkappa(\rho_0,t)
\nonumber \\ & \hspace{6cm}
\quad \forall\ \rho_0\in\partial_0 I\,,\
\forall\ t \in [0,T]\,.
\label{eq:bclimit}
\end{align}
We remark that the weak formulation $(\BGNmckappa)$ is close in spirit to the
weak formulations introduced in \cite{triplejMC,gflows3d} for mean curvature
flow. In particular, the tangential component of $\vec x_t$ is not prescribed,
which on the discrete level leads to an equidistribution property.

Choosing $\vec\eta = (\vec x\,.\,\vec\ek_1)\,\vec x_t \in \Vpartial$
in (\ref{eq:varkappaweak})
and $\chi = (\vec x\,.\,\vec\ek_1)\,(\vec x_t\,.\,\vec\nu)$ in 
(\ref{eq:xtweak}), 
we obtain on recalling (\ref{eq:dEdt}), $\vec x_t \in \Vpartial$, 
(\ref{eq:tau}) and (\ref{eq:xpos}) that
\begin{align}
& \ddt\, E(\vec x(t)) \nonumber \\ & \quad
= 2\,\pi\,\int_I \left[\vec x_t\,.\,\vec\ek_1
+ \vec x\,.\,\vec\ek_1\,\frac{(\vec x_t)_\rho\,.\,\vec x_\rho}{|\vec x_\rho|^2}
\right] |\vec x_\rho| \drho 
+ 2\,\pi \sum_{i=1}^2 \sum_{p\in\partial_i I} 
\sliprho^{(p)}\,(\vec x(p,t)\,.\,\vec\ek_1)\,\vec x_t(p,t)\,.\,\vec\ek_{3-i}
\nonumber \\ & \quad
= 2\,\pi\,\int_I \vec x_t\,.\left[\vec\ek_1
- (\vec\ek_1\,.\,\vec\tau)\,\vec\tau
\right] |\vec x_\rho| \drho
- 2\,\pi\,\int_I (\vec x\,.\,\vec\ek_1)\,\varkappa\,\vec\nu\,.\,
\vec x_t\, |\vec x_\rho| \drho
\nonumber \\ & \quad
= 2\,\pi\,\int_I 
(\vec x_t\,.\,\vec\nu)\,\vec\ek_1\,.\,\vec\nu\,|\vec x_\rho| \drho
- 2\,\pi\,\int_I (\vec x\,.\,\vec\ek_1)\,\varkappa\,
\vec x_t\,.\,\vec\nu\, |\vec x_\rho| \drho
\nonumber \\ & \quad
= - 2\,\pi\,\int_I \vec x\,.\,\vec\ek_1
\left[\varkappa - 
\frac{\vec\nu\,.\,\vec\ek_1}{\vec x\,.\,\vec\ek_1}\right]
\vec x_t\,.\,\vec\nu\, |\vec x_\rho| \drho
= - 2\,\pi\,\int_I \vec x\,.\,\vec\ek_1\left(
\vec x_t\,.\,\vec\nu\right)^2 |\vec x_\rho| \drho \leq 0\,.
\label{eq:gradflow}
\end{align}

An alternative strong formulation of mean curvature flow, in the axisymmetric
setting, to (\ref{eq:xt}) is given by
\begin{equation} \label{eq:mcdziuk}
\vec x_t = \vec\varkappa - 
\frac{\vec\nu\,.\,\vec\ek_1}{\vec x\,.\,\vec\ek_1}\,\vec\nu\,,
\end{equation}
with (\ref{eq:xtbc}), where we have recalled (\ref{eq:varkappa}).
We consider the following weak formulation of (\ref{eq:mcdziuk}). 

$(\GDmckappa)$:
Let $\vec x(0) \in \Vpartialzero$. For $t \in (0,T]$
find $\vec x(t) \in [H^1(I)]^2$, with $\vec x_t(t) \in \Vpartial$, and 
$\vec\varkappa(t)\in [L^2(I)]^2$ such that
\begin{subequations}
\begin{align}
& \int_I \vec x_t\,.\,\vec\chi\,|\vec x_\rho|\drho
= \int_I \left(\vec\varkappa\,.\,\vec\chi - 
\frac{\vec\nu\,.\,\vec\ek_1}{\vec x\,.\,\vec\ek_1}\,\vec\nu\,.\,
\vec\chi\right)|\vec x_\rho| \drho \qquad \forall\ 
\vec\chi \in [L^2(I)]^2\,,
\label{eq:Dziuka} \\
& \int_I \vec\varkappa\,.\,\vec\eta\, |\vec x_\rho| \drho
+ \int_I (\vec x_\rho\,.\,\vec\eta_\rho)\,|\vec x_\rho|^{-1} \drho =
- \sum_{i=1}^2 
 \sum_{p \in \partial_i I} \sliprho^{(p)}\,\vec\eta(p)\,.\,\vec\ek_{3-i}
\qquad \forall\ \vec\eta \in \Vpartial\,.
\label{eq:Dziukb}
\end{align}
\end{subequations}
Similarly to (\ref{eq:varkappaweak}), we observe that
(\ref{eq:Dziukb}) weakly imposes (\ref{eq:bc}) and (\ref{eq:mcbc1},b). 
We remark that the weak formulation $(\GDmckappa)$ in some sense
is close in spirit to the 
weak formulations introduced in \cite{Dziuk91,Dziuk94} for mean curvature
flow. In particular, the tangential component of $\vec x_t$ is fixed to
be zero, as the right hand side of (\ref{eq:Dziuka}) is normal, recall
(\ref{eq:varkappa}). 

Choosing $\vec\chi = \vec\eta = (\vec x\,.\,\vec\ek_1)\,\vec x_t \in \Vpartial$
in (\ref{eq:Dziuka},b), we obtain, similarly to (\ref{eq:gradflow}), that
\begin{align}
& \ddt\, E(\vec x(t)) \nonumber \\ & \quad
= 2\,\pi\,\int_I \left[\vec x_t\,.\,\vec\ek_1
+ \vec x\,.\,\vec\ek_1\,\frac{(\vec x_t)_\rho\,.\,\vec x_\rho}{|\vec x_\rho|^2}
\right] |\vec x_\rho| \drho 
+ 2\,\pi\sum_{i=1}^2
\sum_{p\in\partial_i I} 
\sliprho^{(p)}\,(\vec x(p,t)\,.\,\vec\ek_1)\,\vec x_t(p,t)\,.\,\vec\ek_{3-i}
\nonumber \\ & \quad
= 2\,\pi\,\int_I 
(\vec x_t\,.\,\vec\nu)\,\vec\ek_1\,.\,\vec\nu\,|\vec x_\rho| \drho
- 2\,\pi\,\int_I (\vec x\,.\,\vec\ek_1)\,\vec\varkappa\,.\,
\vec x_t\, |\vec x_\rho| \drho
\nonumber \\ & \quad
= - 2\,\pi\,\int_I \vec x\,.\,\vec\ek_1
\left[\vec\varkappa - 
\frac{\vec\nu\,.\,\vec\ek_1}{\vec x\,.\,\vec\ek_1}\,\vec\nu\right] .\,
\vec x_t\, |\vec x_\rho| \drho
= - 2\,\pi\,\int_I \vec x\,.\,\vec\ek_1\,|\vec x_t|^2 |\vec x_\rho| \drho
\leq 0\,.
\label{eq:GDgradflow}
\end{align}

We remark that it does not appear possible to mimic either 
(\ref{eq:gradflow}) for $(\BGNmckappa)$
or (\ref{eq:GDgradflow}) for $(\GDmckappa)$ on the
discrete level. 
Hence, in order to develop stable approximations, we
investigate alternative formulations based on (\ref{eq:veckappaS}).
The first formulation corresponds to the strong formulation
$(\vec x\,.\,\vec\ek_1)\,\vec x_t\,.\,\vec\nu = \vec
x\,.\,\vec\ek_1\,\varkappa_{\mathcal{S}}$, together with (\ref{eq:veckappaS}).

$(\BGNmc)$:
Let $\vec x(0) \in \Vpartialzero$. For $t \in (0,T]$
find $\vec x(t) \in [H^1(I)]^2$, with $\vec x_t(t) \in \Vpartial$, and 
$\varkappa_{\mathcal{S}}(t) \in \Wpartialzero$ such that
\begin{subequations}
\begin{align}
& \int_I (\vec x \,.\,\vec\ek_1)\,\vec x_t\,.\,\vec\nu\,
\chi\,|\vec x_\rho|\drho
= \int_I \vec x \,.\,\vec\ek_1\,\varkappa_{\mathcal{S}}\,\chi\,
|\vec x_\rho|\drho \qquad \forall\ \chi \in \Wpartialzero\,,
\label{eq:bgnnewa} \\
& \int_I \vec x \,.\,\vec\ek_1\,\varkappa_{\mathcal{S}}\,\vec\nu\,.\,
\vec\eta\,|\vec x_\rho|\drho
+ \int_I \left[\vec\eta \,.\,\vec\ek_1
+ \vec x\,.\,\vec\ek_1 \,\frac{\vec x_\rho\,.\,\vec\eta_\rho}{|\vec x_\rho|^2}
\right] |\vec x_\rho| \drho 
\nonumber \\ & \hspace{4cm}
= - \sum_{i=1}^2 \sum_{p \in \partial_i I} \sliprho^{(p)}\,
(\vec x(p,t)\,.\,\vec\ek_1)\,\vec\eta(p)\,.\,\vec\ek_{3-i}
\qquad \forall\ \vec\eta \in \Vpartial\,.
\label{eq:bgnnewb} 
\end{align}
\end{subequations}
The second formulation corresponds to the strong formulation
$(\vec x\,.\,\vec\ek_1)\,\vec x_t = \vec
x\,.\,\vec\ek_1\,\vec\varkappa_{\mathcal{S}}$, 
where $\vec\varkappa_{\mathcal{S}} = \varkappa_{\mathcal{S}}\,\vec\nu$,
together with (\ref{eq:veckappaS}).

$(\GDmc)$:
Let $\vec x(0) \in \Vpartialzero$. For $t \in (0,T]$
find $\vec x(t) \in [H^1(I)]^2$, with $\vec x_t(t) \in \Vpartial$, 
and $\vec\varkappa_{\mathcal{S}}(t) \in \vecWpartialzero$ such that
\begin{subequations}
\begin{align}
& \int_I (\vec x \,.\,\vec\ek_1)\,\vec x_t\,.\,\vec\chi\,|\vec x_\rho|\drho
= \int_I (\vec x \,.\,\vec\ek_1)\,\vec\varkappa_{\mathcal{S}}\,.\,\vec\chi\,
|\vec x_\rho|\drho \qquad\forall\ \vec\chi \in \vecWpartialzero\,,
\label{eq:Dziuknewa} \\ &
\int_I (\vec x \,.\,\vec\ek_1)\,\vec\varkappa_{\mathcal{S}}\,.\,
\vec\eta\,|\vec x_\rho|\drho
+ \int_I \left[\vec\eta \,.\,\vec\ek_1
+ \vec x\,.\,\vec\ek_1
\,\frac{\vec x_\rho\,.\,\vec\eta_\rho}{|\vec x_\rho|^2}
\right] |\vec x_\rho| \drho 
\nonumber \\ & \hspace{4cm}
= - \sum_{i=1}^2 \sum_{p \in \partial_i I} \sliprho^{(p)}\,
(\vec x(p,t)\,.\,\vec\ek_1)\,\vec\eta(p)\,.\,\vec\ek_{3-i}
\qquad \forall\ \vec\eta \in \Vpartial\,.
\label{eq:Dziuknewb} 
\end{align}
\end{subequations}
We note that the variational formulation for $\vec\varkappa_{\mathcal{S}}$
in (\ref{eq:Dziuknewb}) has previously been employed in 
\cite[p.\ 124]{GanesanT08}. 

Choosing $\vec\eta = \vec x_t \in \Vpartial$ in (\ref{eq:bgnnewb}) and 
$\chi = \varkappa_{\mathcal{S}}$ in (\ref{eq:bgnnewa}), we obtain
for the formulation $(\BGNmc)$, on recalling (\ref{eq:dEdt}), that
\begin{equation} \label{eq:bgnnewstab}
-\frac1{2\,\pi}\,
\ddt\, E(\vec x(t)) = \int_I \vec x \,.\,\vec\ek_1\,
|\varkappa_{\mathcal{S}}|^2\,|\vec x_\rho|\drho
= \int_I \vec x \,.\,\vec\ek_1\,
(\vec x_t\,.\,\vec\nu)^2\,|\vec x_\rho|\drho \geq 0\,.
\end{equation} 
Similarly, choosing $\vec\eta = \vec x_t \in\Vpartial$ in (\ref{eq:Dziuknewb})
and $\vec\chi = \vec\varkappa_{\mathcal{S}}$ in (\ref{eq:Dziuknewa}), 
we obtain for the formulation $(\GDmc)$ that
\begin{equation} \label{eq:Dziuknewstab}
-\frac1{2\,\pi}\,
\ddt\, E(\vec x(t)) = \int_I \vec x \,.\,\vec\ek_1\,
|\vec\varkappa_{\mathcal{S}}|^2\,|\vec x_\rho|\drho
= \int_I \vec x \,.\,\vec\ek_1\,|\vec x_t|^2\,|\vec x_\rho|\drho
\geq 0\,.
\end{equation} 
We observe that (\ref{eq:bgnnewb}) and (\ref{eq:Dziuknewb}) weakly impose
(\ref{eq:mcbc1},b). But, in contrast to (\ref{eq:varkappaweak}) and 
(\ref{eq:Dziukb}), it is not obvious  that they also weakly impose
(\ref{eq:bc}), due to the presence of the degenerate weight 
$\vec x\,.\,\vec\ek_1$. However, we show in the Appendix~\ref{sec:A1}
that in fact they also weakly impose (\ref{eq:bc}).

We also note that the formulation $(\BGNmc)$ is loosely related to 
$(\BGNmckappa)$, in the sense that the tangential component of $\vec x_t$ 
is not prescribed. But in contrast to $(\BGNmckappa)$, discretizations of
$(\BGNmc)$ cannot be shown to have an equidistribution property.
In a similar way, the formulation $(\GDmc)$ is loosely related to 
$(\GDmckappa)$, in the sense that the velocity $\vec x_t$ 
is purely in the normal direction, recall (\ref{eq:veckappaS}) and
(\ref{eq:varkappa}).
Finally, we observe that the variable $\varkappa_{\mathcal{S}}$ can be
eliminated from $(\BGNmc)$, by choosing $\chi = \vec\nu\,.\,\vec\eta$ 
in (\ref{eq:bgnnewa}) for $\vec\eta \in \Vpartial$, and then combining
(\ref{eq:bgnnewa}) and (\ref{eq:bgnnewb}). Similarly, 
$\vec\varkappa_{\mathcal{S}}$ can be eliminated from $(\GDmc)$
by choosing $\vec\chi = \vec\eta$ in (\ref{eq:Dziuknewa}) for 
$\vec\eta \in \Vpartial$, and then combining
(\ref{eq:Dziuknewa}) and (\ref{eq:Dziuknewb}).
We remark that the formulation $(\BGNmc)$, with the 
variable $\varkappa_{\mathcal{S}}$, as well as the formulation 
$(\BGNmckappa)$, are useful with a view towards introducing
numerical approximations of higher order flows, such as surface diffusion,
see \cite{axisd}.

\subsection{Nonlinear mean curvature flow}
It is a simple matter to extend the formulations $(\BGNmckappa)$ and $(\BGNmc)$
to the nonlinear flow (\ref{eq:nlmcf}). In principle this can also be achieved
for $(\GDmckappa)$ and $(\GDmc)$, but as the mean curvature needs to be
recovered from the mean curvature vector, the resulting formulations are less
natural. Hence we concentrate on $(\BGNmckappa)$ and $(\BGNmc)$.
For the former, replacing the right hand side in (\ref{eq:xtweak}) 
with $\int_I f(\varkappa - \frac{\vec\nu\,.\,\vec\ek_1}{\vec x\,.\,\vec\ek_1}
)\, \chi\,|\vec x_\rho| \drho$ yields a weak formulation 
for (\ref{eq:nlmcf}), which we call $(\BGNmckappa^f)$. 
Similarly, replacing $\varkappa_{\mathcal{S}}$ with
$f(\varkappa_{\mathcal{S}})$ in (\ref{eq:bgnnewa}) generalizes 
$(\BGNmc)$ to $(\BGNmc^f)$ for (\ref{eq:nlmcf}). 

Similarly to (\ref{eq:bgnnewstab}),
and using the same choices of $\vec\eta$ and $\chi$, 
it can be shown that solutions to $(\BGNmc^f)$ satisfy
\begin{equation} \label{eq:nlbgnnewstab}
-\frac1{2\,\pi}\,
\ddt\, E(\vec x(t)) = \int_I \vec x \,.\,\vec\ek_1\,
f(\varkappa_{\mathcal{S}})\,\varkappa_{\mathcal{S}}\,|\vec x_\rho|\drho
\,,
\end{equation} 
which yields stability if $f$ is monotonically increasing with $f(0)=0$.

Finally, we may also generalize these nonlinear formulations to the volume
preserving flow (\ref{eq:nlmcf2}). 

$(\BGNmckappa^{f,V})$:
Let $\vec x(0) \in \Vpartialzero$. For $t \in (0,T]$
find $\vec x(t) \in [H^1(I)]^2$, with
$\vec x_t(t) \in \Vpartial$, and $\varkappa(t)\in L^2(I)$ 
such that
\begin{subequations}
\begin{align}
& \int_I \vec x_t\,.\,\vec\nu\,\chi\,|\vec x_\rho|\drho
= \int_I f\left(\varkappa - \frac{\vec\nu\,.\,\vec\ek_1}{\vec x\,.\,\vec\ek_1}
\right) \chi\,|\vec x_\rho| \drho 
\nonumber \\ & \hspace{3.5cm}
- \frac{\int_I \vec x\,.\,\vec\ek_1\,
f(\varkappa - \frac{\vec\nu\,.\,\vec\ek_1}{\vec x\,.\,\vec\ek_1})\,
|\vec x_\rho| \drho}
{\int_I \vec x\,.\,\vec\ek_1\,|\vec x_\rho| \drho}\,
\int_I \chi\,|\vec x_\rho| \drho 
\quad \forall\ \chi \in L^2(I)\,,
\label{eq:nlmcfVa} \\
& \int_I \varkappa\,\vec\nu\,.\,\vec\eta\, |\vec x_\rho| \drho
+ \int_I (\vec x_\rho\,.\,\vec\eta_\rho)\,|\vec x_\rho|^{-1} \drho = 
- \sum_{i=1}^2
\sum_{p \in \partial_i I} \sliprho^{(p)}\,\vec\eta(p)\,.\,\vec\ek_{3-i}
\quad \forall\ \vec\eta \in \Vpartial\,.
\label{eq:nlmcfVb}
\end{align}
\end{subequations}
Choosing $\chi = 2\,\pi\,\vec x\,.\,\vec\ek_1$ in (\ref{eq:nlmcfVa}) yields,
on recalling (\ref{eq:dVdt}), that 
\begin{equation} \label{eq:dVdt0}
\pm \ddt\,\mathcal{L}^3(\Omega(t)) = 
\int_{\mathcal{S}(t)} \mathcal{V}_{\mathcal{S}} \dH{2}
= 2\,\pi\,\int_I (\vec x\,.\,\vec\ek_1)
\,\vec x_t\,.\,\vec\nu\,|\vec x_\rho|\drho = 0\,,
\end{equation}
where $\mathcal{S}(t) = \partial\Omega(t)$, 
and where the sign in (\ref{eq:dVdt0}) depends on whether 
$\vec\normal_{\mathcal{S}}$ is the outer or inner normal
to $\Omega(t)$ on $\mathcal{S}(t)$, recall (\ref{eq:nuS}) and (\ref{eq:calVS}). 

$(\BGNmc^{f,V})$:
Let $\vec x(0) \in \Vpartialzero$. For $t \in (0,T]$
find $\vec x(t) \in [H^1(I)]^2$, with $\vec x_t(t) \in \Vpartial$, and 
$\varkappa_{\mathcal{S}}(t) \in \Wpartialzero$ such that
\begin{subequations}
\begin{align}
& \int_I (\vec x \,.\,\vec\ek_1)\,\vec x_t\,.\,\vec\nu\,
\chi\,|\vec x_\rho|\drho
= \int_I \vec x \,.\,\vec\ek_1\,f(\varkappa_{\mathcal{S}})\,\chi\,
|\vec x_\rho|\drho
\nonumber \\ & \hspace{3cm}
- \frac{\int_I \vec x\,.\,\vec\ek_1\,f(\varkappa_{\mathcal{S}})\,
|\vec x_\rho| \drho}
{\int_I \vec x\,.\,\vec\ek_1\,|\vec x_\rho| \drho}\,
\int_I \vec x \,.\,\vec\ek_1\,\chi\,|\vec x_\rho| \drho 
 \qquad \forall\ \chi \in \Wpartialzero\,,
\label{eq:nlmcfVnewa} \\
& \int_I \vec x \,.\,\vec\ek_1\,\varkappa_{\mathcal{S}}\,\vec\nu\,.\,
\vec\eta\,|\vec x_\rho|\drho
+ \int_I \left[\vec\eta \,.\,\vec\ek_1
+ \vec x\,.\,\vec\ek_1 \,\frac{\vec x_\rho\,.\,\vec\eta_\rho}{|\vec x_\rho|^2}
\right] |\vec x_\rho| \drho 
\nonumber \\ & \hspace{4cm}
= - \sum_{i=1}^2 \sum_{p \in \partial_i I} \sliprho^{(p)}\,
(\vec x(p,t)\,.\,\vec\ek_1)\,\vec\eta(p)\,.\,\vec\ek_{3-i}
\qquad \forall\ \vec\eta \in \Vpartial\,.
\label{eq:nlmcfVnewb} 
\end{align}
\end{subequations}
Choosing $\chi = 2\,\pi$ in (\ref{eq:nlmcfVnewa}) yields (\ref{eq:dVdt0}), as
before.
Moreover, and similarly to (\ref{eq:nlbgnnewstab}), 
it can be shown for solutions of 
$(\BGNmc^{f,V})$ in the case (\ref{eq:fmcf}) that
\begin{equation} \label{eq:nlVbgnnewstab}
-\frac1{2\,\pi}\,
\ddt\, E(\vec x(t)) = \int_I \vec x \,.\,\vec\ek_1\,
|\varkappa_{\mathcal{S}}|^2\,|\vec x_\rho|\drho
- \left[\int_I \vec x\,.\,\vec\ek_1\,|\vec x_\rho| \drho\right]^{-1}
\left|\int_I \vec x\,.\,\vec\ek_1\,\varkappa_{\mathcal{S}}\,
|\vec x_\rho| \drho\right|^2 
\geq 0\,,
\end{equation} 
where we have noted the Cauchy--Schwarz inequality. It does not appear possible
to extend the stability result (\ref{eq:nlVbgnnewstab}) 
to the case of more general $f$.

\setcounter{equation}{0}
\section{Semidiscrete schemes} \label{sec:sd}

Let $[0,1]=\bigcup_{j=1}^J I_j$, $J\geq3$, be a
decomposition of $[0,1]$ into intervals given by the nodes $q_j$,
$I_j=[q_{j-1},q_j]$. 
For simplicity, and without loss of generality,
we assume that the subintervals form an equipartitioning of $[0,1]$,
i.e.\ that 
\begin{equation} \label{eq:Jequi}
q_j = j\,h\,,\quad \mbox{with}\quad h = J^{-1}\,,\qquad j=0,\ldots, J\,.
\end{equation}
Clearly, if $I=\RZ$ we identify $0=q_0 = q_J=1$.

The necessary finite element spaces are given by
$V^h  = \{\chi \in C(\overline I) : \chi\!\mid_{I_j} 
\mbox{ is linear}\ \forall\ j=1\to J\}$, 
$\Vh = [V^h]^2$, $\Vhpartialzero  = \Vh \cap \Vpartialzero$ and 
$\Vhpartial = \Vh \cap \Vpartial$.
Let $\{\chi_j\}_{j=j_0}^J$ denote the standard basis of $V^h$,
where $j_0 = 0$ if $I = (0,1)$ and $j_0 = 1$ if $I=\RZ$.
For later use, we let $\pi^h:C(\overline I)\to V^h$ 
be the standard interpolation operator at the nodes $\{q_j\}_{j=0}^J$.
Let $(\cdot,\cdot)$ denote the $L^2$--inner product on $I$, and 
define the mass lumped $L^2$--inner product $(f,g)^h$,
for two piecewise continuous functions, with possible jumps at the 
nodes $\{q_j\}_{j=1}^J$, via
\begin{equation}
( f, g )^h = \tfrac12\,h\,\sum_{j=1}^J 
\left[(f\,g)(q_j^-) + (f\,g)(q_{j-1}^+)\right],
\label{eq:ip0}
\end{equation}
where we define
$f(q_j^\pm)=\underset{\delta\searrow 0}{\lim}\ f(q_j\pm\delta)$.
The definition (\ref{eq:ip0}) naturally extends to vector valued functions.
It is easily shown that
\begin{equation} \label{eq:normequiv}
(\eta, \eta) \leq (\eta,\eta)^h \leq 3\,(\eta,\eta)\qquad
\forall\ \eta \in V^h\,.
\end{equation}

Let $(\vec X^h(t))_{t\in[0,T]}$, with $\vec X^h(t)\in \Vhpartialzero$,
be an approximation to $(\vec x(t))_{t\in[0,T]}$ and define
$\Gamma^h(t) = \vec X^h(t)(\overline I)$. Throughout this section we assume
that
\begin{equation} \label{eq:Xhpos}
\vec X^h(\rho,t) \,.\,\vec\ek_1 > 0 \quad 
\forall\ \rho \in \overline I\setminus \partial_0 I\,,
\qquad \forall\ t \in [0,T]\,.
\end{equation}
Assuming that $|\vec{X}^h_\rho| > 0$ almost everywhere on $I$,
and similarly to (\ref{eq:tau}), we set
\begin{equation*} 
\vec\tau^h = \vec X^h_s = \frac{\vec X^h_\rho}{|\vec X^h_\rho|} 
\qquad \mbox{and} \qquad \vec\nu^h = -(\vec\tau^h)^\perp\,.
\end{equation*}
We note that
\begin{equation} \label{eq:tauchieta}
(\vec\tau^h, (\vec\pi^h[\chi\,\vec\eta])_\rho) 
= (\vec\tau^h, (\chi\,\vec\eta)_\rho)
\qquad \forall\ \chi \in C(\overline I)\,,\ \vec\eta \in [C(\overline I)]^2\,.
\end{equation}
For later use, we let $\vec\omega^h \in \underline V^h$ be the mass-lumped 
$L^2$--projection of $\vec\nu^h$ onto $\underline V^h$, i.e.\
\begin{equation} \label{eq:omegah}
\left(\vec\omega^h, \vec\varphi \, |\vec X^h_\rho| \right)^h 
= \left( \vec\nu^h, \vec\varphi \, |\vec X^h_\rho| \right)
= \left( \vec\nu^h, \vec\varphi \, |\vec X^h_\rho| \right)^h
\qquad \forall\ \vec\varphi\in\underline V^h\,.
\end{equation}

Recall from (\ref{eq:A}) and (\ref{eq:E}) that
\begin{align} \label{eq:Eh}
E(\vec X^h(t)) & = 2\,\pi\left(\vec X^h(t)\,.\,\vec\ek_1 ,
|\vec X^h_\rho(t)|\right) \nonumber \\ & \quad
+ 2\,\pi \sum_{p\in \partial_1 I} 
\sliprho^{(p)}\,(\vec X^h(p,t)\,.\,\vec\ek_1)\,\vec X^h(p,t)\,.\,\vec\ek_2
+ \pi \sum_{p\in \partial_2 I} 
\sliprho^{(p)}\,(\vec X^h(p,t)\,.\,\vec\ek_1)^2\,.
\end{align}
We have, similarly to (\ref{eq:dEdt}), that
\begin{align}
\ddt\, E(\vec X^h(t)) & = 2\,\pi \left( \left[\vec X^h_t\,.\,\vec\ek_1
+ \vec X^h\,.\,\vec\ek_1 
\,\frac{(\vec X^h_t)_\rho\,.\,\vec X^h_\rho}
{|\vec X^h_\rho|^2} \right], |\vec X^h_\rho| \right)
\nonumber \\ & \qquad
+ 2\,\pi\,\sum_{p\in \partial_1 I} 
\sliprho^{(p)}\left[(\vec X^h_t(p,t)\,.\,\vec\ek_1)\,
\vec X^h(p,t)\,.\,\vec\ek_2 + 
(\vec X^h(p,t)\,.\,\vec\ek_1)\,
\vec X^h_t(p,t)\,.\,\vec\ek_2 \right]
\nonumber \\ & \qquad
+ 2\,\pi\,\sum_{p\in \partial_2 I} 
\sliprho^{(p)}\,(\vec X^h(p,t)\,.\,\vec\ek_1)\,
\vec X^h_t(p,t)\,.\,\vec\ek_1\,.
\label{eq:dEhdt}
\end{align}

\subsection{Mean curvature flow}

In view of the degeneracy on the right hand side of (\ref{eq:xt}), and on
recalling (\ref{eq:bclimit}) and (\ref{eq:omegah}), we introduce,
given a $\kappa^h(t) \in V^h$, the function 
$\doctorkappa^h(\kappa^h(t),t) \in V^h$ such that
\begin{equation} \label{eq:calKh}
[\doctorkappa^h (\kappa^h(t),t)](q_j) = \begin{cases}
\dfrac{\vec\omega^h(q_j,t)\,.\,\vec\ek_1}{\vec X^h(q_j,t)\,.\,\vec\ek_1}
& q_j \in \overline I \setminus \partial_0 I\,, \\
- \kappa^h(q_j,t) & q_j \in \partial_0 I\,.
\end{cases}
\end{equation}

Our semidiscrete finite element approximation of 
$(\BGNmckappa)$, (\ref{eq:xtweak},b), is given as follows.

$(\BGNmckappa_h)^h$:
Let $\vec X^h(0) \in \Vhpartialzero$. For $t \in (0,T]$
find $\vec X^h(t) \in \Vh$, with $\vec X^h_t(t) \in \Vhpartial$, and
$\kappa^h(t) \in V^h$ such that
\begin{subequations}
\begin{align}
\left(\vec X^h_t, \chi\,\vec\nu^h\,|\vec X^h_\rho|\right)^h
= \left(\kappa^h - \doctorkappa^h (\kappa^h),
\chi\,|\vec X^h_\rho| \right)^h 
\qquad \forall\ \chi \in V^h\,, \label{eq:sda}\\
\left(\kappa^h\,\vec\nu^h, \vec\eta\,|\vec X^h_\rho|\right)^h
+ \left(\vec X^h_\rho, \vec\eta_\rho\,|\vec X^h_\rho|^{-1}\right) 
= - \sum_{i=1}^2
\sum_{p \in \partial_i I} \sliprho^{(p)}\,\vec\eta(p)\,.\,\vec\ek_{3-i}
\qquad \forall\ \vec\eta \in \Vhpartial\,.
\label{eq:sdb}
\end{align}
\end{subequations}

\begin{rem} \label{rem:equid}
Let $\vec{h}_j(t) = \vec{X}^h(q_j,t) - \vec{X}^h(q_{j-1},t)$ 
for $j = 1,\ldots, J$, and set $\vec h_0 = \vec h_J$ if 
$\partial I = \emptyset$. Then, if 
$(\vec X^h(t), \kappa^h(t)) \in \Vh \times V^h$ satisfies
{\rm (\ref{eq:sdb})}, it holds that
\begin{equation}
|\vec{h}_j(t)| = |\vec{h}_{j - 1}(t)| \quad \mbox{if} \quad 
\vec{h}_j(t) \nparallel \vec{h}_{j - 1}(t) \quad 
\begin{cases} 
j = 1 ,\ldots, J & \partial I = \emptyset\,, \\
j = 2 ,\ldots, J & \partial I \not= \emptyset\,.
\end{cases}
\label{eq:equid}
\end{equation}
The equidistribution property {\rm (\ref{eq:equid})} can be shown by choosing
$\vec\eta = \chi_{j-1}\,[\vec\omega^h(q_{j-1},t)]^\perp \in \Vhpartial$ 
in {\rm (\ref{eq:sdb})}, recall {\rm (\ref{eq:omegah})}. 
See also \cite[Remark~2.4]{triplej} and \cite[Remark~2.5]{triplejMC}
for more details.

We also remark that it follows from {\rm (\ref{eq:omegah})} that
\begin{equation} \label{eq:omeganorm}
|\vec\omega^h(q_{j-1},t)| < 1 \quad \mbox{if} \quad 
\vec{h}_j(t) \nparallel \vec{h}_{j - 1}(t) \quad 
\begin{cases} 
j = 1 ,\ldots, J & \partial I = \emptyset\,, \\
j = 2 ,\ldots, J & \partial I \not= \emptyset\,.
\end{cases}
\end{equation}
\end{rem}

We note that mass lumping in (\ref{eq:sdb}) is crucial for the proof of the
equidistribution property (\ref{eq:equid}). Hence we only
consider the variant $(\BGNmckappa_h)^h$ with mass lumping.
Of course, in the case $\partial_0 I= \emptyset$, an alternative scheme to 
(\ref{eq:sda},b) is
\begin{equation} \label{eq:sd2a}
\left(\vec X^h_t, \chi\,\vec\nu^h\,|\vec X^h_\rho|\right)^h
= \left(\kappa^h - \frac{\vec\nu^h\,.\,\vec\ek_1}{\vec X^h\,.\,\vec\ek_1},
\chi\,|\vec X^h_\rho| \right)^h 
\qquad \forall\ \chi \in V^h\,,
\end{equation}
together with (\ref{eq:sdb}). Note that if $\partial_0 I= \emptyset$ 
then (\ref{eq:sda}) 
collapses to (\ref{eq:sd2a}) with $\vec\nu^h$ replaced by $\vec\omega^h$.
Unfortunately, neither choice appears to lead to a stability proof.

In an attempt to prove stability, we choose 
$\vec\eta = \vec\pi^h[(\vec X^h\,.\,\vec\ek_1)\,\vec X^h_t]$
in (\ref{eq:sdb}). Then it follows from 
(\ref{eq:dEhdt}), $\vec X^h_t \in \Vhpartial$, 
(\ref{eq:tauchieta}) and (\ref{eq:omegah}) that
\begin{align}
\ddt\,E(\vec X^h(t)) & 
= 2\,\pi \left( \left[\vec X^h_t\,.\,\vec\ek_1
+ \vec X^h\,.\,\vec\ek_1\,\frac{(\vec X^h_t)_\rho\,.\,\vec X^h_\rho}
{|\vec X^h_\rho|^2} \right], |\vec X^h_\rho| \right)
\nonumber \\ & \qquad
+ 2\,\pi \sum_{i=1}^2 \sum_{p\in \partial_i I} 
\sliprho^{(p)}\,(\vec X^h(p,t)\,.\,\vec\ek_1)\,
\vec X^h_t(p,t)\,.\,\vec\ek_{3-i}\,.
\nonumber \\ & 
= 2\,\pi\left(\vec X^h_t,\left[\vec\ek_1
- (\vec\ek_1\,.\,\vec\tau^h)\,\vec\tau^h
\right] |\vec X^h_\rho|\right)
- 2\,\pi\left(\vec X^h\,.\,\vec\ek_1\,\kappa^h\,\vec\nu^h,
\vec X^h_t\, |\vec X^h_\rho|\right)^h
\nonumber \\ &
= 2\,\pi\left(\vec X^h_t\,.\,\vec\nu^h,\vec\ek_1\,.\,\vec\nu^h\,
|\vec X^h_\rho|\right)
- 2\,\pi\left(\vec X^h\,.\,\vec\ek_1\,\kappa^h,
\vec X^h_t\,.\,\vec\nu^h\, |\vec X^h_\rho|\right)^h
\nonumber \\ &
= 2\,\pi\left(\vec X^h_t\,.\,\vec\nu^h,\vec\ek_1\,.\,\vec\nu^h\,|\vec
X^h_\rho|\right)^h
- 2\,\pi\left(\vec X^h\,.\,\vec\ek_1\,\kappa^h,
\vec X^h_t\,.\,\vec\nu^h\, |\vec X^h_\rho|\right)^h
\nonumber \\ &
= - 2\,\pi\left(\vec X^h\,.\,\vec\ek_1\left[\kappa^h
- \frac{\vec\nu^h\,.\,\vec\ek_1}{\vec X^h\,.\,\vec\ek_1}\right], 
\vec X^h_t\,.\,\vec\nu^h\, |\vec X^h_\rho|\right)^h
\nonumber \\ &
= - 2\,\pi\left(\vec X^h\,.\,\vec\ek_1, \kappa^h\,
\vec X^h_t\,.\,\vec\omega^h\, |\vec X^h_\rho|\right)^h
+ 2\,\pi\left( \vec\nu^h\,.\,\vec\ek_1,
\vec X^h_t\,.\,\vec\nu^h\, |\vec X^h_\rho|\right)^h.
\label{eq:stabh}
\end{align}
Moreover, considering for simplicity the case $\partial_0 I= \emptyset$,
and choosing \linebreak
$\chi = -\pi^h[2\,\pi\,(\vec X^h\,.\,\vec\ek_1)\,(\vec
X^h_t\,.\,\vec\omega^h)]$ in (\ref{eq:sd2a}) 
yields, on noting (\ref{eq:omegah}), that
\begin{align} \label{eq:Xhnorm}
0&\geq
- 2\,\pi
\left(\vec X^h\,.\,\vec\ek_1,(\vec X^h_t\,.\,\vec\omega^h)^2\,
|\vec X^h_\rho| \right)^h 
=-2\,\pi\left(\vec X^h\,.\,\vec\ek_1\left[\kappa^h
- \frac{\vec\omega^h\,.\,\vec\ek_1}{\vec X^h\,.\,\vec\ek_1}\right], 
\vec X^h_t\,.\,\vec\omega^h\, |\vec X^h_\rho|\right)^h \nonumber \\ &
= -2\,\pi\left(\vec X^h\,.\,\vec\ek_1, \kappa^h\,
\vec X^h_t\,.\,\vec\omega^h\, |\vec X^h_\rho|\right)^h
+ 2\,\pi\left(\vec\omega^h\,.\,\vec\ek_1, 
\vec X^h_t\,.\,\vec\omega^h\, |\vec X^h_\rho|\right)^h \nonumber \\ &
= -2\,\pi\left(\vec X^h\,.\,\vec\ek_1, \kappa^h\,
\vec X^h_t\,.\,\vec\omega^h\, |\vec X^h_\rho|\right)^h
+ 2\,\pi\left(\vec\nu^h\,.\,\vec\ek_1, 
\vec X^h_t\,.\,\vec\omega^h\, |\vec X^h_\rho|\right)^h .
\end{align}
Unfortunately, the right hand sides in (\ref{eq:stabh}) and
(\ref{eq:Xhnorm}) are not equal, recall (\ref{eq:omegah}),
and so combining (\ref{eq:stabh}) and
(\ref{eq:Xhnorm}) does not yield a stability result. 
On the other hand, the function 
$(\vec X^h\,.\,\vec\ek_1)\,(\vec X^h_t\,.\,\vec\nu^h)$ is discontinuous,
and so
$\pi^h[(\vec X^h\,.\,\vec\ek_1)\,(\vec X^h_t\,.\,\vec\nu^h)]$ 
is not well-defined, and cannot be chosen as a test function in (\ref{eq:sda})
or (\ref{eq:sd2a}).

However, the fully discrete variant of $(\BGNmckappa_h)^h$, (\ref{eq:sda},b), 
performs very well in practice.

A semidiscrete approximation of $(\GDmckappa)$,
(\ref{eq:Dziuka},b), is given as follows. 

$(\GDmckappa_h)^h$:
Let $\vec X^h(0) \in \Vhpartialzero$. For $t \in (0,T]$
find $\vec X^h(t) \in \Vh$, with $\vec X^h_t(t) \in \Vhpartial$, and
$\vec\kappa^h(t) \in \Vh$, such that
\begin{subequations}
\begin{align}
\left(\vec X^h_t, \vec\chi\,|\vec X^h_\rho| \right)^h
= \left(\vec\kappa^h - \vec{\doctorkappa}^h(\vec\kappa^h),
\vec\chi\,|\vec X^h_\rho|\right)^h 
\qquad \forall\ \vec\chi \in \Vh\,, \label{eq:sddziuka}\\
\left(\vec\kappa^h, \vec\eta\,|\vec X^h_\rho|\right)^h
+ \left(\vec X^h_\rho, \vec\eta_\rho\,|\vec X^h_\rho|^{-1}\right)
= - \sum_{i=1}^2 
\sum_{p \in \partial_i I} \sliprho^{(p)}\,\vec\eta(p)\,.\,\vec\ek_{3-i}
\qquad \forall\ \vec\eta \in \Vhpartial\,,
\label{eq:sddziukb}
\end{align}
\end{subequations}
where $\vec{\doctorkappa}^{h}(\vec\kappa^h) \in \Vh$ 
is such that
\begin{equation} \label{eq:dziukcalKh}
[\vec{\doctorkappa}^{h}(\vec\kappa^h(t),t))](q_j) = \begin{cases}
\dfrac{\vec\omega^h(q_j,t)\,.\,\vec\ek_1}{\vec X^h(q_j,t)\,.\,\vec\ek_1}\,
\dfrac{\vec\omega^h(q_j,t)}{|\vec\omega^h(q_j,t)|^2} 
& q_j \in \overline I \setminus \partial_0 I\,, \\
- \vec\kappa^h(q_j,t) & q_j \in \partial_0 I\,.
\end{cases}
\end{equation}
The rescaling factor $|\vec\omega^h(q_j,t)|^2$ in {\rm (\ref{eq:dziukcalKh})}
normalizes the discrete vertex normals $\vec\omega^h(q_j,t)$, 
recall (\ref{eq:omeganorm}), which is the 
most natural approach. Similarly to $(\BGNmckappa_h)^h$, 
it does not appear possible to
prove a stability result for $(\GDmckappa_h)^h$.

However, it turns out that approximations of the formulations 
$(\BGNmc)$ and $(\GDmc)$ can be shown to be stable. In particular,
our semidiscrete approximations of 
$(\BGNmc)$, (\ref{eq:bgnnewa},b), and $(\GDmc)$, (\ref{eq:Dziuknewa},b), 
are given as follows, where we first define
\[
W^h = V^h\,,\quad
W^h_{\partial_0} = \{ \chi \in V^h : \chi(\rho) = 0
\quad \forall\ \rho \in \partial_0 I\}\,, \quad 
\underline W^h = \underline V^h\,,\quad
\underline W^h_{\partial_0} = [W^h_{\partial_0}]^2\,.
\]

$(\BGNmc_h)^{(h)}$:
Let $\vec X^h(0) \in \Vhpartialzero$. For $t \in (0,T]$
find $\vec X^h(t) \in \Vh$, with $\vec X^h_t(t) \in \Vhpartial$, and
$\kappa_{\mathcal{S}}^h(t) \in \Whpartialzero$ such that
\begin{subequations}
\begin{align}
& \left((\vec X^h\,.\,\vec\ek_1)\,
\vec X^h_t, \chi\,\vec\nu^h\,|\vec X^h_\rho|\right)^{(h)}
= \left(\vec X^h\,.\,\vec\ek_1\,\kappa_{\mathcal{S}}^h,
\chi\,|\vec X^h_\rho|\right)^{(h)} \qquad \forall\ \chi \in \Whpartialzero\,,
\label{eq:mcsdnewa} \\
&
\left(\vec X^h\,.\,\vec\ek_1\,\kappa_{\mathcal{S}}^h\,\vec\nu^h,
\vec\eta\,|\vec X^h_\rho|\right)^{(h)}
+ \left( \vec\eta \,.\,\vec\ek_1, |\vec X^h_\rho|\right)
+ \left( (\vec X^h\,.\,\vec\ek_1)\,
\vec X^h_\rho,\vec\eta_\rho\, |\vec X^h_\rho|^{-1} \right)
\nonumber \\ & \hspace{4cm}
= - \sum_{i=1}^2
\sum_{p \in \partial_i I} \sliprho^{(p)}\,
(\vec X^h(p,t)\,.\,\vec\ek_1)\,\vec\eta(p)\,.\,\vec\ek_{3-i}
\qquad \forall\ \vec\eta \in \Vhpartial\,.
\label{eq:mcsdnewb}
\end{align}
\end{subequations}
Here and throughout we use the notation $\cdot^{(h)}$ to denote an 
expression with or without the superscript $h$, and similarly
for the subscripts $\cdot_{(\partial_0)}$. I.e.\ the scheme
$(\BGNmc_h)^h$ employs mass lumping, recall (\ref{eq:ip0}), and
seeks $\kappa_S(t) \in W^h_{\partial_0}$, while 
the scheme $(\BGNmc_h)$ employs true integration throughout
and seeks $\kappa_S(t) \in W^h = V^h$.

$(\GDmc_h)^{(h)}$:
Let $\vec X^h(0) \in \Vhpartialzero$. For $t \in (0,T]$
find $\vec X^h(t) \in \Vh$, with $\vec X^h_t(t) \in \Vhpartial$, and
$\vec\kappa_{\mathcal{S}}^h(t) \in \vecWhpartialzero$ such that
\begin{subequations}
\begin{align}
& \left((\vec X^h\,.\,\vec\ek_1)\,
\vec X^h_t,\vec\chi\,|\vec X^h_\rho|\right)^{(h)}
= \left((\vec X^h\,.\,\vec\ek_1)\,\vec\kappa_{\mathcal{S}}^h,
\vec\chi\,|\vec X^h_\rho|\right)^{(h)} \qquad \forall\ \vec\chi \in 
\vecWhpartialzero\,,
\label{eq:dziuksdnewa} \\
&
\left((\vec X^h\,.\,\vec\ek_1)\,\vec\kappa_{\mathcal{S}}^h,
\vec\eta\,|\vec X^h_\rho|\right)^{(h)}
+ \left( \vec\eta \,.\,\vec\ek_1, |\vec X^h_\rho|\right)
+ \left( (\vec X^h\,.\,\vec\ek_1)\,
\vec X^h_\rho,\vec\eta_\rho\, |\vec X^h_\rho|^{-1} \right) 
\nonumber \\ & \hspace{4cm}
= - \sum_{i=1}^2 \sum_{p \in \partial_i I} \sliprho^{(p)}\,
(\vec X^h(p,t)\,.\,\vec\ek_1)\,\vec\eta(p)\,.\,\vec\ek_{3-i}
\quad \forall\ \vec\eta \in \Vhpartial\,.
\label{eq:dziuksdnewb}
\end{align}
\end{subequations}
We observe that $(\BGNmc_h)^{h}$ and 
$(\GDmc_h)^{h}$ do not
depend on the values of $\kappa_{\mathcal{S}}^h$ and 
$\vec\kappa_{\mathcal{S}}^h$,
respectively, on $\partial_0 I$. Hence we fix these values to be zero by 
requiring that $\kappa_{\mathcal{S}}^h \in W^h_{\partial_0}$ and
$\vec\kappa_{\mathcal{S}}^h \in \underline W^h_{\partial_0}$,
and by using a reduced set of test functions in (\ref{eq:mcsdnewa}) and
(\ref{eq:dziuksdnewa}). As a consequence, it seems at first that 
$\vec X^h_t$ is not defined on $\partial_0 I$. However, $\vec X^h$ on 
$\partial_0 I$ is determined through (\ref{eq:mcsdnewb}) and
(\ref{eq:dziuksdnewb}), respectively.

We have on choosing $\chi = \kappa_{\mathcal{S}}^h$ in (\ref{eq:mcsdnewa}),
$\vec\chi = \vec\kappa_{\mathcal{S}}^h$ in (\ref{eq:dziuksdnewa}) and
$\vec\eta = \vec X^h_t$ in (\ref{eq:mcsdnewb}) 
and (\ref{eq:dziuksdnewb}), on recalling (\ref{eq:dEhdt}), that
\begin{equation} \label{eq:dEhepsdt}
-\frac1{2\,\pi}\,
\ddt\, E(\vec X^h(t)) = 
\begin{cases}
\left(\vec X^h\,.\,\vec\ek_1\,
|\kappa_{\mathcal{S}}^h|^2, |\vec X^h_\rho|\right)^{(h)},\\
\left(\vec X^h\,.\,\vec\ek_1\,
|\vec\kappa_{\mathcal{S}}^h|^2, |\vec X^h_\rho|\right)^{(h)} ,\\
\end{cases}
\end{equation}
respectively. This shows that both methods are stable, where we recall
(\ref{eq:Xhpos}). 
Similarly to (\ref{eq:Dziuknewstab}) and (\ref{eq:bgnnewstab}), we observe that
(\ref{eq:dEhepsdt}) implies for $(\GDmc_h)^{(h)}$ and $(\BGNmc_h)^h$ that
\[
-\frac1{2\,\pi}\,
\ddt\, E(\vec X^h(t)) = 
\begin{cases}
\left(\vec X^h\,.\,\vec\ek_1\,|\vec X^h_t|^2, |\vec X^h_\rho|\right)^{(h)},\\
\left(\vec X^h\,.\,\vec\ek_1\,
(\vec X^h_t\,.\,\vec\omega^h)^2, |\vec X^h_\rho|\right)^h,
\end{cases}
\]
respectively, where we have recalled (\ref{eq:omegah}).
This shows that they can be interpreted as
natural $L^2$--gradient flows of (\ref{eq:Eh}).

We observe that it is possible to eliminate $\vec\kappa^h_{\mathcal{S}}$ from
the schemes $(\GDmc_h)^{(h)}$, which yields (\ref{eq:dziuksdnewb}) with
$\vec\kappa^h_{\mathcal{S}}$ replaced by $\vec X^h_t$. Similarly, 
$\kappa^h_{\mathcal{S}}$ can be removed from the scheme 
$(\BGNmc_h)^{h}$ to yield (\ref{eq:mcsdnewb}) with 
$\kappa^h_{\mathcal{S}}\,\vec\nu^h$ replaced by
$(\vec X^h_t\,.\,\vec\omega^h)\,\vec\omega^h$, on recalling
(\ref{eq:omegah}). For the scheme
$(\BGNmc_h)$ this elimination procedure is not possible.

For the reader's convenience, Table~\ref{tab:schemes} summarises the
main properties of all the schemes introduced in this section.
\begin{table}
\center
\begin{tabular}{l|c|c|c}
scheme & stability proof & implicit tangential motion & equidistribution \\ \hline
$(\BGNmckappa_h)^h$ & no & yes & yes \\
$(\GDmckappa_h)^{h}$ & no & no & no \\
$(\BGNmc_h)^{(h)}$ & yes & yes & no \\
$(\GDmc_h)^{(h)}$ & yes & no & no \\
\end{tabular}
\caption{Properties of the different semidiscrete schemes for mean curvature
flow.}
\label{tab:schemes}
\end{table}%

\subsection{Nonlinear mean curvature flow}
Replacing $\kappa^h - \doctorkappa^h(\kappa^h)$ with
$f(\kappa^h - \doctorkappa^h(\kappa^h))$ in (\ref{eq:sda}) yields the scheme
$(\BGNmckappa_h^f)^h$. 
Similarly, the scheme $(\BGNmckappa_h^{f,V})^h$ is given by
(\ref{eq:sda},b) with the right hand side in (\ref{eq:sda}) 
replaced by
\begin{equation} \label{eq:nlVsda}
\left( f(\kappa^h - \doctorkappa^h(\kappa^h)),
\chi\,|\vec X^h_\rho|\right)^h 
- \frac{\left(\vec X^h\,.\,\vec\ek_1, f(\kappa^h - \doctorkappa^h(\kappa^h))\,
|\vec X^h_\rho|\right)^{h}}
{\left(\vec X^h\,.\,\vec\ek_1, |\vec X^h_\rho|\right)}
\left(\chi,|\vec X^h_\rho|\right)^{h} .
\end{equation}
These two schemes inherit the equidistribution property, 
recall (\ref{eq:equid}).
Replacing $\kappa^h_{\mathcal{S}}$ with
$\pi^h[f(\kappa^h_{\mathcal{S}})]$ in (\ref{eq:mcsdnewa}) yields the schemes
$(\BGNmc_h^f)^{(h)}$ and similarly we can define 
$(\BGNmc_h^{f,V})^{(h)}$ by replacing the right hand side in 
(\ref{eq:mcsdnewa}) by
\begin{equation} \label{eq:nlVmcsdnewa}
\left( \vec X^h\,.\,\vec\ek_1\,\pi^h[f(\kappa^h_{\mathcal{S}})],
\chi\,|\vec X^h_\rho|\right)^{(h)} 
- \frac{\left(\vec X^h\,.\,\vec\ek_1, \pi^h[f(\kappa^h_{\mathcal{S}})]\,
|\vec X^h_\rho|\right)^{(h)}}
{\left(\vec X^h\,.\,\vec\ek_1, |\vec X^h_\rho|\right)}
\left(\vec X^h\,.\,\vec\ek_1,\chi\,|\vec X^h_\rho|\right)^{(h)} .
\end{equation}
Similarly to (\ref{eq:dEhepsdt}),
and using the same choices of $\vec\eta$ and $\chi$, 
it can be shown that solutions to 
the scheme $(\BGNmc_h^f)^{(h)}$ satisfy
$-\frac1{2\,\pi}\,\ddt\, E(\vec X^h(t)) = 
\left((\vec X^h\,.\,\vec\ek_1)\,f(\kappa_{\mathcal{S}}^h),
\kappa_{\mathcal{S}}^h\, |\vec X^h_\rho|\right)^{(h)}$,
which yields a stability bound for $(\BGNmc_h^f)^{h}$
if $f$ is monotonically increasing with $f(0) = 0$. Of course,
(\ref{eq:nlVmcsdnewa}) is a discrete analogue of (\ref{eq:nlbgnnewstab}). 
Moreover, solutions to $(\BGNmc_h^{f,V})^{(h)}$, in the case (\ref{eq:fmcf}),
satisfy
\begin{align*} 
& -\frac1{2\,\pi}\,\ddt\, E(\vec X^h(t)) 
\nonumber \\ & \qquad
= \left(\vec X^h \,.\,\vec\ek_1\,
|\kappa_{\mathcal{S}}^h|^2,|\vec X^h_\rho| \right)^{(h)}
- \left[ \left( \vec X^h\,.\,\vec\ek_1,|\vec X^h_\rho| 
\right)\right]^{-1}
\left|\left( \vec X^h\,.\,\vec\ek_1,\kappa_{\mathcal{S}}^h\,
|\vec X^h_\rho| \right)^{(h)}\right|^2 \geq 0\,,
\end{align*} 
similarly to (\ref{eq:nlVbgnnewstab}), where here we have also used a
Cauchy--Schwarz inequality for the mass lumped inner product (\ref{eq:ip0}). 
Finally, solutions to the scheme $(\BGNmc_h^{f,V})$ conserve the 
volume of the domain $\Omega^h(t) \subset \bR^3$ that is enclosed by 
the three-dimensional axisymmetric surface 
$\mathcal{S}^h(t)$ that is generated by the curve $\Gamma^h(t)$.
To see this, choose $\chi = 2\,\pi$ in  (\ref{eq:mcsdnewa}),
with the modified right hand side (\ref{eq:nlVmcsdnewa}), to obtain 
\begin{equation} \label{eq:constchi}
0 = 2\,\pi \left(\vec X^h\,.\,\vec\ek_1,
\vec X^h_t\,.\, \vec\nu^h\,|\vec X^h_\rho|\right)
= \int_{\mathcal{S}^h(t)} \mathcal{V}^h_{\mathcal{S}^h} \dH{2}
= \ddt\,\mathcal{L}^3(\Omega^h(t))\,,
\end{equation}
recall (\ref{eq:dVdt}). 
Here $\mathcal{V}^h_{\mathcal{S}^h}(t)$ denotes the normal velocity of
$\mathcal{S}^h(t)$ in the direction of $\vec\nu^h_{\mathcal{S}^h}(t)$,
the outer normal to $\Omega^h(t)$ on $\mathcal{S}^h(t)$, 
where $\vec\nu^h_{\mathcal{S}^h}(t)$ is induced by $\vec\nu^h$ through
a discrete analogue of (\ref{eq:nuS}). Using the same testing procedure for the
scheme $(\BGNmc_h^{f,V})^h$ yields that 
\begin{equation} \label{eq:constchih}
0 = 2\,\pi \left(\vec X^h\,.\,\vec\ek_1,
\vec X^h_t\,.\, \vec\nu^h\,|\vec X^h_\rho|\right)^h\,,
\end{equation}
and so the enclosed volume is only approximately preserved, compare with
(\ref{eq:constchi}). 
Finally, 
choosing $\chi = \vec X^h\,.\,\vec\ek_1$ in $(\BGNmckappa_h^{f,V})^h$,
recall (\ref{eq:nlVsda}), also
yields (\ref{eq:constchih}), and so an approximate volume preservation
property. 

\setcounter{equation}{0}
\section{Fully discrete schemes} \label{sec:fd}

Let $0= t_0 < t_1 < \ldots < t_{M-1} < t_M = T$ be a
partitioning of $[0,T]$ into possibly variable time steps 
$\ttau_m = t_{m+1} - t_{m}$, $m=0\to M-1$. 
We set $\ttau = \max_{m=0\to M-1}\ttau_m$.
For a given $\vec{X}^m\in \Vhpartialzero$,
assuming that $|\vec{X}^m_\rho| > 0$ almost everywhere on $I$,
we set $\vec\nu^m = - \frac{[\vec X^m_\rho]^\perp}{|\vec X^m_\rho|}$.
Let $\vec\omega^m \in \Vh$ be the natural fully discrete analogue of
$\vec\omega^h \in \Vh$, recall (\ref{eq:omegah}), i.e.\
\begin{equation} \label{eq:omegam}
\left(\vec\omega^m, \vec\varphi \, |\vec X^m_\rho| \right)^h 
= \left( \vec\nu^m, \vec\varphi \, |\vec X^m_\rho| \right)
= \left( \vec\nu^m, \vec\varphi \, |\vec X^m_\rho| \right)^h
\qquad \forall\ \vec\varphi\in\underline V^h\,.
\end{equation}

\subsection{Mean curvature flow}

Similarly to (\ref{eq:calKh}), and given a $\kappa^{m+1} \in V^h$,
we introduce $\doctorkappa^{m}(\kappa^{m+1}) \in V^h$ such that
\begin{equation} \label{eq:calKm}
[\doctorkappa^{m}(\kappa^{m+1})](q_j) = \begin{cases}
\dfrac{\vec\omega^m(q_j)\,.\,\vec\ek_1}{\vec X^m(q_j)\,.\,\vec\ek_1}
& q_j \in \overline I \setminus \partial_0 I\,, \\
- \kappa^{m+1}(q_j) & q_j \in \partial_0 I\,.
\end{cases}
\end{equation}

Then our fully discrete analogue of $(\BGNmckappa_h)^h$,
(\ref{eq:sda},b), is given as follows.

$(\BGNmckappa_m)^h$:
Let $\vec X^0 \in \Vhpartialzero$. For $m=0,\ldots,M-1$, 
find $(\delta\vec X^{m+1}, \kappa^{m+1}) \in \Vhpartial \times V^h$,
where $\vec X^{m+1} = \vec X^m + \delta \vec X^{m+1}$, 
such that
\begin{subequations}
\begin{align}
& 
\left(\frac{\vec X^{m+1} - \vec X^m}{\ttau_m}, \chi\,\vec\nu^m\,|\vec
X^m_\rho|\right)^h
= \left(\kappa^{m+1} - \doctorkappa^{m}(\kappa^{m+1}),
\chi\,|\vec X^m_\rho|\right)^h 
\quad \forall\ \chi \in V^h\,, \label{eq:fda}\\
& \left(\kappa^{m+1}\,\vec\nu^m, \vec\eta\,|\vec X^m_\rho|\right)^h
+ \left(\vec X^{m+1}_\rho, \vec\eta_\rho\,|\vec X^m_\rho|^{-1}\right) 
= - \sum_{i=1}^2 
\sum_{p \in \partial_i I} \sliprho^{(p)}\,\vec\eta(p)\,.\,\vec\ek_{3-i}
\quad \forall\ \vec\eta \in \Vhpartial\,.
\label{eq:fdb}
\end{align}
\end{subequations}

We make the following mild assumptions.
\begin{tabbing}
$(\mathfrak A)$ \quad \= Let
$|\vec{X}^m_\rho| > 0$ for almost all $\rho\in I$, and let
$\vec{X}^m \,.\,\vec\ek_1 > 0$ for all $\rho\in \overline I \setminus
\partial_0 I$.\\
$(\mathfrak B)^h$\quad \>
Let $\mathcal Z^{h} = 
\left\{ \left( \vec\nu^m , \chi |\vec X^m_\rho| \right)^{h} 
: \chi \in V^h \right \} \subset \bR^2$ and assume that
$\dim \spa \mathcal Z^{h} = 2$.
\end{tabbing}
Note that the assumption $(\mathfrak B)^h$, on recalling
(\ref{eq:omegah}), is equivalent to assuming that \linebreak
$\dim \spa\{\vec{\omega}^m(q_j)\}_{j=0}^{J}= 2$.

\begin{lem} \label{lem:ex}
Let $\vec X^m \in \Vhpartialzero$ satisfy the assumptions $(\mathfrak A)$
and $(\mathfrak B)^h$.
Then there exists a unique solution 
$(\delta\vec X^{m+1}, \kappa^{m+1}) \in \Vhpartial \times V^h$ to 
$(\BGNmckappa_m)^h$.
\end{lem}
\begin{proof}
We note that since $\vec X^m \in \Vhpartialzero$ satisfies the assumption 
$(\mathfrak A)$, the right hand side of (\ref{eq:fda}) is well-defined.
As (\ref{eq:fda},b) is linear, existence follows from uniqueness. 
To investigate the latter, we consider the system: 
Find $(\delta\vec X,\kappa) \in \Vhpartial \times V^h$ such that
\begin{subequations}
\begin{align}
\left(\frac{\delta\vec X}{\ttau_m}, \chi\,\vec\nu^m\,|\vec
X^m_\rho|\right)^h
= \left(\lambda\,\kappa ,
\chi\,|\vec X^m_\rho|\right)^h 
\qquad \forall\ \chi \in V^h\,, \label{eq:proofa}\\
\left(\kappa\,\vec\nu^m, \vec\eta\,|\vec X^m_\rho|\right)^h
+ \left((\delta\vec X)_\rho, \vec\eta_\rho\,|\vec X^m_\rho|^{-1}\right) = 0 
\qquad \forall\ \vec\eta \in \Vhpartial\,,
\label{eq:proofb}
\end{align}
\end{subequations}
where we recall from (\ref{eq:calKm}) that $\lambda \in V^h$ with
\begin{equation} \label{eq:betam}
\lambda(q_j) = \begin{cases} 
1 & q_j \in \overline I \setminus \partial_0 I\,, \\
2 & q_j \in \partial_0 I\,.
\end{cases}
\end{equation}
Choosing $\chi=\kappa\in V^h$ in (\ref{eq:proofa}) and 
$\vec\eta= \delta\vec X \in \Vhpartial$ in (\ref{eq:proofb}) yields that
\begin{equation} \label{eq:unique0}
\left(|(\delta\vec X)_\rho|^2, |\vec X^m_\rho|^{-1}\right)
+\ttau_m \left(\lambda\,|\kappa|^2 , |\vec X^m_\rho|\right)^h = 0\,.
\end{equation}
It follows from (\ref{eq:unique0}) that $\kappa = 0$ and that
$\delta\vec X \equiv \vec X^c\in\bR^2$; and hence that
\begin{equation}
0 = \left(\vec X^c, \chi\,\vec\nu^m\,|\vec X^m_\rho|\right)^h = 
\vec X^c \,. \left(\vec\nu^m, \chi\,|\vec X^m_\rho|\right)^h 
\quad \forall\ \chi \in V^h \,. \label{eq:unique1}
\end{equation}
It follows from (\ref{eq:unique1}) and
assumption $(\mathfrak B)^h$ that $\vec X^c=\vec0$.
Hence we have shown that (\ref{eq:fda},b) has a unique solution
$(\delta\vec X^{m+1},\kappa^{m+1}) \in \Vhpartial\times V^h$.
\end{proof}

We remark that a fully discrete approximation of $(\GDmckappa_h)^h$, 
(\ref{eq:sddziuka},b), is given by:

$(\GDmckappa_m)^h$:
Let $\vec X^0 \in \Vhpartialzero$. For $m=0,\ldots,M-1$, 
find $(\delta\vec X^{m+1}, \vec\kappa^{m+1}) \in \Vhpartial \times \Vh$,
where $\vec X^{m+1} = \vec X^m + \delta \vec X^{m+1}$, 
such that
\begin{subequations}
\begin{align}
\left(\frac{\vec X^{m+1} - \vec X^m}{\ttau_m}, \vec\chi\,|\vec X^m_\rho|
\right)^h
= \left(\vec\kappa^{m+1} - \vec{\doctorkappa}^{m}(\vec\kappa^{m+1}),
\vec\chi\,|\vec X^m_\rho|\right)^h 
\qquad \forall\ \vec\chi \in \Vh\,, \label{eq:fddziuka}\\
\left(\vec\kappa^{m+1}, \vec\eta\,|\vec X^m_\rho|\right)^h
+ \left(\vec X^{m+1}_\rho, \vec\eta_\rho\,|\vec X^m_\rho|^{-1}\right)
= - \sum_{i=1}^2
\sum_{p \in \partial_i I} \sliprho^{(p)}\,\vec\eta(p)\,.\,\vec\ek_{3-i}
\qquad \forall\ \vec\eta \in \Vhpartial\,,
\label{eq:fddziukb}
\end{align}
\end{subequations}
where $\vec{\doctorkappa}^{m}(\vec\kappa^{m+1}) \in \Vh$ 
is such that
\begin{equation*} 
[\vec{\doctorkappa}^{m}(\vec\kappa^{m+1})](q_j) = \begin{cases}
\dfrac{\vec\omega^m(q_j)\,.\,\vec\ek_1}{\vec X^m(q_j)\,.\,\vec\ek_1}\,
\dfrac{\vec\omega^m(q_j)}{|\vec\omega^m(q_j)|^2} 
& q_j \in \overline I \setminus \partial_0 I\,, \\
- \vec\kappa^{m+1}(q_j) & q_j \in \partial_0 I\,.
\end{cases}
\end{equation*}
In practice the scheme {\rm (\ref{eq:fddziuka},b)}, 
for reasonable time step sizes, can lead to oscillations and poor results,
see e.g.\ Figure~\ref{fig:torusR1r05} below.

\begin{lem} \label{lem:GDex}
Let $\vec X^m \in \Vhpartialzero$ satisfy the assumption $(\mathfrak A)$.
There exists a unique solution
$(\delta\vec X^{m+1}, \vec\kappa^{m+1}) \in \Vhpartial \times \Vh$ to 
$(\GDmckappa_m)^h$.
\end{lem}
\begin{proof}
Similarly to the proof of Lemma~\ref{lem:ex}, we obtain that
\begin{equation} \label{eq:uniqueGD0}
\left(|(\delta\vec X)_\rho|^2, |\vec X^m_\rho|^{-1}\right)
+\ttau_m \left(\lambda\,|\vec\kappa|^2 , |\vec X^m_\rho|\right)^h = 0\,,
\end{equation}
where $(\delta\vec X,\vec\kappa) \in \Vhpartial \times \Vh$ solve
the linear homogeneous system corresponding to (\ref{eq:fddziuka},b).
It follows from (\ref{eq:uniqueGD0}) that $\vec\kappa = \vec 0$ and then
from the homogeneous variant of (\ref{eq:fddziuka}) 
that $\delta\vec X = \vec 0$.
Hence we have shown that (\ref{eq:fddziuka},b) has a unique solution
$(\delta\vec X^{m+1},\vec\kappa^{m+1}) \in \Vhpartial\times\Vh$.
\end{proof}

Our fully discrete analogues of the schemes $(\BGNmc_h)^{(h)}$,
(\ref{eq:mcsdnewa},b), and $(\GDmc_h)^{(h)}$, (\ref{eq:dziuksdnewa},b), 
are given as follows.

$(\BGNmc_m)^{(h)}$:
Let $\vec X^0 \in \Vhpartialzero$. For $m=0,\ldots,M-1$, 
find $(\delta\vec X^{m+1}, \kappa_{\mathcal{S}}^{m+1}) \in \Vhpartial \times
\Whpartialzero$, where $\vec X^{m+1} = \vec X^m + \delta \vec X^{m+1}$, 
such that
\begin{subequations}
\begin{align}
& \left(\vec X^m\,.\,\vec\ek_1\,
\frac{\vec X^{m+1} - \vec X^m}{\ttau_m}, \chi\,\vec\nu^m\,
|\vec X^m_\rho|\right)^{(h)}
= \left(\vec X^m\,.\,\vec\ek_1\,\kappa_{\mathcal{S}}^{m+1},
\chi\,|\vec X^m_\rho|\right)^{(h)} \quad \forall\ \chi \in \Whpartialzero\,,
\label{eq:fdnewa} \\
&
\left(\vec X^m\,.\,\vec\ek_1\,\kappa_{\mathcal{S}}^{m+1}\,\vec\nu^m,
\vec\eta\,|\vec X^m_\rho|\right)^{(h)}
+ \left( \vec\eta \,.\,\vec\ek_1, |\vec X^m_\rho|\right)
+ \left( (\vec X^m\,.\,\vec\ek_1)\,
\vec X^{m+1}_\rho,\vec\eta_\rho\, |\vec X^m_\rho|^{-1} \right)
\nonumber \\ & \qquad
= - \sum_{i=1}^2 \sum_{p \in \partial_i I} \sliprho^{(p)}\,
(\vec X^m(p)\,.\,\vec\ek_1)\,\vec\eta(p)\,.\,\vec\ek_{3-i}
\quad \forall\ \vec\eta \in \Vhpartial\,.
\label{eq:fdnewb}
\end{align}
\end{subequations}

For the second variant, which is going to lead to systems of nonlinear
equations and for which a stability result can be shown, we introduce
the notation $[ r ]_\pm = \pm \max \{ \pm r, 0 \}$ for $r \in \bR$.

$(\BGNmc_{m,\star})^{(h)}$:
Let $\vec X^0 \in \Vhpartialzero$. For $m=0,\ldots,M-1$, 
find $(\delta\vec X^{m+1}, \kappa_{\mathcal{S}}^{m+1}) \in \Vhpartial \times
\Whpartialzero$, where $\vec X^{m+1} = \vec X^m + \delta \vec X^{m+1}$, 
such that
\begin{subequations}
\begin{align}
& \left(\vec X^m\,.\,\vec\ek_1\,
\frac{\vec X^{m+1} - \vec X^m}{\ttau_m}, \chi\,\vec\nu^m\,
|\vec X^m_\rho|\right)^{(h)}
= \left(\vec X^m\,.\,\vec\ek_1\,\kappa_{\mathcal{S}}^{m+1},
\chi\,|\vec X^m_\rho|\right)^{(h)} \quad \forall\ \chi \in \Whpartialzero\,,
\label{eq:fdnonlineara} \\
&
\left(\vec X^m\,.\,\vec\ek_1\,\kappa_{\mathcal{S}}^{m+1}\,\vec\nu^m,
\vec\eta\,|\vec X^m_\rho|\right)^{(h)}
+ \left( \vec\eta \,.\,\vec\ek_1, |\vec X^{m+1}_\rho|\right)
+ \left( (\vec X^m\,.\,\vec\ek_1)\,
\vec X^{m+1}_\rho,\vec\eta_\rho\, |\vec X^m_\rho|^{-1} \right)
\nonumber \\ & \qquad
= - \sum_{p \in \partial_1 I} \sliprho^{(p)}\,
(\vec X^m(p)\,.\,\vec\ek_1)\,\vec\eta(p)\,.\,\vec\ek_2
\nonumber \\ & \qquad \quad
 - \sum_{p \in \partial_2 I} ( ( [\sliprho^{(p)}]_+ \,
\,\vec X^{m+1}(p)
+ [\sliprho^{(p)}]_- \,
\,\vec X^{m}(p))
\,.\,\vec\ek_1)\,\vec\eta(p)\,.\,\vec\ek_1
\quad \forall\ \vec\eta \in \Vhpartial\,.
\label{eq:fdnonlinearb}
\end{align}
\end{subequations}

$(\GDmc_m)^{(h)}$:
Let $\vec X^0 \in \Vhpartialzero$. For $m=0,\ldots,M-1$, 
find $(\delta\vec X^{m+1},\vec\kappa_{\mathcal{S}}^{m+1})\in \Vhpartial \times
\vecWhpartialzero$, where $\vec X^{m+1} = \vec X^m + \delta \vec X^{m+1}$, 
such that
\begin{subequations}
\begin{align}
& \left(\vec X^m\,.\,\vec\ek_1\,
\frac{\vec X^{m+1} - \vec X^m}{\ttau_m}, \vec\chi\,|\vec X^m_\rho|\right)^{(h)}
= \left((\vec X^m\,.\,\vec\ek_1)\,\vec\kappa_{\mathcal{S}}^{m+1},
\vec\chi\,|\vec X^m_\rho|\right)^{(h)} \quad \forall\ \vec\chi \in
\vecWhpartialzero\,,
\label{eq:fddziuknewa} \\
&
\left((\vec X^m\,.\,\vec\ek_1)\,\vec\kappa_{\mathcal{S}}^{m+1},
\vec\eta\,|\vec X^m_\rho|\right)^{(h)}
+ \left( \vec\eta \,.\,\vec\ek_1, |\vec X^m_\rho|\right)
+ \left( (\vec X^m\,.\,\vec\ek_1)\,
\vec X^{m+1}_\rho,\vec\eta_\rho\, |\vec X^m_\rho|^{-1} \right)
\nonumber \\ & \hspace{4cm}
= - \sum_{i=1}^2 \sum_{p \in \partial_i I} \sliprho^{(p)}\,
(\vec X^m(p)\,.\,\vec\ek_1)\,\vec\eta(p)\,.\,\vec\ek_{3-i}
\quad \forall\ \vec\eta \in \Vhpartial\,.
\label{eq:fddziuknewb}
\end{align}
\end{subequations}

$(\GDmc_{m,\star})^{(h)}$:
Let $\vec X^0 \in \Vhpartialzero$. For $m=0,\ldots,M-1$, 
find $(\delta\vec X^{m+1},\vec\kappa_{\mathcal{S}}^{m+1})\in \Vhpartial \times
\vecWhpartialzero$, where $\vec X^{m+1} = \vec X^m + \delta \vec X^{m+1}$, 
such that
\begin{subequations}
\begin{align}
& \left(\vec X^m\,.\,\vec\ek_1\,
\frac{\vec X^{m+1} - \vec X^m}{\ttau_m}, \vec\chi\,|\vec X^m_\rho|\right)^{(h)}
= \left((\vec X^m\,.\,\vec\ek_1)\,\vec\kappa_{\mathcal{S}}^{m+1},
\vec\chi\,|\vec X^m_\rho|\right)^{(h)} \quad \forall\ \vec\chi \in
\vecWhpartialzero\,,
\label{eq:dziukfdnonlineara} \\
&
\left((\vec X^m\,.\,\vec\ek_1)\,\vec\kappa_{\mathcal{S}}^{m+1},
\vec\eta\,|\vec X^m_\rho|\right)^{(h)}
+ \left( \vec\eta \,.\,\vec\ek_1, |\vec X^{m+1}_\rho|\right)
+ \left( (\vec X^m\,.\,\vec\ek_1)\,
\vec X^{m+1}_\rho,\vec\eta_\rho\, |\vec X^m_\rho|^{-1} \right)
\nonumber \\ & \qquad
= - \sum_{p \in \partial_1 I} \sliprho^{(p)}\,
(\vec X^m(p)\,.\,\vec\ek_1)\,\vec\eta(p)\,.\,\vec\ek_2
\nonumber \\ & \qquad \quad
 - \sum_{p \in \partial_2 I} ( ( [\sliprho^{(p)}]_+ \,
\,\vec X^{m+1}(p)
+ [\sliprho^{(p)}]_- \,
\,\vec X^{m}(p))
\,.\,\vec\ek_1)\,\vec\eta(p)\,.\,\vec\ek_1
\quad \forall\ \vec\eta \in \Vhpartial\,.
\label{eq:dziukfdnonlinearb}
\end{align}
\end{subequations}
Here we observe that $(\BGNmc_m)^{(h)}$ and $(\GDmc_m)^{(h)}$ are
linear schemes, while $(\BGNmc_{m,\star})^{(h)}$ and
$(\GDmc_{m,\star})^{(h)}$ are nonlinear. For the linear schemes we can
prove existence and uniqueness, while for the nonlinear schemes we can prove
unconditional stability. For the scheme $(\GDmc_{m,\star})^{(h)}$ we can also
prove existence if $\partial_0 I = \emptyset$ and if 
$\ttau_m$ is sufficiently small. It does not appear possible to extend the
techniques of the existence proof for $(\GDmc_{m,\star})^{(h)}$ to the scheme
$(\BGNmc_{m,\star})^{(h)}$.
We note that in practice we solve the nonlinear
schemes with a Newton iteration, which in all our experiments 
always converged with at most three iterations.

\begin{rem} \label{rem:oneliner}
Similarly to the semidiscrete variants,
we observe that in most of the above fully discrete schemes it is possible to
eliminate the discrete curvatures, $\kappa_{\mathcal{S}}^{m+1}$ or
$\vec\kappa_{\mathcal{S}}^{m+1}$. 
For example, on recalling {\rm (\ref{eq:omegah})} and
on choosing $\chi = \pi^h[\vec\eta\,.\,\vec\omega^m] \in \Whpartialzero$ in
{\rm (\ref{eq:fdnonlineara})} for $\vec\eta \in \Vhpartial$, the scheme
$(\BGNmc_{m,\star})^h$ reduces to: 
Find $\delta\vec X^{m+1}\in \Vhpartial$ such that, for all
$\vec\eta \in \Vhpartial$,
\begin{align}
& \left(\vec X^m\,.\,\vec\ek_1\,
\frac{\vec X^{m+1} - \vec X^m}{\ttau_m}\,.\,\vec\omega^m, \vec\eta\,.\,
\vec\omega^m\,|\vec X^m_\rho|\right)^h
+ \left( \vec\eta \,.\,\vec\ek_1, |\vec X^{m+1}_\rho|\right)
\nonumber \\ & \quad
+ \left( (\vec X^m\,.\,\vec\ek_1)\,
\vec X^{m+1}_\rho,\vec\eta_\rho\, |\vec X^m_\rho|^{-1} \right)
= - \sum_{i=1}^2 
\sum_{p \in \partial_i I} \sliprho^{(p)}\,
(\vec X^{m+1}(p)\,.\,\vec\ek_1)\,\vec\eta(p)\,.\,\vec\ek_{3-i}\,.
\label{eq:fdnonlinear}
\end{align}
and similarly for $(\BGNmc_m)^h$, $(\GDmc_m)^{(h)}$ and
$(\GDmc_{m,\star})^{(h)}$, with the latter leading to
(\ref{eq:dziukfdnonlinearb}) with $\vec\kappa^{m+1}_{\mathcal{S}}$ replaced by 
$(\ttau_m)^{-1}\,(\vec X^{m+1} - \vec X^m)$.
For the schemes $(\BGNmc_m)$ and $(\BGNmc_{m,\star})$
this elimination procedure is not possible.
A related variant to {\rm (\ref{eq:fdnonlinear})} is given by:
Find $\delta\vec X^{m+1}\in \Vhpartial$ such that, for all
$\vec\eta \in \Vhpartial$,
\begin{align}
& \left(\vec X^m\,.\,\vec\ek_1\,
\frac{\vec X^{m+1} - \vec X^m}{\ttau_m}\,.\,\vec\nu^m, \vec\eta\,.\,
\vec\nu^m\,|\vec X^m_\rho|\right)^h
+ \left( \vec\eta \,.\,\vec\ek_1, |\vec X^{m+1}_\rho|\right)
\nonumber \\ & \quad
+ \left( (\vec X^m\,.\,\vec\ek_1)\,
\vec X^{m+1}_\rho,\vec\eta_\rho\, |\vec X^m_\rho|^{-1} \right)
= - \sum_{i=1}^2
\sum_{p \in \partial_i I} \sliprho^{(p)}\,
(\vec X^{m+1}(p)\,.\,\vec\ek_1)\,\vec\eta(p)\,.\,\vec\ek_{3-i}\,.
\label{eq:fdnonlinearnu}
\end{align}
\end{rem}

We make the following mild assumption. 
\begin{tabbing}
$(\mathfrak C_{(\partial_0)})^{(h)}$ \quad \=
Let $\mathcal Z^{(h)}_{(\partial_0)} = 
\left\{ \left( (\vec X^m\,.\,\vec\ek_1)\,\vec\nu^m,\chi\, 
|\vec X^m_\rho| \right)^{(h)} : \chi \in \Whpartialzero \right
\} \subset \bR^2$ and assume that \\ \>
$\dim \spa \mathcal Z^{(h)}_{(\partial_0)} = 2$. 
\end{tabbing}
Note that the assumption $(\mathfrak C_{\partial_0})^{h}$, on recalling
(\ref{eq:omegah}), is equivalent to
assuming that \linebreak 
$\dim \spa\{\vec{\omega}^m(q_j)\}_{
j \in \{ k \in \{0,\ldots,J\} : q_k \in \overline I \setminus \partial_0 I\}}
= 2$,
and so it is slightly stronger than the assumption $(\mathfrak B)^h$.

\begin{lem} \label{lem:exnew}
Let $\vec X^m \in \Vhpartialzero$ satisfy the assumptions $(\mathfrak A)$ and
$(\mathfrak C_{(\partial_0)})^{(h)}$.
Then there exists a unique solution
$(\delta\vec X^{m+1},\kappa_{\mathcal{S}}^{m+1}) \in \Vhpartial \times 
\Whpartialzero$ to $(\BGNmc_m)^{(h)}$.
\end{lem}
\begin{proof}
As (\ref{eq:fdnewa},b) is linear, existence follows from uniqueness. 
To investigate the latter, we consider the system: 
Find $(\delta\vec X, \kappa_{\mathcal{S}}) \in \Vhpartial\times \Whpartialzero$ 
such that
\begin{subequations}
\begin{align}
&
\left(\vec X^m\,.\,\vec\ek_1\,\frac{\delta\vec X}{\ttau_m}, 
\chi\,\vec\nu^m\,|\vec X^m_\rho|\right)^{(h)}
= \left(\vec X^m\,.\,\vec\ek_1\,\kappa_{\mathcal{S}},
\chi\,|\vec X^m_\rho|\right)^{(h)} \qquad \forall\ \chi \in \Whpartialzero\,,
\label{eq:proofnewa} \\
& \left(\vec X^m\,.\,\vec\ek_1\,\kappa_{\mathcal{S}}\,\vec\nu^m,
\vec\eta\,|\vec X^m_\rho|\right)^{(h)}
+ \left((\vec X^m\,.\,\vec\ek_1)\,
(\delta\vec X)_\rho, \vec\eta_\rho\,|\vec X^m_\rho|^{-1}\right) 
= 0 \qquad \forall\ \vec\eta \in \Vhpartial\,.
\label{eq:proofnewb}
\end{align}
\end{subequations}
Choosing $\chi = \kappa_{\mathcal{S}} \in \Whpartialzero$ 
in (\ref{eq:proofnewa}) and 
$\vec\eta= \delta\vec X \in \Vhpartial$ in (\ref{eq:proofnewb}) yields that
\begin{equation} \label{eq:uniquenew0}
\ttau_m
\left( \vec X^m\,.\,\vec\ek_1 \,
|(\delta\vec X)_\rho|^2, |\vec X^m_\rho|^{-1}\right)
+ \left(\vec X^m \,.\,\vec\ek_1\,
|\kappa_{\mathcal{S}}|^2, |\vec X^m_\rho| \right)^{(h)} = 0\,.
\end{equation}
It immediately
follows from (\ref{eq:uniquenew0})
and the assumption $(\mathfrak A)$ that $\kappa_{\mathcal{S}} = 0$,
and that $\delta\vec X \equiv \vec X^c\in\bR^2$.
Hence it follows from 
(\ref{eq:proofnewa}) that
$\vec X^c\,.\,\vec z = 0$ for all $\vec z \in \mathcal Z^{(h)}_{(\partial_0)}$,
and so assumption $(\mathfrak C_{(\partial_0)})^{(h)}$ yields 
that $\vec X^c = \vec 0$.
Hence we have shown that $(\BGNmc_m)^{(h)}$ has a unique solution
$(\delta\vec X^{m+1},\kappa_{\mathcal{S}}^{m+1}) \in \Vhpartial \times
\Whpartialzero$.
\end{proof}

\begin{lem} \label{lem:exdziuk}
Let $\vec X^m \in \Vhpartialzero$ satisfy the assumption $(\mathfrak A)$.
Then there exists a unique solution 
$(\delta\vec X^{m+1},\vec\kappa_{\mathcal{S}}^{m+1}) \in \Vhpartial \times 
\vecWhpartialzero$ to $(\GDmc_m)^{(h)}$. 
\end{lem}
\begin{proof}
Similarly to the proof of Lemma~\ref{lem:exnew}, we obtain that
\begin{equation} \label{eq:uniquedziuk0}
\ttau_m \left( \vec X^m\,.\,\vec\ek_1 \,
|(\delta\vec X)_\rho|^2, |\vec X^m_\rho|^{-1}\right)
+ \left(\vec X^m \,.\,\vec\ek_1 \,|\vec\kappa_{\mathcal{S}}|^2, 
|\vec X^m_\rho| \right)^{(h)} = 0 \,,
\end{equation}
where $(\delta\vec X,\vec\kappa_{\mathcal{S}}) \in \Vhpartial \times
\vecWhpartialzero$ solve the linear homogeneous system corresponding to
(\ref{eq:fddziuknewa},b).
It immediately follows from (\ref{eq:uniquedziuk0})
and the assumption $(\mathfrak A)$ that $\vec\kappa_{\mathcal{S}} = \vec0$, 
and that $\delta\vec X = \vec X^c \in \bR^2$.
Combined with 
the homogeneous variant of (\ref{eq:fddziuknewa}) these imply that
$\vec X^c = \vec 0$.
Hence we have shown that (\ref{eq:fddziuknewa},b) has a unique solution
$(\delta\vec X^{m+1},\vec\kappa_{\mathcal{S}}^{m+1}) \in \Vhpartial \times
\vecWhpartialzero$.
\end{proof}

\begin{thm} \label{thm:C1}
Let $\partial_0 I = \emptyset$ and let $\vec X^m \in \Vh$ satisfy
the assumption $(\mathfrak A)$. 
Then there exists a solution 
$(\delta\vec X^{m+1},\vec\kappa_{\mathcal{S}}^{m+1}) \in \Vhpartial\times 
\vecWhpartialzero$ to $(\GDmc_{m,\star})^{(h)}$, 
{\rm (\ref{eq:dziukfdnonlineara},b)}, if $\ttau_m < 3\,
\min_{\overline I} (\vec X^m\,.\,\vec\ek_1)^2$.
\end{thm}
\begin{proof}
Let $\mathcal{F}_h^{(h)} : \Vhpartial \to \Vhpartial$ be such that for any
$\vec\chi\in\Vhpartial$ it holds that
\begin{align*} 
 \left(\mathcal{F}_h^{(h)}(\vec\chi), \vec\eta\right)^h & = 
\left(\vec X^m\,.\,\vec\ek_1\,\frac{\vec\chi}{\ttau_m},
\vec\eta\,|\vec X^m_\rho|\right)^{(h)}
+ \left( \vec\eta \,.\,\vec\ek_1, 
|\vec X^{m}_\rho + \vec\chi_\rho|\right)
\nonumber \\ & \quad
+ \left( (\vec X^m\,.\,\vec\ek_1)\,
(\vec X^{m}_\rho + \vec\chi_\rho), 
\vec\eta_\rho\, |\vec X^m_\rho|^{-1} \right) 
+ \sum_{p \in \partial_1 I} \sliprho^{(p)}\,
(\vec X^{m}(p)\,.\,\vec\ek_1)\,\vec\eta(p)\,.\,\vec\ek_2
\nonumber \\ & \quad
+ \sum_{p \in \partial_2 I} \left( \sliprho^{(p)}\,
\vec X^{m}(p)\,.\,\vec\ek_1 + [\sliprho^{(p)}]_+\,
\vec\chi(p)\,.\,\vec\ek_1 \right) \vec\eta(p)\,.\,\vec\ek_1
\qquad \forall\ \vec\eta \in \Vhpartial\,. 
\end{align*}
Upon eliminating $\vec\kappa_S^{m+1}$ from $(\GDmc_{m,\star})^{(h)}$, we can
rewrite it as:
Given $\vec X^m \in \Vhpartialzero$,
find $\delta\vec X^{m+1}\in \Vhpartial$ such that
\begin{equation} \label{eq:Dnew}
\left(\mathcal{F}_h^{(h)}(\delta\vec X^{m+1}),\vec\eta\right)^h
= 0 \qquad \forall\ \vec\eta \in \Vhpartial\,,
\end{equation}
which is equivalent to writing 
$\mathcal{F}_h^{(h)}(\delta\vec X^{m+1}) = \vec 0 \in \Vhpartial$.

On recalling assumption $(\mathfrak A)$, we note that 
$\partial_0 I = \emptyset$ implies that 
$\mu = \min_{\overline I} \vec X^m\,.\,\vec\ek_1 > 0$. It holds that
\begin{align}
& \left( \mathcal{F}_h^{(h)}(\vec\eta),\vec\eta\right)^h
 \geq \mu\,(\ttau_m)^{-1} \left( |\vec\eta|^2, |\vec X^m_\rho|\right)^{(h)}
- \left( |\vec\eta\,.\,\vec\ek_1|, 
|\vec X^{m}_\rho| + |\vec\eta_\rho|\right)
+ \mu \left( |\vec\eta_\rho|^2, |\vec X^m_\rho|^{-1} \right) 
\nonumber \\ & \quad
- \left( \vec X^m\,.\,\vec\ek_1, |\vec\eta_\rho| \right) 
+ \sum_{p \in \partial_2 I} [\sliprho^{(p)}]_+\,
(\vec\eta(p)\,.\,\vec\ek_1)^2
-\sum_{p \in \partial_1 I\cup \partial_2 I} |\sliprho^{(p)}|\,
\vec X^{m}(p)\,.\,\vec\ek_1\, |\vec\eta(p)|\,.
\label{eq:HGbound1}
\end{align}
In relation to the second term on the right hand side of (\ref{eq:HGbound1}) 
we observe, on recalling (\ref{eq:normequiv}), that
\begin{align} \label{eq:HGbound2}
\left(|\vec\eta\,.\,\vec\ek_1|,|\vec\eta_\rho|\right) \leq
\left(|\vec\eta|,|\vec\eta_\rho|\right)
& \leq \frac1{4\,\mu} \left( |\vec\eta|^2, |\vec X^m_\rho| \right)
+ \mu \left( |\vec\eta_\rho|^2, |\vec X^m_\rho|^{-1} \right) 
\nonumber \\ & 
\leq \frac1{4\,\mu} \left( |\vec\eta|^2, |\vec X^m_\rho| \right)^{(h)}
+ \mu \left( |\vec\eta_\rho|^2, |\vec X^m_\rho|^{-1} \right) .
\end{align}
Moreover, it holds that
\begin{equation} \label{eq:HGbound3}
\left(|\vec\eta\,.\,\vec\ek_1|,|\vec X^m_\rho|\right) \leq
\left(|\vec\eta|,|\vec X^m_\rho|\right)
\leq \tfrac18\,\mu\,(\ttau_m)^{-1} 
\left( |\vec\eta|^2, |\vec X^m_\rho| \right)^{(h)}
+ C(\mu, \ttau_m) \left( 1, |\vec X^m_\rho| \right) ,
\end{equation}
recall (\ref{eq:normequiv}), and similarly
\begin{equation} \label{eq:HGbound4}
\left( \vec X^m\,.\,\vec\ek_1, |\vec\eta_\rho| \right) \leq
\left( |\vec X^m|, |\vec\eta_\rho| \right) 
\leq \tfrac18\,\mu\,\left( |\vec\eta_\rho|^2, |\vec X^m_\rho|^{-1}\right)
+ C(\mu) \left( |\vec X^m|^2, |\vec X^m_\rho| \right)^{(h)}. 
\end{equation}
Combining (\ref{eq:HGbound1}), (\ref{eq:HGbound2}), (\ref{eq:HGbound3}) 
and (\ref{eq:HGbound4}) yields, on recalling (\ref{eq:rhobound}), that
\begin{align}
\left( \mathcal{F}_h^{(h)}(\vec\eta),\vec\eta\right)^h
& \geq \left(\tfrac78\,\mu\,(\ttau_m)^{-1} 
- \tfrac14\,\mu^{-1} \right)
\left( |\vec\eta|^2, |\vec X^m_\rho|\right)^{(h)}
-\sum_{p \in \partial_1 I \cup \partial_2 I} |\vec X^{m}(p)|\, |\vec\eta(p)|
\nonumber \\ & \quad
-C(\mu,\ttau_m,\vec X^m)
\nonumber \\ & 
\geq \tfrac14 \left(3\,\mu\,(\ttau_m)^{-1} - \mu^{-1} \right)
\left( |\vec\eta|^2, |\vec X^m_\rho|\right)^{(h)}
-C(\mu,\ttau_m,\vec X^m)\,,
\label{eq:HGbound5}
\end{align}
where in the last inequality we have used a Young's inequality and observed
from (\ref{eq:ip0}) and (\ref{eq:normequiv}) that $|\vec\eta(p)|^2
\leq 2\,(h\,|\vec X^m_\rho(p)|)^{-1} 
\left( |\vec\eta|^2, |\vec X^m_\rho|\right)^h
\leq C(\vec X^m)\,\left( |\vec\eta|^2, |\vec X^m_\rho|\right)^{(h)}$
for $p\in \partial_i I$, $i=1,2$.

Now choosing $\ttau_m < 3\,\mu^2$ in (\ref{eq:HGbound5}) we obtain
that 
\begin{equation*} 
\left( \mathcal{F}_h^{(h)}(\vec\eta),\vec\eta\right)^h
\geq 0 \quad \forall\ \vec\eta \in B^h_\gamma
= \{ \vec\zeta \in \Vhpartial:(\vec\zeta,\vec\zeta)^h = \gamma^2\}
\end{equation*}
holds for $\gamma$ sufficiently large, and so
the existence of a solution $\delta\vec X^{m+1} \in \Vhpartial$ 
to (\ref{eq:Dnew}) with
$(\delta\vec X^{m+1},\delta\vec X^{m+1})^h \leq \gamma^2$ follows from 
\cite[Prop.~2.8]{Zeidler86}. 
The existence of $\vec\kappa_{\mathcal{S}}^{m+1} \in \vecWhpartialzero$
then follows immediately from (\ref{eq:dziukfdnonlineara}). 
\end{proof}

We remark that although one can eliminate $\kappa_{\mathcal{S}}^{m+1}$ from the scheme $(\BGNmc_{m,\star})^{h}$, recall (\ref{eq:fdnonlinear}),
one cannot adapt the above proof for
the scheme $(\GDmc_{m,\star})^{h}$, as one would obtain
(\ref{eq:HGbound1}) with $( |\vec\eta|^2, |\vec X^m_\rho|)^{h}$ replaced by
$( |\vec\eta\,.\,\vec\omega^m|^2, |\vec X^m_\rho|)^{h}$, and so it is no longer
possible to bound e.g.\ the second term on the right hand side of 
(\ref{eq:HGbound1}).

\begin{thm} \label{thm:stab}
Let $\vec X^m \in \Vhpartialzero$ satisfy the assumption $(\mathfrak A)$.
Let $(\vec X^{m+1},\kappa_{\mathcal{S}}^{m+1})$ be a solution to 
$(\BGNmc_{m,\star})^{(h)}$, 
or let 
$(\vec X^{m+1},\vec\kappa_{\mathcal{S}}^{m+1})$ be a solution to 
$(\GDmc_{m,\star})^{(h)}$. 
Then it holds that
\begin{equation} \label{eq:stab}
E(\vec X^{m+1}) + 2\,\pi\,\ttau_m
\begin{cases}
\left(\vec X^m\,.\,\vec\ek_1\,|\kappa_{\mathcal{S}}^{m+1}|^2,
|\vec X^m_\rho|\right)^{(h)} \\
\left(\vec X^m\,.\,\vec\ek_1\,|\vec\kappa_{\mathcal{S}}^{m+1}|^2,
|\vec X^m_\rho|\right)^{(h)}
\end{cases}
 \leq E(\vec X^m)\,,
\end{equation}
respectively, where we recall the definition {\rm (\ref{eq:Eh})}.
\end{thm}
\begin{proof}
Choosing $\chi = \ttau_m\,\kappa_{\mathcal{S}}^{m+1}$ in 
(\ref{eq:fdnonlineara}) and
$\vec\eta = \vec X^{m+1} - \vec X^m \in \Vhpartial$ in (\ref{eq:fdnonlinearb}) 
yields, on noting that $\vec X^{m}(p)\,.\,\vec\ek_1 = \vec
X^{m+1}(p)\,.\,\vec\ek_1$ for $p \in \partial_1 I$, that
\begin{align}
& - \ttau_m \left(\vec X^m\,.\,\vec\ek_1\,|\kappa_{\mathcal{S}}^{m+1}|^2,
|\vec X^m_\rho|\right)^{(h)} \nonumber \\ & \qquad
= \left( \vec X^{m+1} - \vec X^m  ,\vec\ek_1\, |\vec X^{m+1}_\rho|\right)
+ \left( (\vec X^m\,.\,\vec\ek_1)\,
(\vec X^{m+1} - \vec X^m)_\rho,\vec X^{m+1}_\rho\, |\vec X^m_\rho|^{-1} \right)
\nonumber \\ & \qquad \qquad
+ \sum_{p \in \partial_1 I} \sliprho^{(p)}\,
(\vec X^{m}(p)\,.\,\vec\ek_1)\,(\vec X^{m+1}(p) - \vec X^m(p))\,.\,\vec\ek_2
\nonumber \\ & \qquad \qquad
+ \sum_{p \in \partial_2 I} ([\sliprho^{(p)}]_+\,\vec X^{m+1}(p) + 
[\sliprho^{(p)}]_-\,\vec X^{m}(p)]\,.\,\vec\ek_1)\,
(\vec X^{m+1}(p) - \vec X^m(p))\,.\,\vec\ek_1
\nonumber \\ & \qquad
\geq \left( \vec X^{m+1} - \vec X^m  ,\vec\ek_1\, |\vec X^{m+1}_\rho|\right)
+ \left( \vec X^m\,.\,\vec\ek_1,
|\vec X^{m+1}_\rho| - |\vec X^m_\rho| \right)
\nonumber \\ & \qquad \qquad
+ \sum_{p \in \partial_1 I} \sliprho^{(p)}\,
(\vec X^{m}(p)\,.\,\vec\ek_1)\,\vec X^{m+1}(p) \,.\,\vec\ek_2
- \sum_{p \in \partial_1 I} \sliprho^{(p)}\,
(\vec X^{m}(p)\,.\,\vec\ek_1)\,\vec X^{m}(p) \,.\,\vec\ek_2
\nonumber \\ & \qquad \qquad
+ \tfrac12 \sum_{p \in \partial_2 I} [\sliprho^{(p)}]_+\,
(\vec X^{m+1}(p)\,.\,\vec\ek_1)^2\,
- \tfrac12 \sum_{p \in \partial_2 I} [\sliprho^{(p)}]_+\,
(\vec X^{m}(p)\,.\,\vec\ek_1)^2\,
\nonumber \\ & \qquad \qquad
+ \tfrac12 \sum_{p \in \partial_2 I} [\sliprho^{(p)}]_-\,
(\vec X^{m+1}(p)\,.\,\vec\ek_1)^2\,
- \tfrac12 \sum_{p \in \partial_2 I} [\sliprho^{(p)}]_-\,
(\vec X^{m}(p)\,.\,\vec\ek_1)^2\,
\nonumber \\ & \qquad
= \left( \vec X^{m+1} \,.\,\vec\ek_1, |\vec X^{m+1}_\rho|\right)
- \left( \vec X^m\,.\,\vec\ek_1, |\vec X^m_\rho| \right)
\nonumber \\ & \qquad \qquad
+ \sum_{p \in \partial_1 I} \sliprho^{(p)}\,
(\vec X^{m+1}(p)\,.\,\vec\ek_1)\,\vec X^{m+1}(p) \,.\,\vec\ek_2
- \sum_{p \in \partial_1 I} \sliprho^{(p)}\,
(\vec X^{m}(p)\,.\,\vec\ek_1)\,\vec X^{m}(p) \,.\,\vec\ek_2
\nonumber \\ & \qquad \qquad
+ \tfrac12 \sum_{p \in \partial_2 I} \sliprho^{(p)}\,
(\vec X^{m+1}(p)\,.\,\vec\ek_1)^2\,
- \tfrac12 \sum_{p \in \partial_2 I} \sliprho^{(p)}\,
(\vec X^{m}(p)\,.\,\vec\ek_1)^2\,
\nonumber \\ & \qquad
= \frac1{2\,\pi}\,E(\vec X^{m+1}) 
- \frac1{2\,\pi}\,E(\vec X^m)\,,
\label{eq:stab1}
\end{align}
where we have used the two inequalities
$\vec a\,.\,(\vec a - \vec b) \geq |\vec b|\,(|\vec a| - |\vec b|)$
for $\vec a$, $\vec b \in \bR^2$,
and $2\,\gamma\,(\gamma - \alpha) \geq \gamma^2 - \alpha^2$ for
$\alpha,\gamma\in\bR$. This proves the desired result 
(\ref{eq:stab}) for (\ref{eq:fdnonlineara},b). The proof for
(\ref{eq:dziukfdnonlineara},b) is analogous.
\end{proof}

We note that the scheme {\rm (\ref{eq:fdnonlinearnu})}
can also be shown to be unconditionally stable, i.e.\ a solution to
{\rm (\ref{eq:fdnonlinearnu})} satisfies
\begin{equation*} 
E(\vec X^{m+1}) + 
2\,\pi\,\ttau_m
 \left(\vec X^m \,.\,\vec\ek_1
\left|\dfrac{\vec X^{m+1} - \vec X^m}{\ttau_m}\,.\,\vec\nu^m\right|^2 ,
|\vec X^m_\rho| \right)^h \leq E(\vec X^m)\,.
\end{equation*}

\subsection{Nonlinear mean curvature flow}
It is a simple matter to extend the presented fully discrete approximations to
the nonlinear mean curvature flow (\ref{eq:nlmcf}) and 
the volume preserving variant (\ref{eq:nlmcf2}). 
We recall the fully 3d parametric finite
element schemes (2.4a,b) and (2.5) in \cite{gflows3d} for the approximation
of (\ref{eq:nlmcf}) and (\ref{eq:nlmcf2}), respectively.
We can now define their natural axisymmetric analogues. 
For example, the natural adaptation $(\BGNmckappa^f_{m})^{h}$ of
the scheme $(\BGNmckappa_{m})^{h}$ to (\ref{eq:nlmcf}) is given by 
(\ref{eq:fda},b) with (\ref{eq:fda}) replaced by
\begin{equation} \label{eq:nlfda}
\left(\frac{\vec X^{m+1} - \vec X^m}{\ttau_m}, \chi\,\vec\nu^m\,|\vec
X^m_\rho|\right)^h
= \left( f(\kappa^{m+1} - \doctorkappa^{m}(\kappa^{m+1})),
\chi\,|\vec X^m_\rho|\right)^h 
\qquad \forall\ \chi \in V^h\,.
\end{equation}
Similarly, the natural adaptation $(\BGNmc^f_{m,\star})^{(h)}$ of
the scheme $(\BGNmc_{m,\star})^{(h)}$ to (\ref{eq:nlmcf}) is given by 
(\ref{eq:fdnonlineara},b) with (\ref{eq:fdnonlineara}) replaced by
\begin{align} \label{eq:nlfdnonlineara}
&  \left(\vec X^m\,.\,\vec\ek_1\,
\frac{\vec X^{m+1} - \vec X^m}{\ttau_m}, \chi\,\vec\nu^m\,
|\vec X^m_\rho|\right)^{(h)}
= \left(\vec X^m\,.\,\vec\ek_1\,\pi^h[f(\kappa_{\mathcal{S}}^{m+1})],
\chi\,|\vec X^m_\rho|\right)^{(h)} 
\nonumber \\ & \hspace{11cm}
\qquad \forall\ \chi \in \Whpartialzero\,.
\end{align}
Similarly to (\ref{eq:stab}), with the same choices of $\chi$ and $\vec\eta$,
it is then possible to prove that solutions to
$(\BGNmc^f_{m,\star})^{h}$
satisfy $E(\vec X^{m+1}) + 2\,\pi\,\ttau_m\,
\left( \vec X^m\,.\,\vec\ek_1\, f( \kappa_{\mathcal{S}}^{m+1}),
\kappa_{\mathcal{S}}^{m+1} \,|\vec X^m_\rho| \right)^{h} \leq E(\vec X^m)$, 
which 
provides a stability bound if $f$ is monotonically increasing with $f(0)=0$.

Finally, replacing (\ref{eq:nlfda}) with
\begin{align} \label{eq:nlVfda}
& \left(\frac{\vec X^{m+1} - \vec X^m}{\ttau_m}, \chi\,\vec\nu^m\,|\vec
X^m_\rho|\right)^h
= \left( f(\kappa^{m+1} - \doctorkappa^{m}(\kappa^{m+1})),
\chi\,|\vec X^m_\rho|\right)^h 
\nonumber \\ & \hspace{2cm}
- \frac{\left(\vec X^m\,.\,\vec\ek_1, f(\kappa^m - \doctorkappa^{m}(\kappa^m))\,
|\vec X^m_\rho|\right)^{h}}
{\left(\vec X^m\,.\,\vec\ek_1, |\vec X^m_\rho|\right)}
\left(\chi,|\vec X^m_\rho|\right)^{h} 
\quad \forall\ \chi \in V^h
\end{align}
and replacing (\ref{eq:nlfdnonlineara}) with
\begin{align} \label{eq:nlVfdnonlineara}
& \left(\vec X^m\,.\,\vec\ek_1\,
\frac{\vec X^{m+1} - \vec X^m}{\ttau_m}, \chi\,\vec\nu^m\,
|\vec X^m_\rho|\right)^{(h)}
= \left(\vec X^m\,.\,\vec\ek_1\,\pi^h[f(\kappa_{\mathcal{S}}^{m+1})],
\chi\,|\vec X^m_\rho|\right)^{(h)} 
\nonumber \\ & \hspace{2cm}
- \frac{\left(\vec X^m\,.\,\vec\ek_1, \pi^h[f(\kappa_{\mathcal{S}}^{m+1})]\,
|\vec X^m_\rho|\right)^{(h)}}
{\left(\vec X^m\,.\,\vec\ek_1, |\vec X^m_\rho|\right)}
\left(\vec X^m\,.\,\vec\ek_1, \chi\,|\vec X^m_\rho|\right)^{(h)} 
\quad \forall\ \chi \in \Whpartialzero
\end{align}
gives the fully discrete approximations 
$(\BGNmckappa^{f,V}_{m})^{h}$ and $(\BGNmc^{f,V}_{m,\star})^{(h)}$, 
respectively, of (\ref{eq:nlmcf2}).
For the case (\ref{eq:fmcf}) it is possible to prove a stability bound for
$(\BGNmc^{f,V}_{m,\star})^{(h)}$. In particular,
solutions to $(\BGNmc^{f,V}_{m,\star})^{(h)}$ satisfy, 
similarly to (\ref{eq:stab1}), that
\begin{align*}
& 
\frac1{2\,\pi}\,E(\vec X^{m}) - \frac1{2\,\pi}\,E(\vec X^{m+1})
\nonumber \\ & \
\geq \ttau_m \left(\vec X^m\,.\,\vec\ek_1\,|\kappa_{\mathcal{S}}^{m+1}|^2,
|\vec X^m_\rho|\right)^{(h)}
- \left[ \left( \vec X^m\,.\,\vec\ek_1,|\vec X^m_\rho| 
\right)\right]^{-1}
\left|\left( \vec X^m\,.\,\vec\ek_1,\kappa_{\mathcal{S}}^{m+1}\,
|\vec X^m_\rho| \right)^{(h)}\right|^2 \geq 0\,,
\end{align*}
where we have noted the Cauchy--Schwarz inequality.
For the fully discrete approximations $(\BGNmckappa^{f,V}_{m})^{h}$ 
and $(\BGNmc^{f,V}_{m,\star})^{(h)}$ it is not possible to prove a volume
conservation property. However, in practice all three schemes preserve the
enclosed volume well, with the relative volume loss decreasing as the
discretization parameters become smaller.

Finally, we note that for the schemes $(\BGNmckappa_m^f)^h$ and 
$(\BGNmckappa_m^{f,V})^h$, depending on the choice of $f$, 
existence and uniqueness results can be shown, see Appendix~\ref{sec:B}.

\begin{rem} \label{rem:Fscheme}
In order to be able to compute evolutions for the general flow
{\rm (\ref{eq:Fmg})}, we propose the scheme 
$(\BGNmckappa^F_{m})^{h}$, which can be obtained from
the scheme $(\BGNmckappa_{m})^{h}$, {\rm (\ref{eq:fda},b)}, 
by replacing {\rm (\ref{eq:fda})} with
\begin{equation} \label{eq:nlFda}
\left(\frac{\vec X^{m+1} - \vec X^m}{\ttau_m}, \chi\,\vec\nu^m\,|\vec
X^m_\rho|\right)^h
= \left( F(\kappa^{m} - \doctorkappa^{m}(\kappa^{m}), 
-\kappa^{m}\,\doctorkappa^{m}(\kappa^m)), \chi\,|\vec X^m_\rho|\right)^h 
\quad \forall\ \chi \in V^h\,.
\end{equation}
This is linear scheme for which existence of a unique solution, provided that
the assumption $(\mathfrak B)^h$ holds, can easily be shown. Moreover,
solutions to the semidiscrete variant of $(\BGNmckappa^F_{m})^{h}$
satisfy the equidistribution property {\rm (\ref{eq:equid})}. 
\end{rem}

\setcounter{equation}{0}
\section{Numerical results} \label{sec:nr}

As the fully discrete energy, we consider $E(\vec X^m)$, recall (\ref{eq:Eh}). 
We always employ uniform time steps, $\ttau_m = \ttau$, $m=0,\ldots,M-1$.

We also consider the ratio
\begin{equation} \label{eq:ratio}
\ratio^m = \dfrac{\max_{j=1\to J} |\vec{X}^m(q_j) - \vec{X}^m(q_{j-1})|}
{\min_{j=1\to J} |\vec{X}^m(q_j) - \vec{X}^m(q_{j-1})|}
\end{equation}
between the longest and shortest element of $\Gamma^m$, and are often
interested in the evolution of this ratio over time.

On recalling (\ref{eq:varkappa}), and given $\Gamma^0 = \vec X^0(\overline I)$,
for the scheme $(\BGNmckappa_m^{f,V})^h$ we define 
$\kappa^0 \in V^h$ via
\begin{equation*} 
\kappa^0 = \pi^h\left[\frac{\vec\kappa^0\,.\,\vec\omega^0}{|\vec\omega^0|}
\right],
\end{equation*}
recall (\ref{eq:omegam}), where $\vec\kappa^0\in \Vh$ is such that
\begin{equation*} 
\left( \vec\kappa^{0},\vec\eta\, |\vec X^0_\rho| \right)^h
+ \left( \vec{X}^{0}_\rho , \vec\eta_\rho\,|\vec X^0_\rho|^{-1} \right)
 = 0 \quad \forall\ \vec\eta \in \Vh\,.
\end{equation*}

\subsection{Numerical results for mean curvature flow} \label{sec:mcnr}
\subsubsection{Sphere}

It is easy to show that a sphere of radius $r(t)$, with
\begin{equation} \label{eq:truer}
r(t) = [r^2(0) - 4\,t]^\frac12\,,
\end{equation}
is a solution to (\ref{eq:mcfS}). 
We use this true solution for a convergence test for the various schemes for
mean curvature flow, similarly to Table~1 in \cite{gflows3d}. Here
we start with a nonuniform partitioning of a semicircle of radius 
$r(0)=r_0=1$ and compute the flow until time $T = 0.125$. 
In particular, we have $\partial_0 I = \partial I = \{0,1\}$ and we choose
$\vec X^0 \in \Vhpartialzero$ with
\begin{equation} \label{eq:X0}
\vec X^0(q_j) = r_0 \begin{pmatrix} 
\cos[(q_j-\tfrac12)\,\pi + 0.1\,\cos((q_j-\tfrac12)\,\pi)] \\
\sin[(q_j-\tfrac12)\,\pi + 0.1\,\cos((q_j-\tfrac12)\,\pi)]
\end{pmatrix}, \quad j = 0,\ldots,J\,,
\end{equation}
recall (\ref{eq:Jequi}). 
We compute the error
\begin{equation} \label{eq:errorXx}
\errorXx = \max_{m=1,\ldots,M} \max_{j=0,\ldots,J} | |\vec X^m(q_j)| - r(t_m)|
\end{equation}
over the time interval $[0,T]$ between
the true solution (\ref{eq:truer}) and the discrete solutions for the schemes
$(\BGNmckappa_m)^h$, $(\GDmckappa_m)^h$, $(\BGNmc_m)^{(h)}$, 
$(\GDmc_m)^{(h)}$, $(\BGNmc_{m,\star})^{(h)}$ and 
$(\GDmc_{m,\star})^{(h)}$.
Here we used the time step size $\ttau=0.1\,h^2_{\Gamma^0}$,
where $h_{\Gamma^0}$ is the maximal edge length of $\Gamma^0$.
The computed errors are reported in 
Tables~\ref{tab:mcf1}--\ref{tab:mcf_schemeD}.
Comparing the reported numbers with the values in 
Table~1 in \cite{gflows3d}, we see that the errors for the axisymmetric schemes
are significantly smaller than the values in Table~1 in \cite{gflows3d} for
similar discretization parameters.
It is clear from Table~\ref{tab:mcf1} that the schemes $(\BGNmckappa_m)^h$
and $(\GDmckappa_m)^h$ appear to converge with the optimal convergence rate 
of $\mathcal{O}(h^2_{\Gamma^0})$. Similarly, Tables~\ref{tab:mcf_schemeC} 
and \ref{tab:mcf_schemeD} suggest that the
schemes $(\BGNmc_{m,\star})^{(h)}$, $(\GDmc_m)^{(h)}$
and $(\GDmc_{m,\star})^{(h)}$
converge with an order slightly less than quadratic. 
We note that the
linear schemes $(\BGNmc_m)^{(h)}$ lead to 
solutions $\vec X^m$ with 
$\min_{\rho \in \overline I} \vec X^m(\rho)\,.\,\vec\ek_1 < 0$,
and so we cannot complete the evolutions. In particular, in practice the two 
boundary elements shrink in size due to the scheme's tangential motion. 
Once the element has shrunk to a length almost zero, the freely moving vertex 
can become negative. The linear schemes $(\GDmc_m)^{(h)}$ behave well, on the
other hand. But as they are very close to the nonlinear schemes
$(\GDmc_{m,\star})^{(h)}$, from now on we concentrate on the
schemes $(\GDmc_{m,\star})^{(h)}$, $(\BGNmc_{m,\star})^{(h)}$, 
$(\BGNmckappa_{m})^{h}$ and $(\GDmckappa_{m})^{h}$.
\begin{table}
\center
\begin{tabular}{|rr|c|c|c|c|}
\hline
& & \multicolumn{2}{c|}{$(\BGNmckappa_m)^h$}&
\multicolumn{2}{c|}{$(\GDmckappa_m)^h$} \\
$J$  & $h_{\Gamma^0}$ & $\errorXx$ & EOC & $\errorXx$ & EOC \\ \hline
32   & 1.0792e-01 & 7.3110e-04 & ---      & 1.2074e-03 & --- \\
64   & 5.3988e-02 & 1.8422e-04 & 1.990129 & 3.0227e-04 & 1.999490 \\         
128  & 2.6997e-02 & 4.6098e-05 & 1.998974 & 7.5534e-05 & 2.000961 \\
256  & 1.3499e-02 & 1.1525e-05 & 2.000044 & 1.8878e-05 & 2.000527 \\
512  & 6.7495e-03 & 2.8813e-06 & 1.999975 & 4.7192e-06 & 2.000092 \\
\hline
\end{tabular}
\caption{Errors for the convergence test for (\ref{eq:truer})
with $r_0 = 1$ over the time interval $[0,0.125]$.}
\label{tab:mcf1}
\end{table}%
\begin{table}
\center
\begin{tabular}{|r|c|c|c|c|c|c|}
\hline
 & $(\BGNmc_m)^h$ & $(\BGNmc_m)$ & 
\multicolumn{2}{c|}{$(\BGNmc_{m,\star})^h$} &
\multicolumn{2}{c|}{$(\BGNmc_{m,\star})$} \\ 
$J$ & $\errorXx$ & $\errorXx$ & $\errorXx$ & EOC & $\errorXx$ & EOC \\ \hline
32  & --- & --- & 6.5076e-03 & ---      & 3.7596e-03 & --- \\
64  & --- & --- & 1.9553e-03 & 1.736035 & 1.1565e-03 & 1.702088 \\
128 & --- & --- & 5.8247e-04 & 1.747414 & 3.5226e-04 & 1.715328 \\
256 & --- & --- & 1.7056e-04 & 1.771999 & 1.0672e-04 & 1.722902 \\
512 & --- & --- & 4.9112e-05 & 1.796132 & 3.2277e-05 & 1.725252 \\
\hline
\end{tabular}
\caption{Errors for the convergence test for (\ref{eq:truer})
with $r_0 = 1$ over the time interval $[0,0.125]$.}
\label{tab:mcf_schemeC}
\end{table}%
\begin{table}
\center
\begin{tabular}{|r|c|c|c|c|c|c|}
\hline
 & $(\GDmc_m)^h$ & $(\GDmc_m)$ & 
\multicolumn{2}{c|}{$(\GDmc_{m,\star})^h$} &
\multicolumn{2}{c|}{$(\GDmc_{m,\star})$} \\ 
$J$ & $\errorXx$ & $\errorXx$ & $\errorXx$ & EOC & $\errorXx$ & EOC \\ \hline
32  & 8.1006e-03 & 3.0757e-03 & 8.0470e-03 & ---      & 3.6921e-03 & --- \\
64  & 2.4707e-03 & 8.8590e-04 & 2.4549e-03 & 1.714070 & 1.0449e-03 & 1.822441 \\
128 & 7.3144e-04 & 2.5363e-04 & 7.2755e-04 & 1.754827 & 2.9111e-04 & 1.844024 \\
256 & 2.1165e-04 & 7.2522e-05 & 2.1075e-04 & 1.787609 & 8.0222e-05 & 1.859594 \\
512 & 6.0176e-05 & 2.0472e-05 & 5.9972e-05 & 1.813172 & 2.1916e-05 & 1.872013 \\
\hline
\end{tabular}
\caption{Errors for the convergence test for (\ref{eq:truer})
with $r_0 = 1$ over the time interval $[0,0.125]$.}
\label{tab:mcf_schemeD}
\end{table}%

\subsubsection{Torus}
We repeat the two torus experiments in Figures 5 and 6 in \cite{gflows3d}.
To this end, we let $\partial I = \emptyset$.
For an initial torus with radii $R=1$, $r=0.7$, we obtain a surface that closes
up towards a genus-0 surface, as in 
\cite[Fig.\ 5]{gflows3d}. See Figure~\ref{fig:torusR1r07}
for the simulation results for the scheme $(\BGNmckappa_m)^h$, for
the discretization parameters $J=256$ and $\ttau = 10^{-4}$.
\begin{figure}
\center
\mbox{
\includegraphics[angle=-90,width=0.30\textwidth]{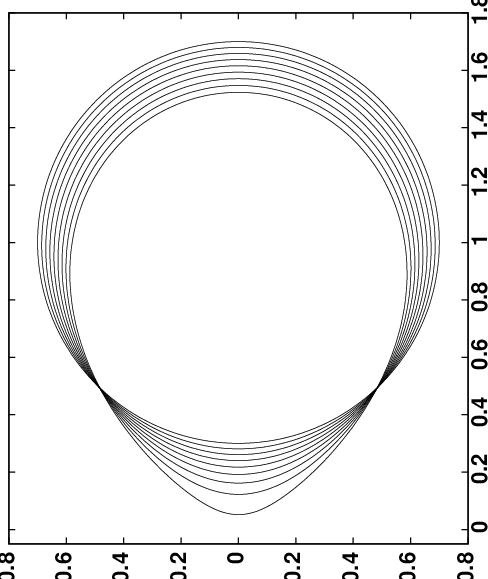}
\includegraphics[angle=-90,width=0.30\textwidth]{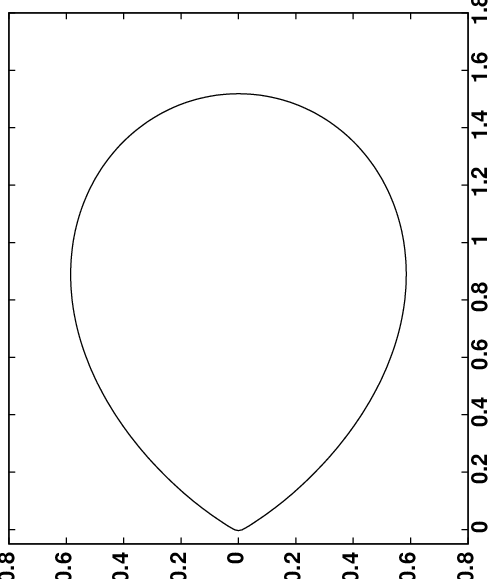}
\includegraphics[angle=-90,width=0.36\textwidth]{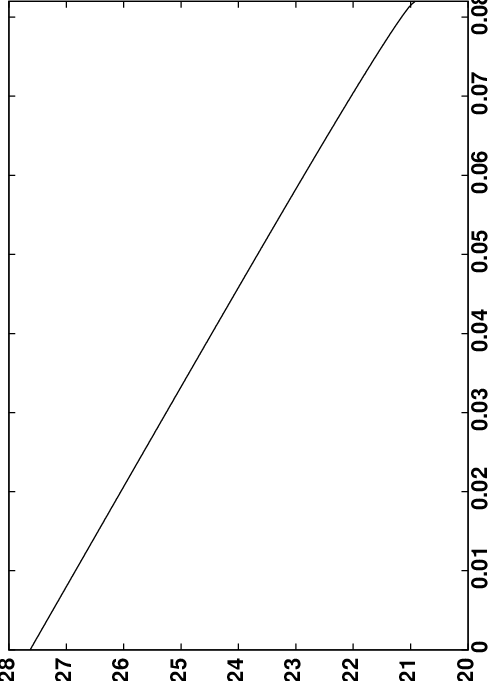}}
\includegraphics[angle=-90,width=0.20\textwidth]{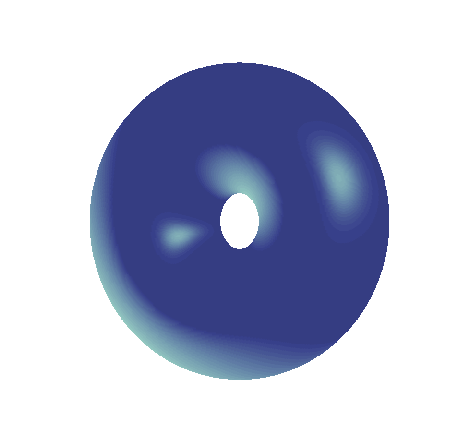}
\qquad
\includegraphics[angle=-90,width=0.20\textwidth]{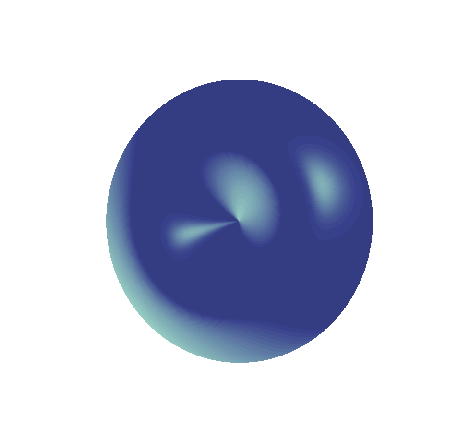}
\caption{$(\BGNmckappa_m)^h$
Evolution for a torus with radii $R=1$, $r=0.7$. Plots are at times
$t=0,0.01,\ldots,0.08$. We also show a plot at time $t=0.082$,
together with a plot of the discrete energy.
Below we visualize the axisymmetric surface $\mathcal{S}^m$ generated by
$\Gamma^m$ at times $t=0$ and $t=0.082$.}
\label{fig:torusR1r07}
\end{figure}%
On the other hand, for an initial torus with radii $R=1$, $r=0.5$, 
we obtain a shrinking evolution towards a circle, as in 
\cite[Fig.\ 6]{gflows3d}. See Figure~\ref{fig:torusR1r05}
for the evolution for the scheme $(\BGNmckappa_m)^h$, again for 
$J=256$ and $\ttau = 10^{-4}$.
\begin{figure}
\center
\includegraphics[angle=-90,width=0.30\textwidth]{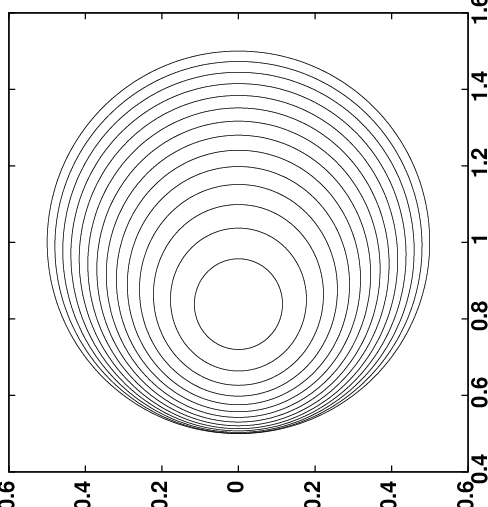}
\includegraphics[angle=-90,width=0.36\textwidth]{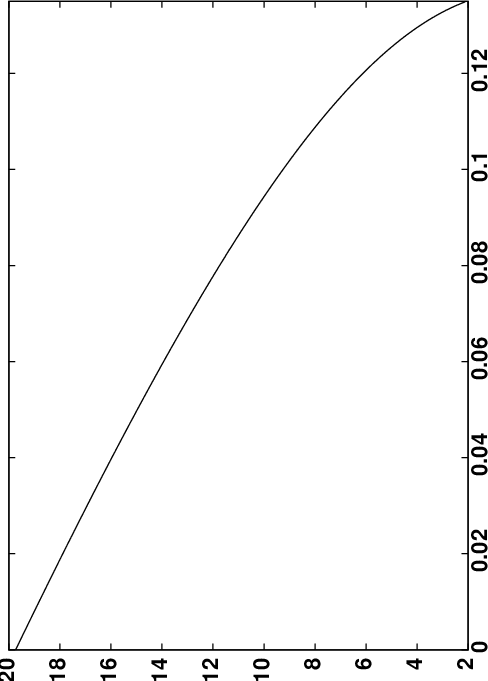}
\includegraphics[angle=-90,width=0.30\textwidth]{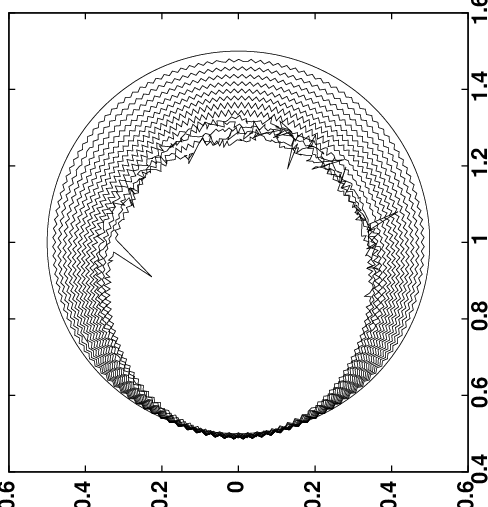}
\includegraphics[angle=-90,width=0.20\textwidth]{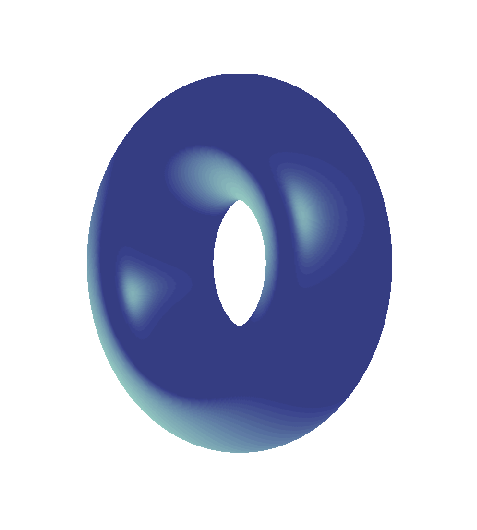}
\includegraphics[angle=-90,width=0.20\textwidth]{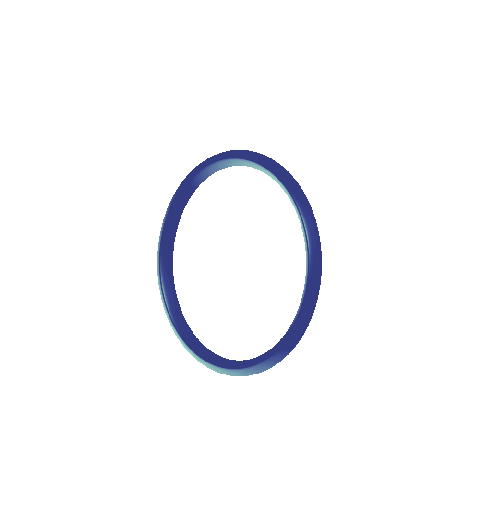}
\caption{$(\BGNmckappa_m)^h$
Evolution for a torus with radii $R=1$, $r=0.5$. Plots are at times
$t=0,0.01,\ldots,0.13$. 
We also show a plot of the discrete energy and, below,
we visualize the axisymmetric surface $\mathcal{S}^m$ generated by
$\Gamma^m$ at times $t=0$ and $t=0.135$.
On the top far right, the evolution for the scheme $(\GDmckappa_m)^h$,
with plots at times $t=0,0.01,\ldots,0.13$. 
}
\label{fig:torusR1r05}
\end{figure}%
On repeating the numerical experiment for the $(\GDmckappa_m)^h$, we observe
strong oscillations, as shown on the right of Figure~\ref{fig:torusR1r05}.
These oscillations become smaller in magnitude as $\ttau$ is decreased. 
The remaining schemes can integrate the evolution shown in 
Figure~\ref{fig:torusR1r05} in a stable way, and their numerical results
are very close to the ones displayed in Figure~\ref{fig:torusR1r05} for
the scheme $(\BGNmckappa_m)^h$. However, the schemes differ in the
exhibited tangential motions, which leads to diverse evolutions of
the ratio $\ratio^m$, see Figure~\ref{fig:torusR1r05ratios}.
The best distribution of mesh points is shown by the
schemes $(\BGNmckappa_m)^h$ and $(\BGNmc_{m,\star})^h$, followed by 
$(\BGNmc_{m,\star})$. The most nonuniform distribution of mesh points can be
observed for the two schemes $(\GDmc_{m,\star})^{(h)}$.
\begin{figure}
\center
\hspace*{-4mm}
\mbox{
\includegraphics[angle=-90,width=0.20\textwidth]{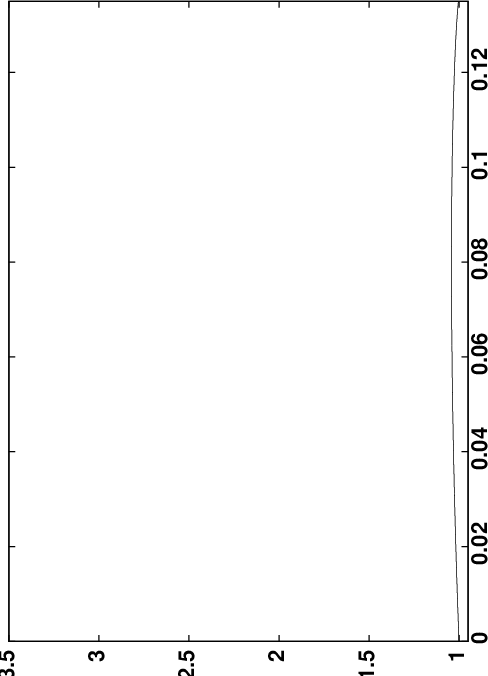}
\includegraphics[angle=-90,width=0.20\textwidth]{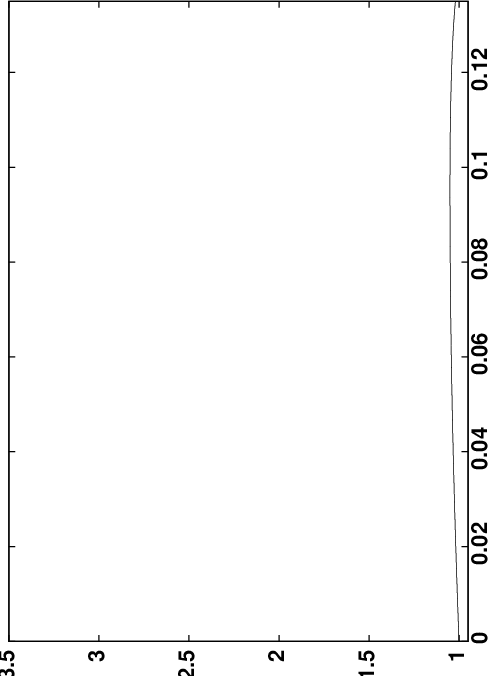}
\includegraphics[angle=-90,width=0.20\textwidth]{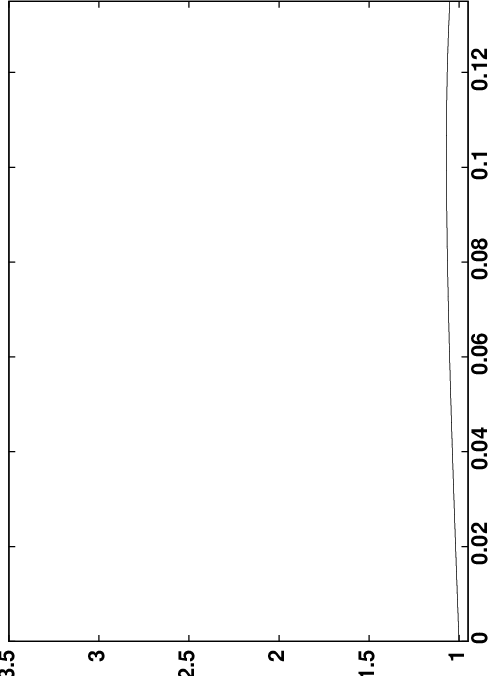}
\includegraphics[angle=-90,width=0.20\textwidth]{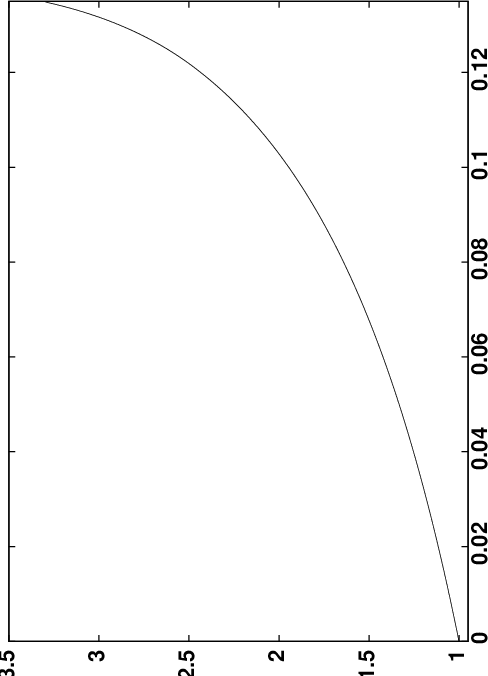}
\includegraphics[angle=-90,width=0.20\textwidth]{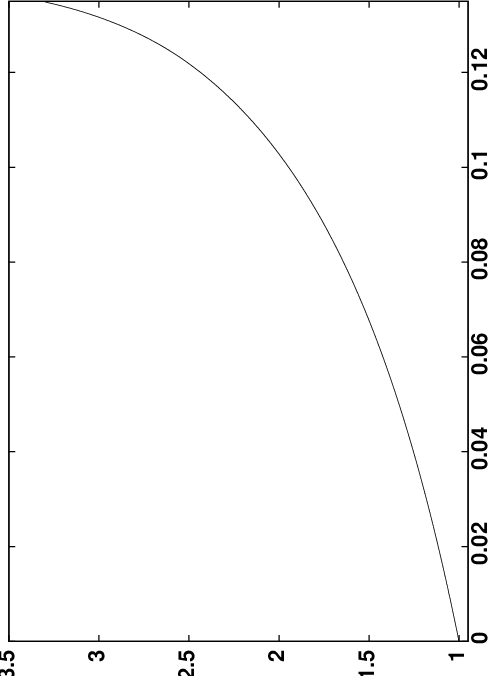}
}
\caption{
Plots of the ratio $\ratio^m$ for the schemes
$(\BGNmckappa_m)^h$, $(\BGNmc_{m,\star})^h$, 
$(\BGNmc_{m,\star})$, $(\GDmc_{m,\star})^h$, $(\GDmc_{m,\star})$.
}
\label{fig:torusR1r05ratios}
\end{figure}%

\subsubsection{Cylinder}

For the scheme $(\BGNmckappa_m)^h$
we repeat the singular evolution from \cite[Fig.\ 1]{DziukK91}. 
To this end, we set $\partial_D I = \partial I = \{0,1\}$.
In particular, starting with a cylinder, 
mean curvature flow leads to a pinch-off.
We show the results for the scheme $(\BGNmckappa_m)^h$,
with the discretization parameters $J=128$ and $\ttau = 10^{-4}$, 
in Figure~\ref{fig:dziukk91}.
\begin{figure}
\center
\begin{minipage}{0.7\textwidth}
\includegraphics[angle=-90,width=0.2\textwidth]{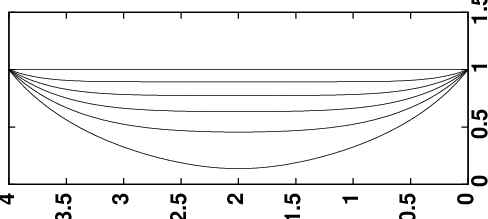}
\includegraphics[angle=-90,width=0.2\textwidth]{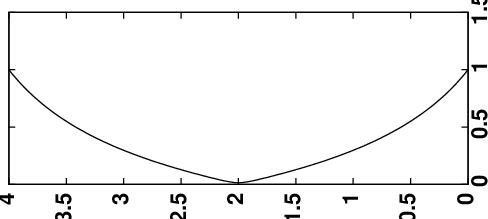}
\includegraphics[angle=-90,width=0.5\textwidth]{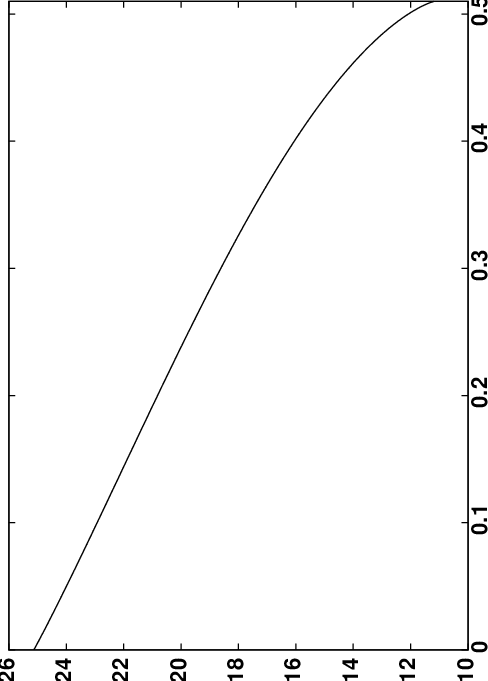}
\end{minipage}\quad
\begin{minipage}{0.25\textwidth}
\includegraphics[angle=-90,width=0.95\textwidth]{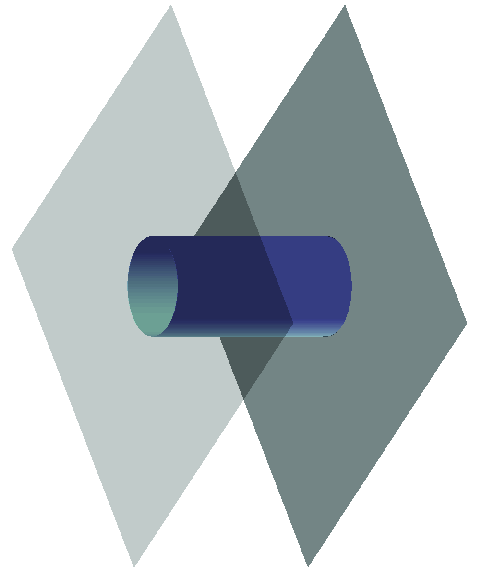}
\includegraphics[angle=-90,width=0.95\textwidth]{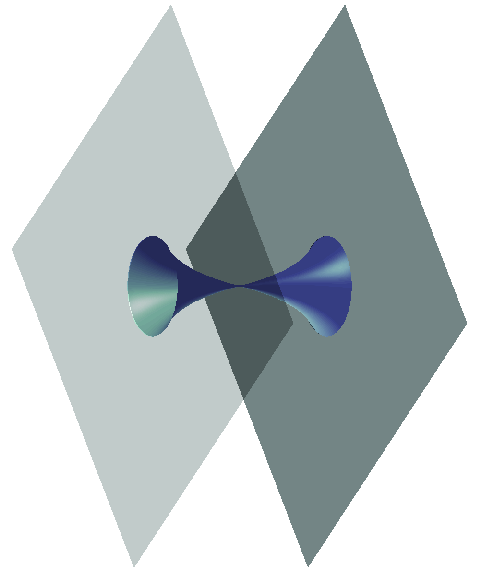}
\end{minipage}
\caption{$(\BGNmckappa_m)^h$ $[\partial_D I = \partial I = \{0,1\}]$
Evolution for a cylinder with fixed boundary. Plots are at times
$t=0,0.1,\ldots,0.5$. 
We also show a plot at time $t=0.51$, as well as a plot of the discrete energy.
On the right we visualize the axisymmetric surface $\mathcal{S}^m$ generated by
$\Gamma^m$ at times $t=0$ and $t=0.51$.
}
\label{fig:dziukk91}
\end{figure}%

For the next two experiments, we consider a cylinder attached to two
parallel hyperplanes, with prescribed contact angle conditions, recall
(\ref{eq:mcbc2}). To this end, we set $\partial_2 I = \partial I = \{0,1\}$ 
and use the discretization parameters $J=128$ and $\ttau = 10^{-3}$.
Letting $\sliprho^{(0)} = \sliprho^{(1)} = -\frac12$ and starting with a
cylinder, the evolution yields a growing catenoid-like surface,
see Figure~\ref{fig:cylinder_rho-05}. We observe convergence to a
travelling wave type solution, with the associated energy unbounded from 
below. We conjecture that the profile of the curve approaches in the limit the
so-called grim reaper solution, see \cite[p.\ 15]{Mantegazza11} and 
\cite{Grayson87},
\begin{equation} \label{eq:grimreaper}
\vec g (\rho,t) = (z_0 + \tfrac\pi3\,t)\,\vec\ek_1 + 
(-\tfrac3\pi\,\ln\cos\left(\tfrac\pi3\,(\rho-\tfrac12)\right), \rho)^T\,,
\end{equation}
where $z_0 \in \bR$ specifies the position of the travelling wave solution at
time $t=0$.
In fact, plotting $\vec g (\rho,t) - (z_0 + \tfrac\pi3\,t)\,\vec\ek_1$
at time $t=4$ versus 
$\vec X^m - (\min_{\rho \in \overline I} \vec X^m\,.\,\vec\ek_1)\,\vec\ek_1$
for our final solution in Figure~\ref{fig:cylinder_rho-05},
yields perfect agreement between the two graphs.
\begin{figure}
\center
\mbox{
\includegraphics[angle=-90,width=0.45\textwidth]{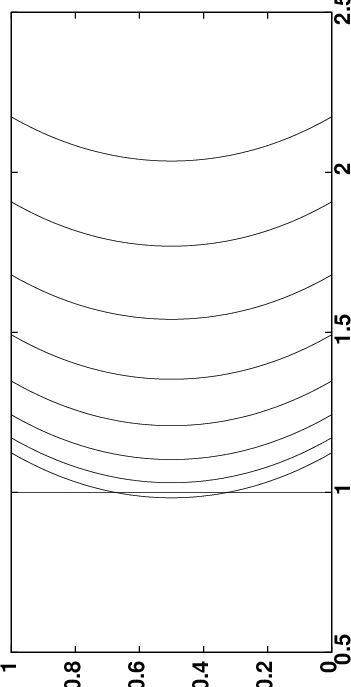}
\includegraphics[angle=-90,width=0.31\textwidth]{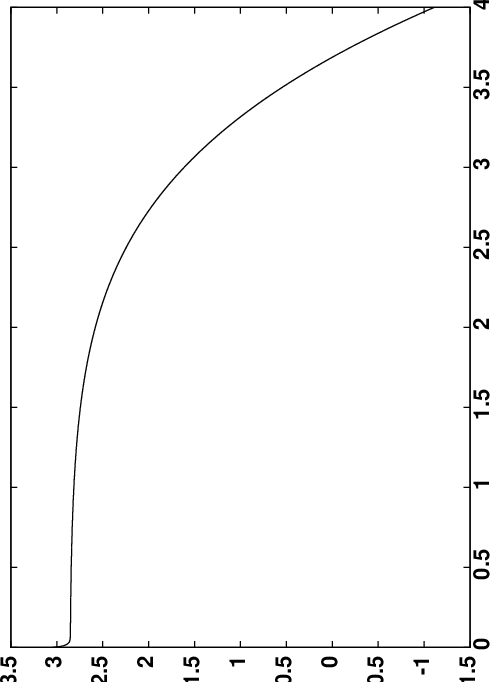}
\includegraphics[angle=-90,width=0.23\textwidth]{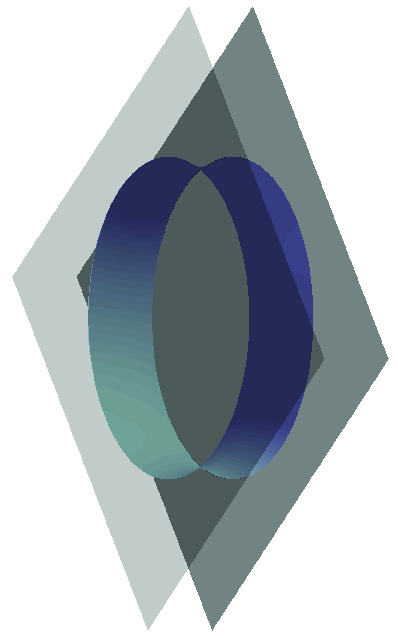}}
\caption{
$(\BGNmckappa_m)^h$ [$\partial_2 I = \partial I = \{0,1\}$, 
$\sliprho^{(0)} = \sliprho^{(1)} = -\frac12$]
Evolution for an open cylinder attached to $\bR \times \{ 0 \} \times \bR$
and $\bR \times \{ 1 \} \times \bR$. Solution at times $t=0,0.5,\ldots,4$,
as well as a plot of the discrete energy over time.
We also visualize the axisymmetric surface $\mathcal{S}^m$ generated by
$\Gamma^m$ at time $t=4$.
}
\label{fig:cylinder_rho-05}
\end{figure}%
We conjecture that the speed of the travelling wave type solution will approach
$\frac\pi3$, the speed of (\ref{eq:grimreaper}). To test this
conjecture, we continue the evolution until $t=100$ and plot the evolution of
$\vec X^m(0)\,.\,\vec\ek_1$ over time, comparing the graph with a suitably
chosen line with slope $\frac\pi3$, see Figure~\ref{fig:cylinder_rho-05_long}.
As we can see, the speed of the curve does indeed approach $\frac\pi3$.
\begin{figure}
\center
\includegraphics[angle=-90,width=0.3\textwidth]{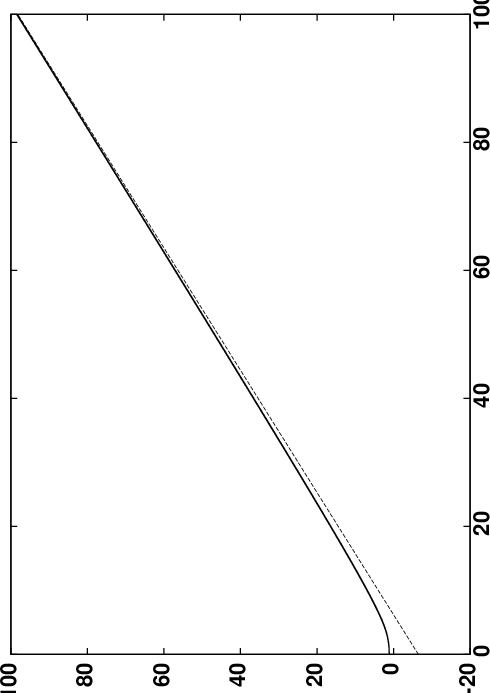}
\caption{
We plot $\vec X^m(0)\,.\,\vec\ek_1$ over time, compared with the linear
function $t \mapsto \frac\pi3\,t - \frac{13}2$, for the evolution in 
Figure~\ref{fig:cylinder_rho-05} over the larger time interval $[0,100]$.
}
\label{fig:cylinder_rho-05_long}
\end{figure}%

If we let $\sliprho^{(0)} = \sliprho^{(1)} = \frac12$, on the other hand,
we observe a shrinking surface, with the radius of the contact circles 
eventually converging to zero. On reaching two single contact points with the
external substrates, we allow the discrete surface to detach from the two
hyperplanes and to continue the evolution as a closed genus 0 surface, 
see Figure~\ref{fig:cylinder_rho05} for the evolution. To allow for an 
accurate resolution of the detaching, we employ the smaller time step size
$\ttau=10^{-6}$ for this simulation.
\begin{figure}
\center
\includegraphics[angle=-90,width=0.65\textwidth]{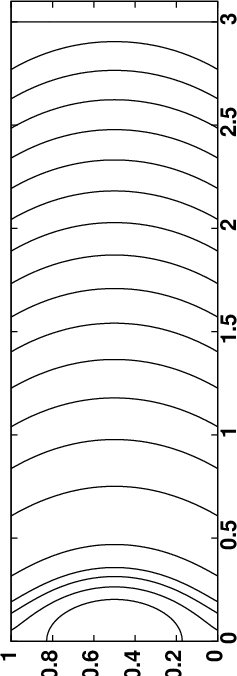}
\includegraphics[angle=-90,width=0.3\textwidth]{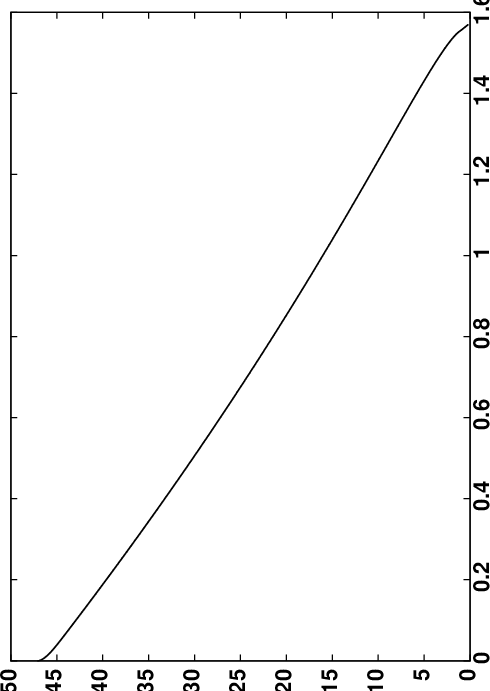}
\mbox{
\includegraphics[angle=-90,width=0.2\textwidth]{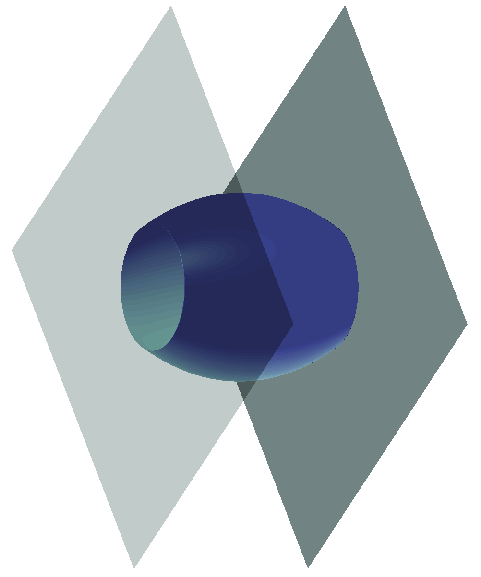}
\includegraphics[angle=-90,width=0.2\textwidth]{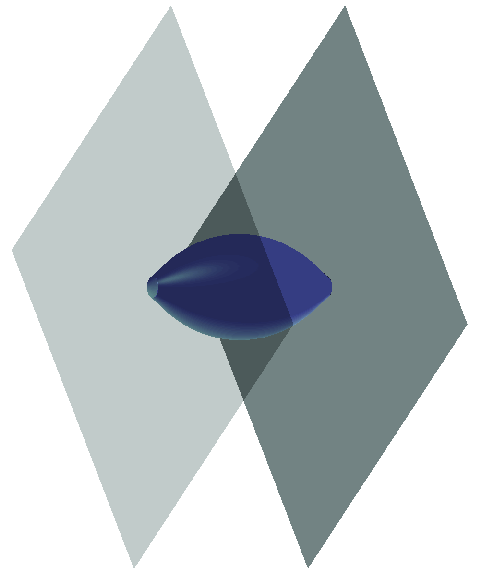}
\includegraphics[angle=-90,width=0.2\textwidth]{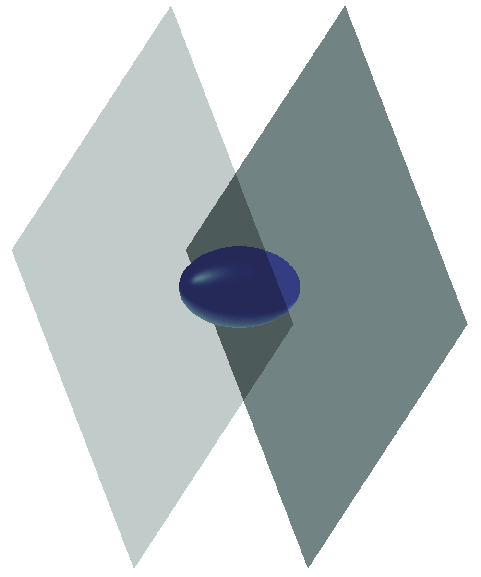}
}
\caption{
$(\BGNmckappa_m)^h$ [$\partial_2 I = \partial I = \{0,1\}$, 
$\sliprho^{(0)} = \sliprho^{(1)} = \frac12$]
Evolution for an open cylinder attached to $\bR \times \{ 0 \} \times \bR$
and $\bR \times \{ 1 \} \times \bR$. Solution at times $t=0,0.1,\ldots,
1.5,1.53,1.54,1.55,1.56$,
as well as a plot of the discrete energy over time.
We also visualize the axisymmetric surface $\mathcal{S}^m$ generated by
$\Gamma^m$ at times $t=1.5$, $t=1.55$ and $t=1.56$.
}
\label{fig:cylinder_rho05}
\end{figure}%

\subsubsection{Surface patch within a cylinder}
For the next experiment, we consider a disk attached to an infinite
cylinder of radius 1, with prescribed contact angle conditions, recall
(\ref{eq:mcbc1}). To this end, we set $\partial_0 I = \{0\}$ and
$\partial_1 I = \{1\}$,
and use the discretization parameters $J=128$ and $\ttau = 10^{-3}$.
Letting $\sliprho^{(1)} = -\frac12$ and starting with a
disk, the evolution seems to converge to a translating surface patch,
see Figure~\ref{fig:grimreaper_rho-05}. 
Taking  the angle condition (\ref{eq:mcbc1}), there is a unique convex
scaled surface grim reaper profile moving with constant speed by translation.
We conjecture that a general class of initial data will converge to this shape
for large times. We refer to \cite{AltschulerW94} for more information
on the grim reaper analogues in higher dimensions.
\begin{figure}
\center
\begin{minipage}{0.3\textwidth}
\includegraphics[angle=-90,width=0.75\textwidth]{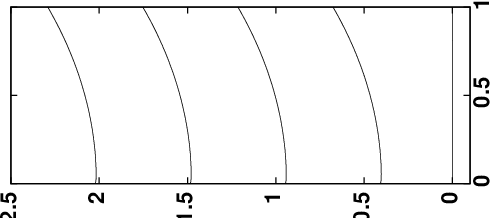}
\end{minipage} \quad
\begin{minipage}{0.45\textwidth}
\center
\includegraphics[angle=-90,width=0.70\textwidth]{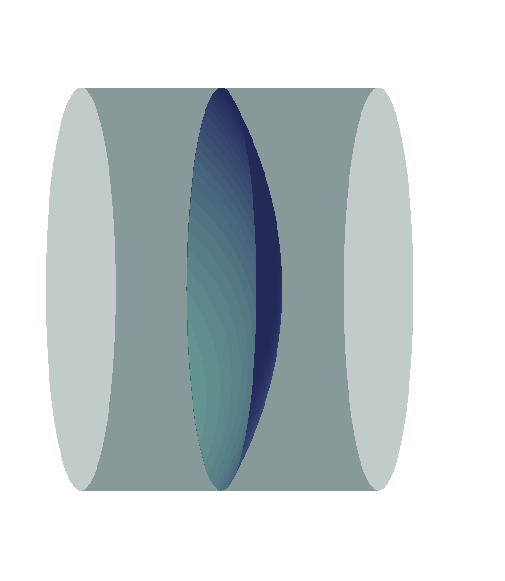}
\includegraphics[angle=-90,width=0.60\textwidth]{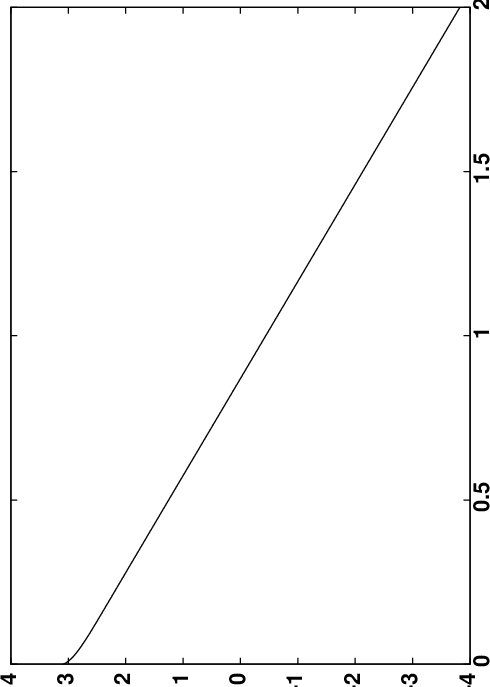}
\end{minipage}
\caption{
$(\BGNmckappa_m)^h$ [$\partial_0 I = \{0\}$, $\partial_1 I = \{1\}$, 
$\sliprho^{(1)} = -\frac12$]
Evolution for a disk attached to an infinite cylinder of radius 1. 
Solution at times $t=0,0.5,\ldots,2$,
as well as a plot of the discrete energy over time.
We also visualize the axisymmetric surface $\mathcal{S}^m$ generated by
$\Gamma^m$ at time $t=2$.
}
\label{fig:grimreaper_rho-05}
\end{figure}%

\subsection{Numerical results for conserved mean curvature flow} 
\label{sec:mcnrV}

\subsubsection{Sphere}
Clearly, a sphere is a stationary solution for conserved mean curvature flow,
(\ref{eq:nlmcf2}) with (\ref{eq:fmcf}). 
Hence, setting $\partial_0 I = \partial I = \{0,1\}$ and
choosing as initial data the nonuniform approximation of a semicircle
(\ref{eq:X0}) with $J=64$, 
we now investigate the different tangential motions exhibited by our proposed
schemes.
The initial data $\vec X^0$ has a ratio $\ratio^0=1.22$, 
recall (\ref{eq:ratio}).
We set $\ttau = 10^{-4}$ and integrate the evolution until time $T=1$. 
For the three schemes 
$(\BGNmckappa_m)^h$, $(\BGNmc_{m,\star})^h$ and $(\BGNmc_{m,\star})$
the element ratios $\ratio^m$ at time $T=1$ are 
$1.01, 73.13, 2.94$, 
and the enclosed volume is preserved almost exactly by all 
the schemes. We show the final distributions of vertices, 
and plots of $\ratio^m$ over time in Figure~\ref{fig:mcfTM}.
\begin{figure}
\center
\hspace*{-2mm}
\mbox{
\includegraphics[angle=-90,width=0.2\textwidth]{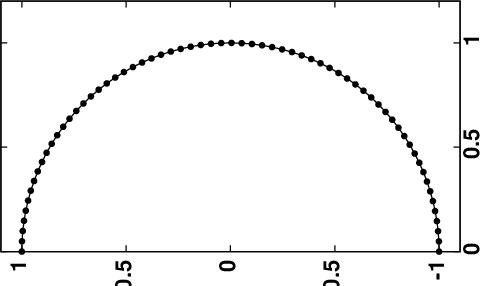}
\qquad\qquad
\includegraphics[angle=-90,width=0.2\textwidth]{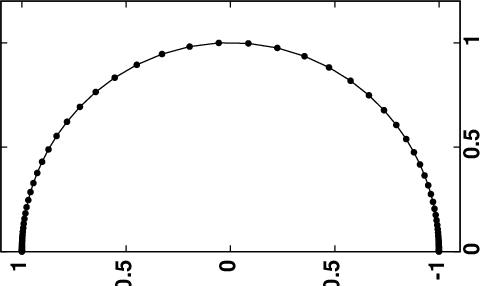}
\qquad\qquad
\includegraphics[angle=-90,width=0.2\textwidth]{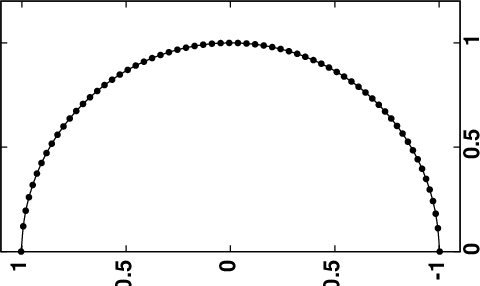}}
\hspace*{-2mm}
\mbox{
\includegraphics[angle=-90,width=0.3\textwidth]{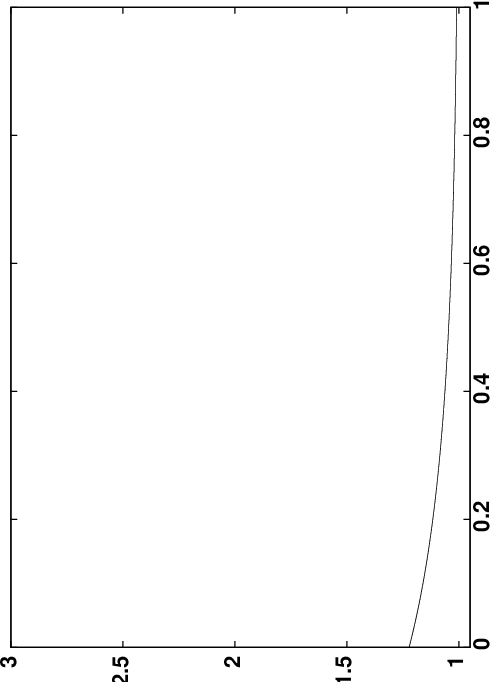}
\includegraphics[angle=-90,width=0.3\textwidth]{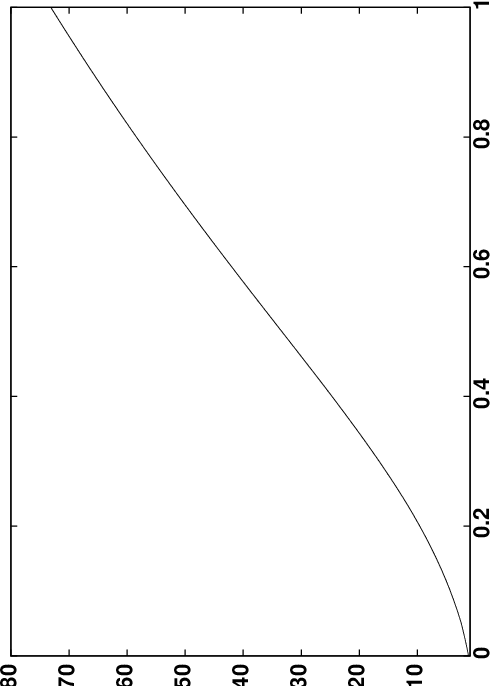}
\includegraphics[angle=-90,width=0.3\textwidth]{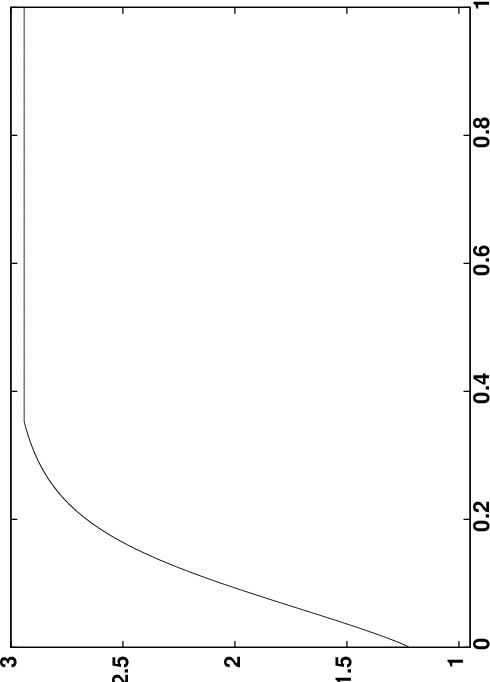}}
\caption{
Comparison of the different schemes for conserved mean curvature flow,
(\ref{eq:nlmcf2}) with (\ref{eq:fmcf}), of the unit sphere.
Left to right: $(\BGNmckappa_m^{f,V})^h$, $(\BGNmc_{m,\star}^{f,V})^h$ 
and $(\BGNmc_{m,\star}^{f,V})$.
Plots are for $\vec X^m$ at time $t=1$ and for the ratio $\ratio^m$ over time.
The element ratios $\ratio^m$ at time $t=1$ are
$1.01$, $73.13$ and $2.94$, respectively. 
}
\label{fig:mcfTM}
\end{figure}%
An insight that we gain from this set of experiments is that the tangential
motion displayed by the scheme $(\BGNmc_{m,\star})^h$ can lead to very
nonuniform meshes. Hence, for the remainder of this paper, we will only 
present numerical results for the two schemes $(\BGNmckappa_m)^h$ and 
$(\BGNmc_{m,\star})$ and their nonlinear variants. 
Note that the former is a linear fully
discrete approximation of $(\BGNmckappa_h)^h$, for which the
equidistribution property (\ref{eq:equid}) holds. The latter, on the other
hand, is a nonlinear scheme that is unconditionally stable, recall
Theorem~\ref{thm:stab}.
As the results for $(\BGNmckappa_m)^h$ and 
$(\BGNmc_{m,\star})$ are often indistinguishable, we only 
visualize the numerical results for the former from now on. 

\subsubsection{Genus 0 surface}
An experiment for a cigar shape can be seen in 
Figure~\ref{fig:mctallcigar}. Here we have once again that
$\partial_0 I = \partial I = \{0,1\}$.
The discretization parameters are $J=128$ and $\ttau = 10^{-4}$.
The relative volume loss for this experiment for $(\BGNmckappa_m^{f,V})^h$ 
is $0.09\%$, while for $(\BGNmc_{m,\star}^{f,V})$ it is $-0.01\%$.
The same experiment for the scheme $(\BGNmc_{m,\star}^{f,V})^h$ yields a very
nonuniform mesh, with the final ratio $\ratio^m > 430$.
\begin{figure}
\center
\begin{minipage}{0.35\textwidth}
\includegraphics[angle=-90,width=0.5\textwidth]{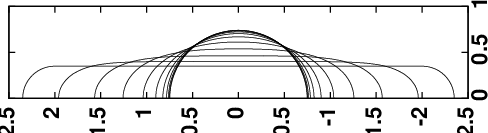} \
\includegraphics[angle=-90,width=0.45\textwidth]{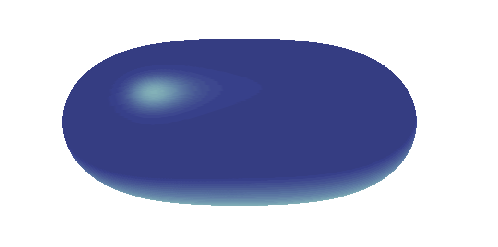}
\end{minipage} \qquad
\begin{minipage}{0.35\textwidth}
\includegraphics[angle=-90,width=0.95\textwidth]{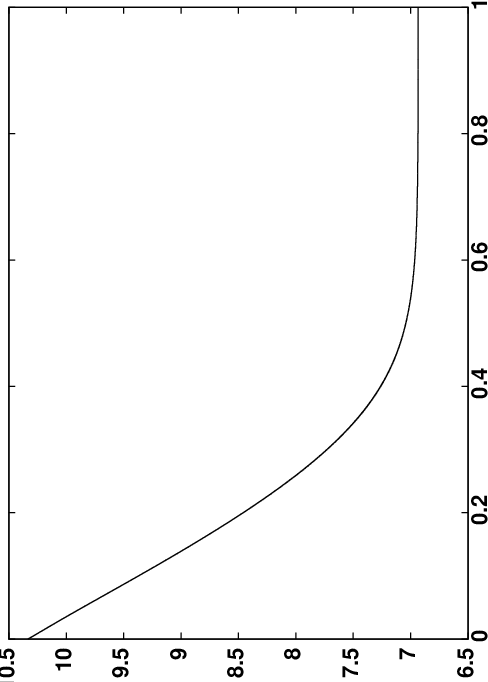} \\
\includegraphics[angle=-90,width=0.95\textwidth]{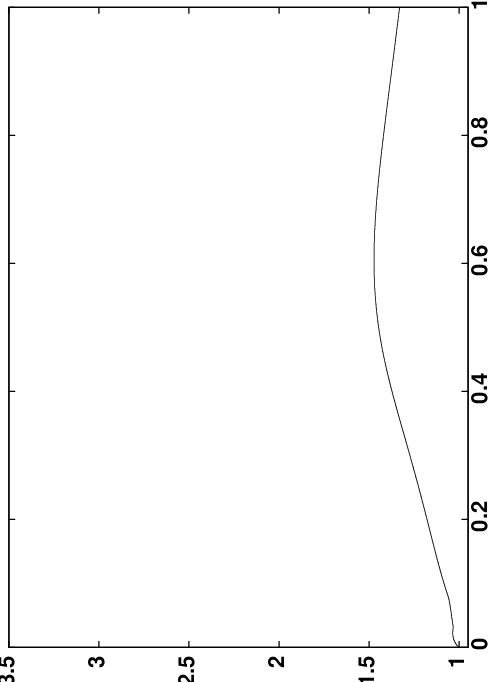} \\
\includegraphics[angle=-90,width=0.95\textwidth]{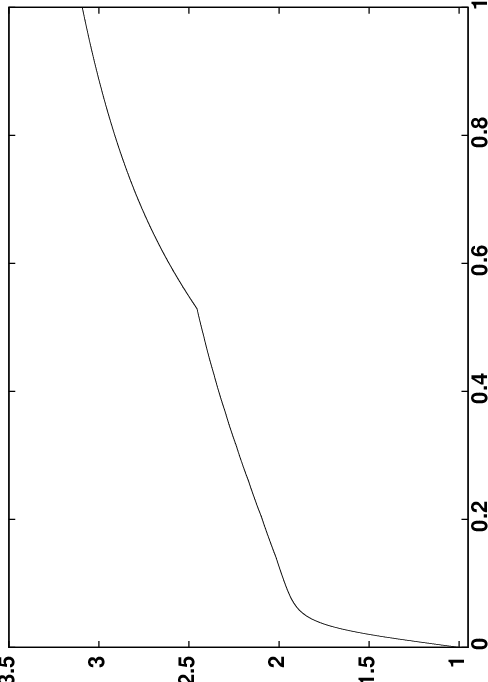}
\end{minipage}
\caption{
$(\BGNmckappa_m^{f,V})^h$ for (\ref{eq:fmcf}). 
Conserved mean curvature flow for a cigar. Plots are at times
$t=0,0.1,\ldots,1$. 
We also visualize the axisymmetric surface $\mathcal{S}^m$ generated by
$\Gamma^m$ at time $t=0.3$.
On the right are plots of the discrete energy and the ratio $\ratio^m$ and,
as a comparison, a plot of the ratio $\ratio^m$ for the scheme 
$(\BGNmc_{m,\star}^{f,V})$.
}
\label{fig:mctallcigar}
\end{figure}%
An experiment for a disc shape is shown in Figure~\ref{fig:mcflatcigar}.
The discretization parameters are $J=128$ and $\ttau = 10^{-4}$.
The relative volume loss for this experiment for $(\BGNmckappa_m^{f,V})^h$ 
is $-0.02\%$, while for $(\BGNmc_{m,\star}^{f,V})$ it is $-0.01\%$.
Once again, the scheme $(\BGNmc_{m,\star}^{f,V})^h$ yields a very
nonuniform mesh for this simulation, with the final ratio $\ratio^m > 145$.
\begin{figure}
\center
\begin{minipage}{0.35\textwidth}
\includegraphics[angle=-90,width=0.95\textwidth]{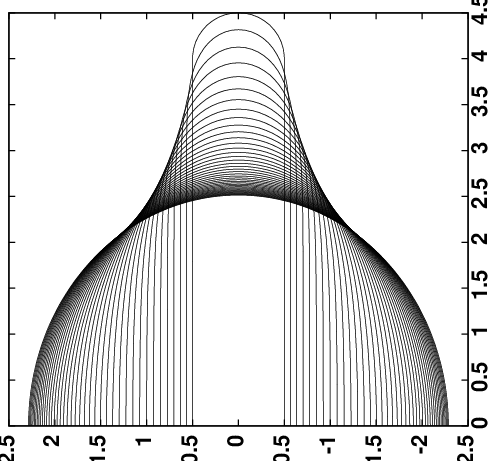} \\
\includegraphics[angle=-90,width=0.95\textwidth]{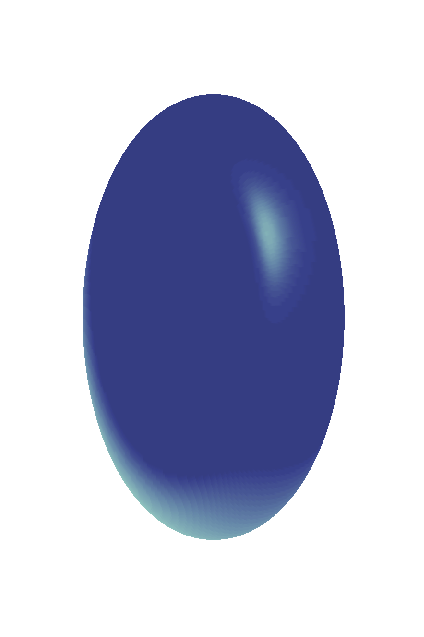}
\end{minipage} \qquad
\begin{minipage}{0.35\textwidth}
\includegraphics[angle=-90,width=0.95\textwidth]{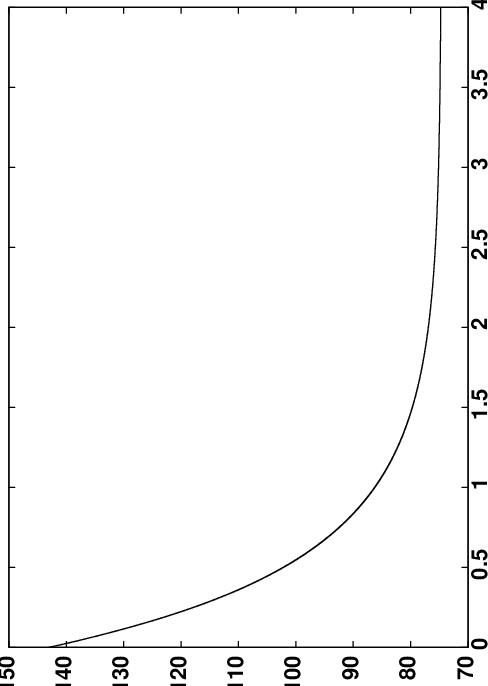} \\
\includegraphics[angle=-90,width=0.95\textwidth]{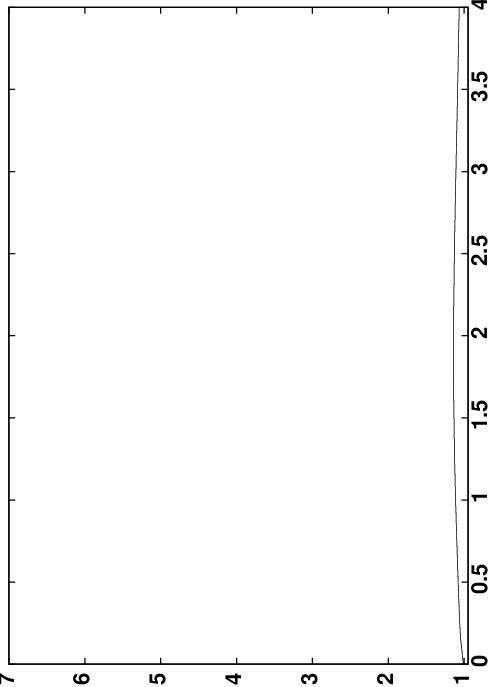} \\
\includegraphics[angle=-90,width=0.95\textwidth]{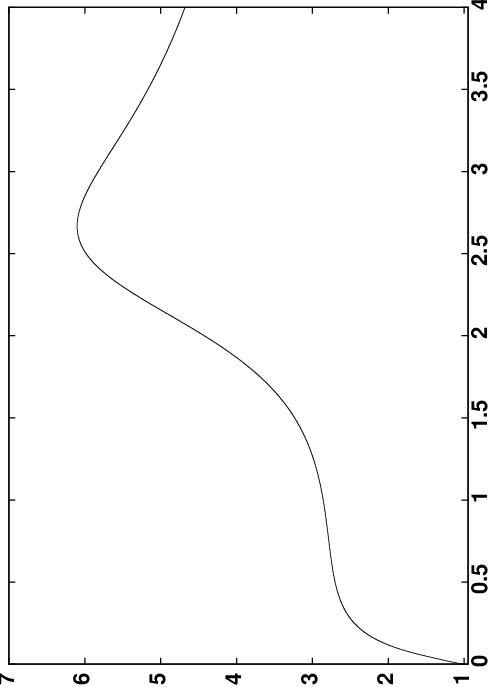}
\end{minipage}
\caption{
$(\BGNmckappa_m^{f,V})^h$ for (\ref{eq:fmcf}). 
Conserved mean curvature flow for a disc. 
Plots are at times $t=0,0.1,\ldots,4$. 
We also visualize the axisymmetric surface $\mathcal{S}^m$ generated by
$\Gamma^m$ at time $t=0.5$.
On the right are plots of the discrete energy and the ratio $\ratio^m$ and,
as a comparison, a plot of the ratio $\ratio^m$ for the scheme 
$(\BGNmc_{m,\star}^{f,V})$.
}
\label{fig:mcflatcigar}
\end{figure}%

\subsubsection{Genus 1 surface}
We repeat the simulation in Figure~\ref{fig:torusR1r05} for conserved mean 
curvature flow, i.e.\ (\ref{eq:nlmcf2}) with (\ref{eq:fmcf}),
using the scheme $(\BGNmckappa_m^{f,V})^h$. Conservation of
the enclosed volume means that the torus can no longer shrink to a circle.
Hence the torus now attempts to close up and change topology, as can be seen
from the numerical results in Figure~\ref{fig:torusR1r05cons}.
As for the original experiment, we use the discretization parameters 
$J=256$ and $\ttau = 10^{-4}$. The relative enclose volume loss for this
experiment is $-0.00\%$.
The evolutions for the schemes $(\BGNmc_{m,\star}^{f,V})^{h}$ and
$(\BGNmc_{m,\star}^{f,V})$ are nearly identical to what is shown in 
Figure~\ref{fig:torusR1r05cons}, 
with a relative volume loss of $0.01\%$ in both cases.
\begin{figure}
\center
\mbox{
\includegraphics[angle=-90,width=0.32\textwidth]{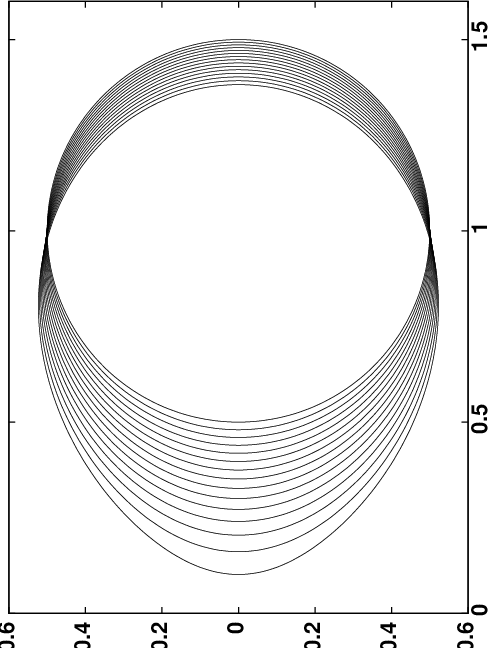}
\includegraphics[angle=-90,width=0.32\textwidth]{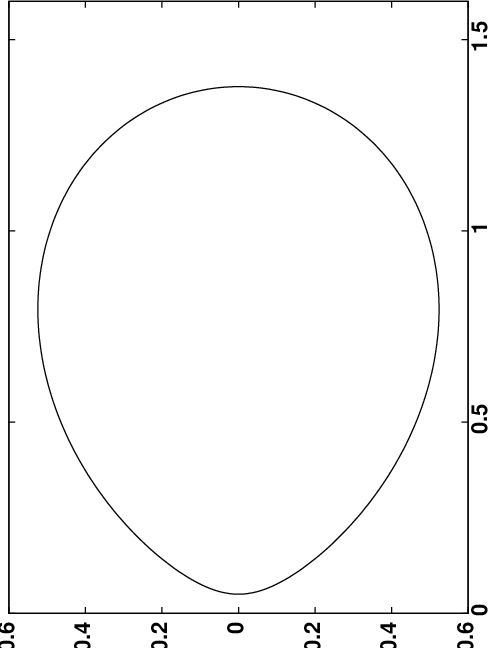}
\includegraphics[angle=-90,width=0.34\textwidth]{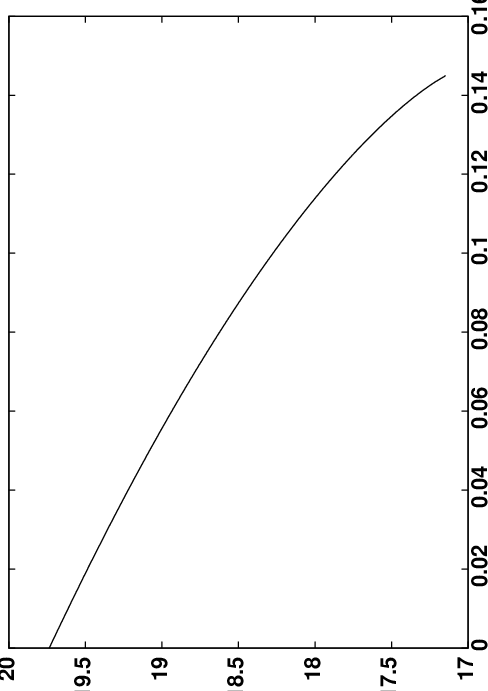}}
\includegraphics[angle=-90,width=0.20\textwidth]{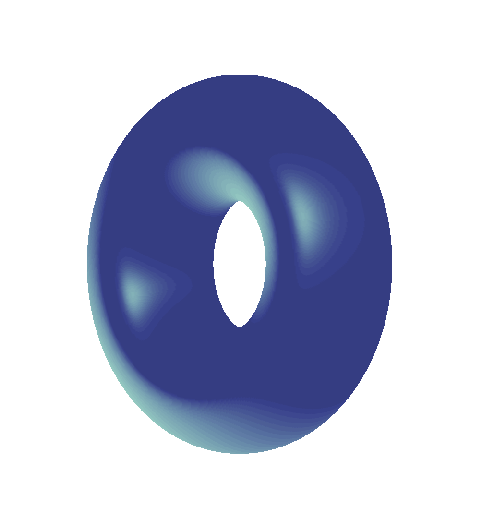}
\qquad
\includegraphics[angle=-90,width=0.20\textwidth]{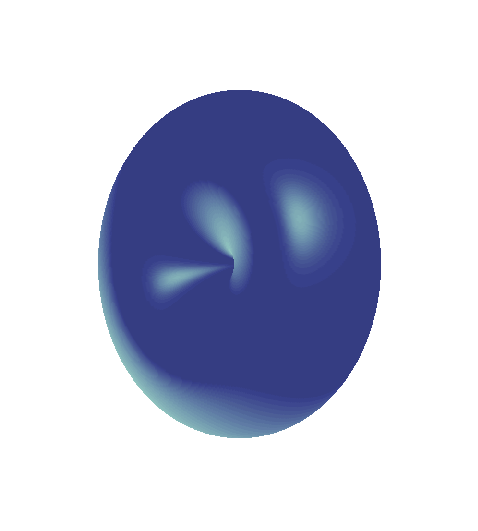}
\caption{$(\BGNmckappa_m^{f,V})^h$ for (\ref{eq:fmcf}). 
Conserved mean curvature flow for a torus with radii $R=1$, $r=0.5$. 
Plots are at times $t=0,0.01,\ldots,0.14$. 
We also show a plot at time $t=0.145$,
together with a plot of the discrete energy over time.
Below we visualize the axisymmetric surface $\mathcal{S}^m$ generated by
$\Gamma^m$ at times $t=0$ and $t=0.145$.}
\label{fig:torusR1r05cons}
\end{figure}%

Finally, we present an example for conserved mean curvature flow, 
(\ref{eq:nlmcf2}) with (\ref{eq:fmcf}), for the scheme 
$(\BGNmckappa_m^{f,V})^h$ with the initial data $\vec X^0$ parameterizing a
closed spiral, so that the approximated surface has genus 1. 
As can be seen from Figure~\ref{fig:spiral}, the spiral slowly untangles,
until the surface becomes a torus.
For this experiment we use the discretization parameters 
$J=1024$ and $\ttau = 10^{-6}$. The relative enclosed volume loss for this
experiment is $0.01\%$.
The evolutions for the schemes $(\BGNmc_{m,\star}^{f,V})^{h}$ and
$(\BGNmc_{m,\star}^{f,V})$ are nearly identical to what is shown in 
Figure~\ref{fig:spiral}, 
with a relative volume loss of $0.01\%$ in both cases.
\begin{figure}
\center
\newcommand\localwidth{0.24\textwidth}
\includegraphics[angle=-90,width=\localwidth]{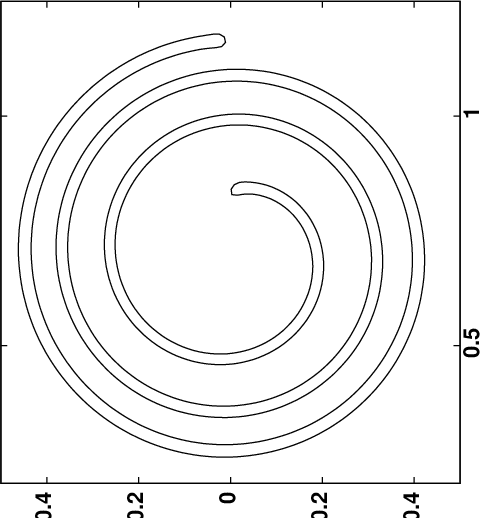}
\includegraphics[angle=-90,width=\localwidth]{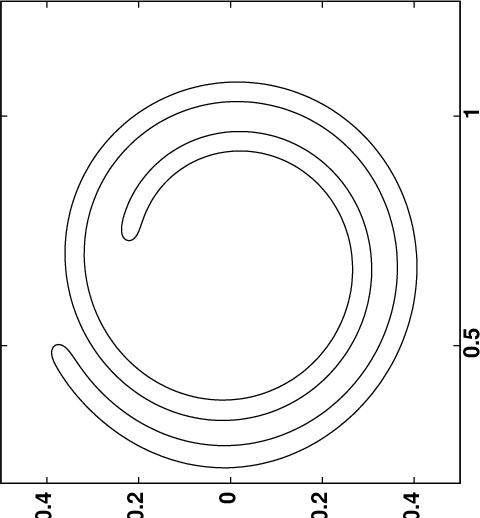}
\includegraphics[angle=-90,width=\localwidth]{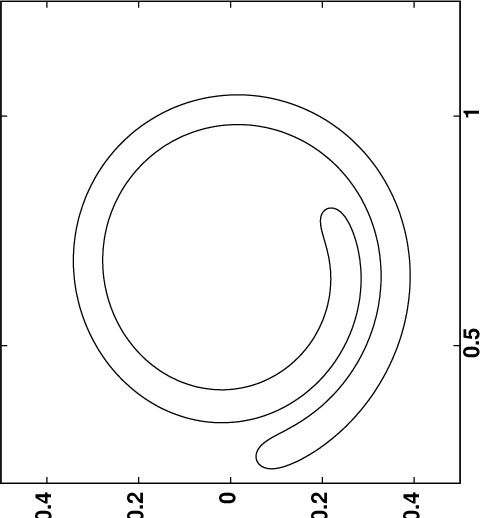}
\includegraphics[angle=-90,width=\localwidth]{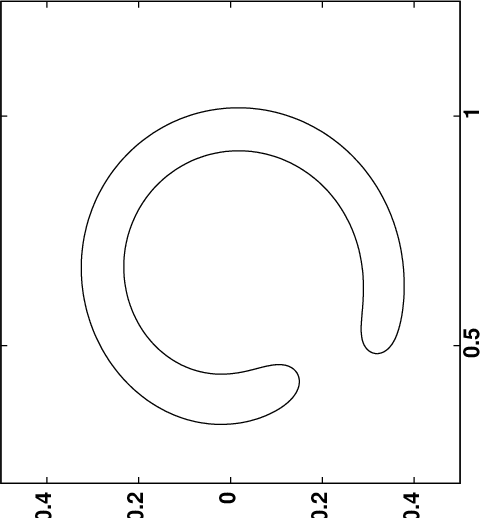}
\includegraphics[angle=-90,width=\localwidth]{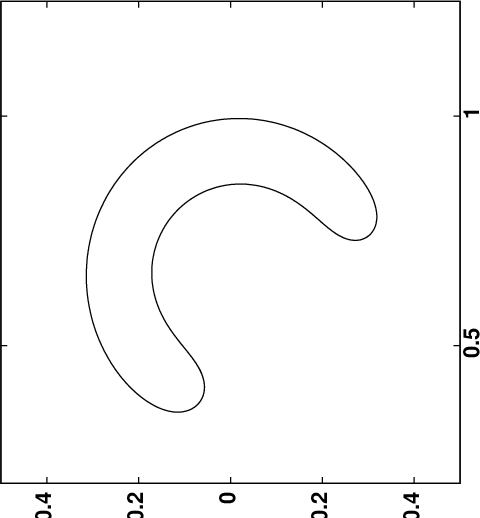}
\includegraphics[angle=-90,width=\localwidth]{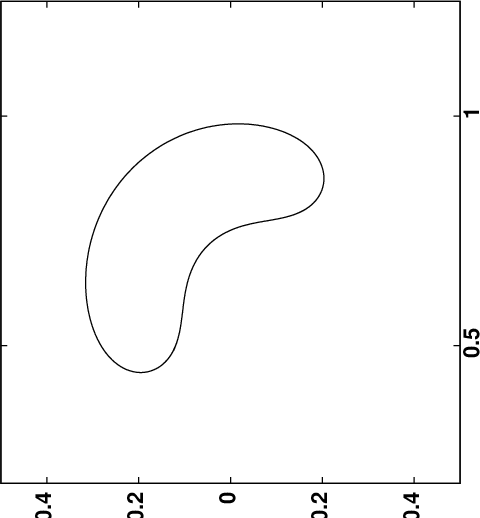}
\includegraphics[angle=-90,width=\localwidth]{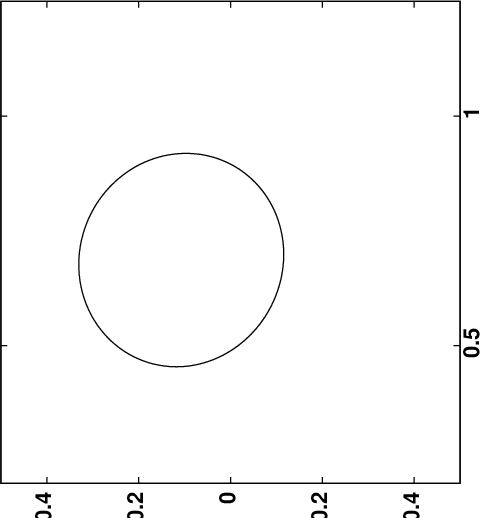}
\
\includegraphics[angle=-90,width=0.24\textwidth]{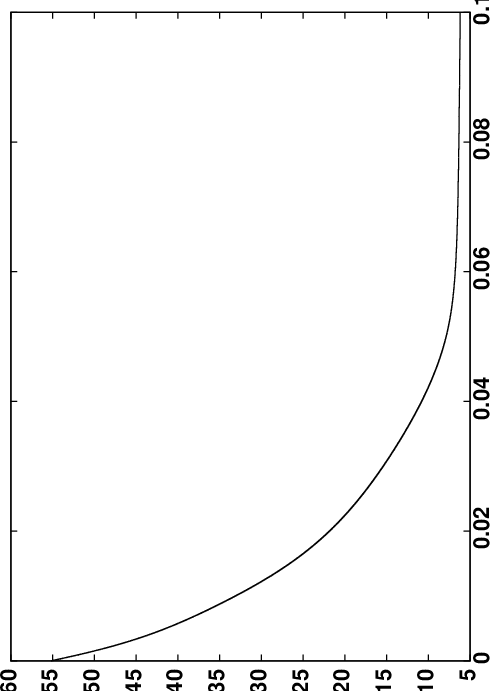}
\mbox{
\includegraphics[angle=-90,width=0.25\textwidth]{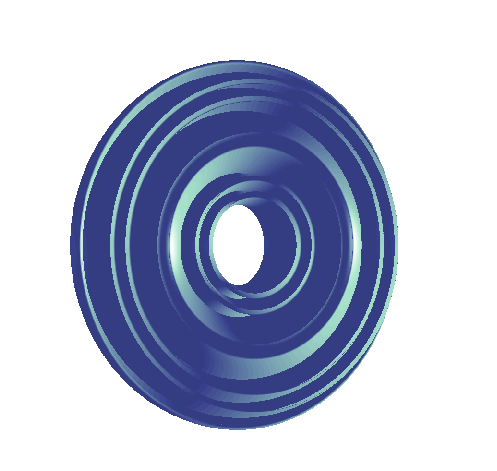}
\includegraphics[angle=-90,width=0.25\textwidth]{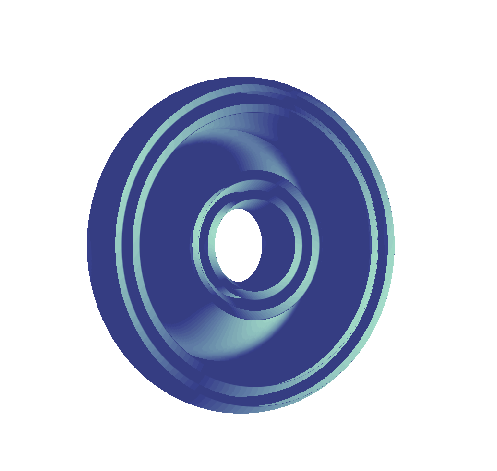}
\includegraphics[angle=-90,width=0.25\textwidth]{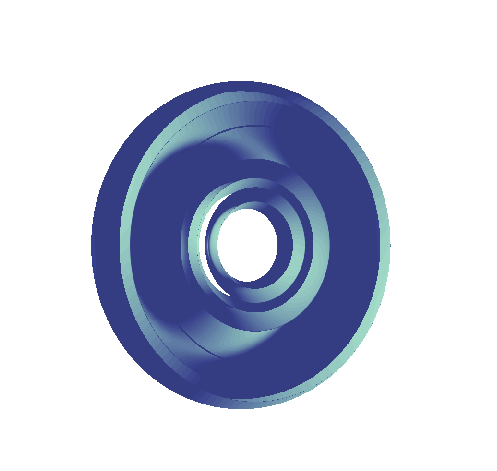}
\includegraphics[angle=-90,width=0.25\textwidth]{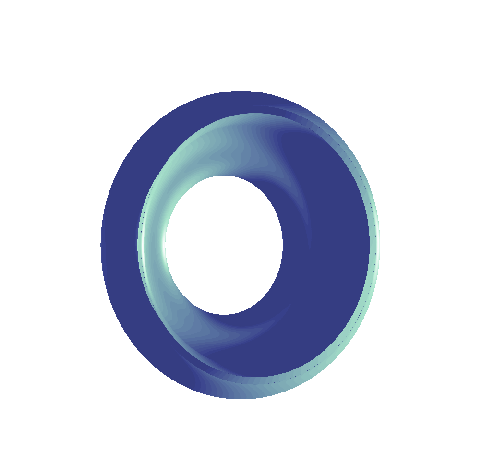}
}
\caption{$(\BGNmckappa_m^{f,V})^h$ for (\ref{eq:fmcf}). 
Conserved mean curvature flow. 
Plots are at times $t=0,0.01,\ldots,0.05,0.1$. 
We also show a plot of the discrete energy over time.
Below we visualize the part of the axisymmetric surface $\mathcal{S}^m$ 
generated by $\Gamma^m \cap \bR \times [-0.2,\infty)$ 
at times $t=0$, $t=0.01$, $t=0.02$ and $t=0.05$.}
\label{fig:spiral}
\end{figure}%

\subsection{Numerical results for nonlinear mean curvature flow} 
\label{sec:nlmcfnr}

Similarly to (\ref{eq:truer}), it is easy to show that a sphere of radius 
$r(t)$, with
\begin{equation} \label{eq:truerbeta}
r(t) = [1 - 2^\beta\,(\beta+1)\,t]^\frac1{\beta+1}\,,\quad r(0) = 1\,,
\end{equation}
is a solution to (\ref{eq:nlmcf}) with (\ref{eq:fbeta}). 
We use this true solution for a convergence test for $\beta = \tfrac12$, 
similarly to Table~2 in \cite{gflows3d}. 
Here we start with the nonuniform partitioning (\ref{eq:X0}) 
of a semicircle of radius $r(0)=r_0=1$ and compute the flow until time 
$T = \tfrac12\,\overline T$, where $\overline T =
\tfrac23\,2^{-\frac12}$ denotes that extinction time of the shrinking sphere.
We compute the error $\errorXx$, recall (\ref{eq:errorXx}),
over the time interval $[0,T]$ between the true solution (\ref{eq:truerbeta}) 
and the discrete solutions for the schemes
$(\BGNmckappa_m^f)^h$ and $(\BGNmc_{m,\star}^f)$.
Here we used the time step size $\ttau=0.1\,h^2_{\Gamma^0}$,
where $h_{\Gamma^0}$ is the maximal edge length of $\Gamma^0$.
The computed errors are reported in Table~\ref{tab:nlmcfbeta}.
\begin{table}
\center
\begin{tabular}{|rr|c|c|c|c|}
\hline
& & \multicolumn{2}{c|}{$(\BGNmckappa_m^f)^h$}&
\multicolumn{2}{c|}{$(\BGNmc_{m,\star}^f)$} \\
$J$ & $h_{\Gamma^0}$ & $\errorXx$ & EOC & $\errorXx$ & EOC \\ \hline
32   & 1.0792e-01 & 7.4955e-05 & --       & 3.0322e-03 & -- \\ 
64   & 5.3988e-02 & 1.8223e-05 & 2.041792 & 1.0450e-03 & 1.538013 \\          
128  & 2.6997e-02 & 4.5218e-06 & 2.011114 & 3.5931e-04 & 1.540449 \\ 
256  & 1.3499e-02 & 1.1282e-06 & 2.002981 & 1.2357e-04 & 1.539983 \\ 
512  & 6.7495e-03 & 2.8189e-07 & 2.000819 & 4.2698e-05 & 1.533088 \\ 
\hline
\end{tabular}
\caption{Errors for the convergence test for (\ref{eq:truerbeta})
over the time interval $[0,\tfrac12\,\overline{T}]$.}
\label{tab:nlmcfbeta}
\end{table}%

We repeat the same convergence experiment for the inverse mean curvature flow,
where we note that a sphere of radius $r(t)$, with
\begin{equation} \label{eq:truerIMCF}
r(t) = \exp(\tfrac12\,t)\,,\quad r(0) = 1\,,
\end{equation}
is a solution to (\ref{eq:nlmcf}) with (\ref{eq:fimcf}). 
The errors are reported in Table~\ref{tab:nlmcfimcf}. We recall that these
numbers can be compared to the fully 3d results in 
Table~3 in \cite{gflows3d}.
\begin{table}
\center
\begin{tabular}{|rr|c|c|c|c|}
\hline
& & \multicolumn{2}{c|}{$(\BGNmckappa_m^f)^h$}&
\multicolumn{2}{c|}{$(\BGNmc_{m,\star}^f)$} \\
$J$ & $h_{\Gamma^0}$ & $\errorXx$ & EOC & $\errorXx$ & EOC \\ \hline
32   & 1.0792e-01 & 7.1401e-04 & --       & 1.2445e-02 & --       \\ 
64   & 5.3988e-02 & 1.8106e-04 & 1.980959 & 4.7424e-03 & 1.392919 \\ 
128  & 2.6997e-02 & 4.5484e-05 & 1.993356 & 1.7539e-03 & 1.435281 \\ 
256  & 1.3499e-02 & 1.1388e-05 & 1.997952 & 6.3806e-04 & 1.458880 \\ 
512  & 6.7495e-03 & 2.8483e-06 & 1.999341 & 2.3002e-04 & 1.471933 \\ 
\hline
\end{tabular}
\caption{Errors for the convergence test for (\ref{eq:truerIMCF})
over the time interval $[0,1]$.}
\label{tab:nlmcfimcf}
\end{table}%
It is clear from Tables~\ref{tab:nlmcfbeta} and \ref{tab:nlmcfimcf}
that the solutions to the scheme $(\BGNmckappa_m^f)^h$
appear to converge with the optimal convergence rate 
of $\mathcal{O}(h^2_{\Gamma^0})$. For the scheme $(\BGNmc_{m,\star}^f)$, 
on the other hand, the solutions appear to 
converge with an order less than quadratic, and closer to $\frac32$. We believe
that these lower convergence rates 
are caused by the nonuniform meshes induced by the
scheme $(\BGNmc_{m,\star}^f)$, recall Figure~\ref{fig:mcfTM},
and also by the degeneracy of the coefficients 
$\vec x\,.\,\vec\ek_1$ in $(\BGNmc^f)$.

In the next experiment we repeat the simulation in \cite[Fig.\ 8]{gflows3d} 
for the inverse mean curvature of a torus with radii $R=1$, $r=0.25$.
We recall that for this nonconvex initial data, with 
$\varkappa_{\mathcal{S}}(\cdot,0)<0$, the classical inverse mean curvature 
develops a singularity in finite time, see also \cite{Gerhardt90,Urbas90}.
For the axisymmetric setting we use $I = \bR/\bZ$, so that $\partial I =
\emptyset$.
As the discretization parameters for the scheme $(\BGNmckappa_m^f)^h$
we use $J=256$ and $\ttau = 10^{-4}$.
See Figure~\ref{fig:imcftorus} for the simulation results. Similarly to the
results in \cite[Fig.\ 8]{gflows3d}, the discrete solution becomes unphysical
after around time $0.52$, where we conjecture that the singularity for the
continuous flow occurs.
\begin{figure}
\center
\includegraphics[angle=-90,width=0.35\textwidth]{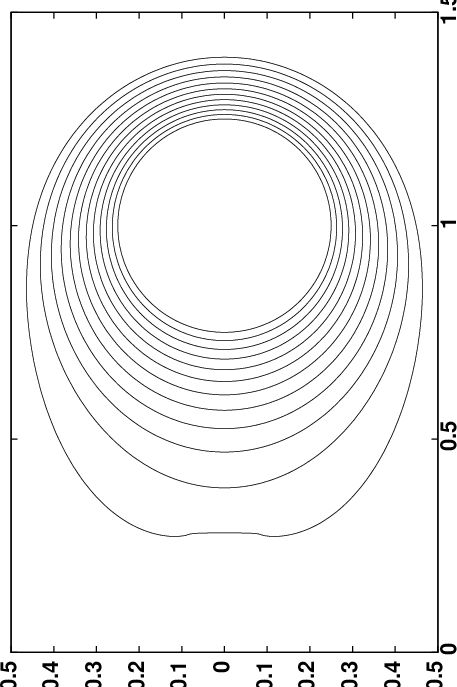}
\includegraphics[angle=-90,width=0.2\textwidth]{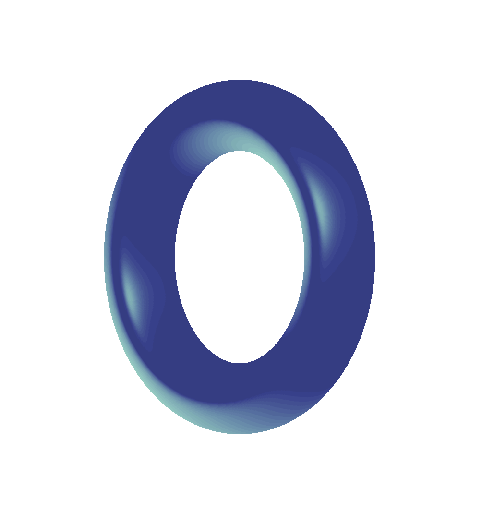}
\includegraphics[angle=-90,width=0.2\textwidth]{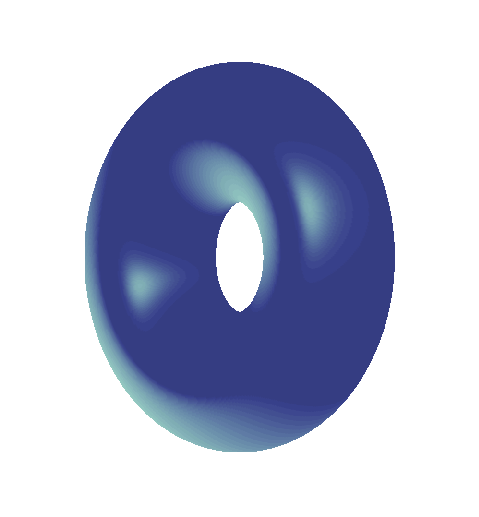}
\caption{$(\BGNmckappa_m^{f})^h$ for (\ref{eq:fimcf}). 
Inverse mean curvature flow for a torus with radii $R=1$, $r=0.25$. 
Plots are at times $t=0,0.05,\ldots,0.55$. 
We also visualize the axisymmetric surface $\mathcal{S}^m$ generated by
$\Gamma^m$ at times $t=0$ and $t=0.5$.}
\label{fig:imcftorus}
\end{figure}%

\subsection{Numerical results for Gauss curvature flow} 
\label{sec:gaussnr}

An experiment for Gauss curvature flow, (\ref{eq:Gaussflow}),
for the same initial data as in Figure~\ref{fig:mctallcigar}, can be
seen in
Figure~\ref{fig:gausstallcigar}.
Here we have once again that $\partial_0 I = \partial I = \{0,1\}$.
The discretization parameters for the scheme $(\BGNmckappa^F_{m})^{h}$
from Remark~\ref{rem:Fscheme} are $J=128$ and $\ttau = 10^{-5}$.
As a comparison, we also show the evolution for standard mean curvature
flow,
computed with the scheme $(\BGNmckappa_m)^{h}$, in 
Figure~\ref{fig:gausstallcigar}.
It was suggested by Firey, \cite{Firey74}, that surfaces of stones, which are
pounded by waves and other stones, move according to Gauss curvature flow. 
It is more likely that parts of the surface, where both principal
curvature directions are highly curved, will be hit by waves and other stones.
He hence proposed the Gauss curvature flow as the governing equation
for the evolution of the stone's surface.
In Figure~\ref{fig:gausstallcigar} it is clearly seen that
the upper and lower part, which have two highly curved principal curvature
directions, move faster within Gauss curvature flow when
compared to mean curvature flow. The parts closer to the origin have a nearly
flat principal curvature direction. Hence they move far slower under Gauss 
curvature flow than under mean curvature flow, 
as can be clearly seen in Figure~\ref{fig:gausstallcigar}.
\begin{figure}
\center
\includegraphics[angle=-90,width=0.1\textwidth]{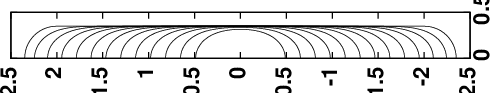}
\includegraphics[angle=-90,width=0.2\textwidth]{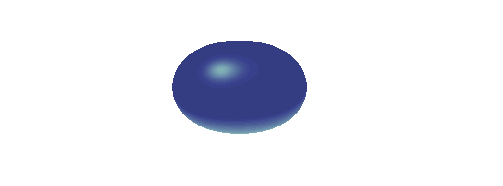}
\qquad\qquad
\includegraphics[angle=-90,width=0.1\textwidth]{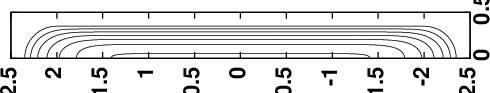}
\includegraphics[angle=-90,width=0.2\textwidth]{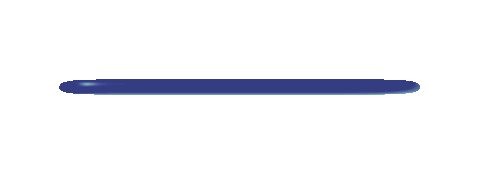}
\caption{
$(\BGNmckappa_m^{F})^h$ for (\ref{eq:Gaussflow}).
Gauss curvature flow for a cigar, on the left. Plots are at times
$t=0,0.01,\ldots,0.12$.
We also visualize the axisymmetric surface $\mathcal{S}^m$ generated by
$\Gamma^m$ at time $t=0.12$.
As a comparison, we show the evolution of $(\BGNmckappa_m)^h$
for mean curvature flow, (\ref{eq:mcfS}), on the right.
Here the plots are at times $t=0,0.01,\ldots,0.06$,
and we visualize the axisymmetric surface $\mathcal{S}^m$ generated by
$\Gamma^m$ at time $t=0.06$.
}
\label{fig:gausstallcigar}
\end{figure}%

The Gauss curvature flow is not well-defined for general hypersurfaces
as for non-convex hypersurfaces the resulting equation is not parabolic,
see \cite{Andrews00,Jeffres09}. 
In the axisymmetric situation the degrees
of freedom are reduced and the resulting equation is
\[
\vec x_t\,.\,\vec\nu =
\varkappa\,\frac{\vec\nu\,.\,\vec\ek_1}{\vec x\,.\,\vec\ek_1}
\quad\text{on }\ I\,,
\]
which is parabolic as long as ${\vec\nu\,.\,\vec\ek_1}$ is positive.
In conclusion,  even in the axisymmetric case the evolution
is not well-defined if the initial surface
has the topology of a torus, and
so we do not present results for genus 1 surfaces. 
However, although this would not be possible in the general formulation
 we can start the
Gauss curvature flow in the axisymmetric case with some nonconvex initial 
data, and we do so in the simulation in Figure~\ref{fig:gaussnonconvex2},
where we used the discretization parameters $J=128$ and $\ttau = 10^{-5}$.
For a mathematical analysis for Gauss curvature flow in the axisymmetric
case we refer to \cite{Jeffres09}.
\begin{figure}
\center
\includegraphics[angle=-90,width=0.1\textwidth]{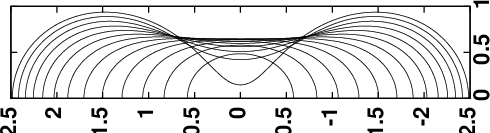} 
\qquad
\includegraphics[angle=-90,width=0.15\textwidth]{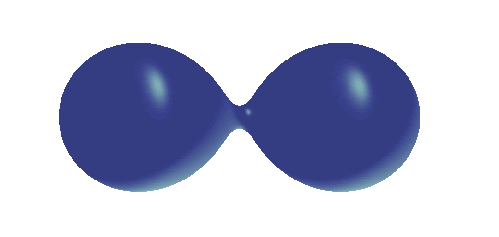} 
\includegraphics[angle=-90,width=0.15\textwidth]{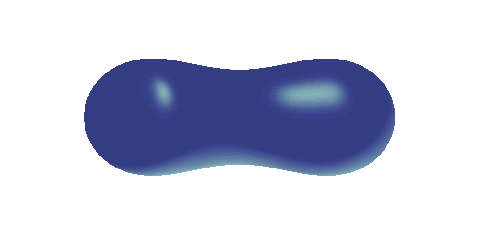} 
\includegraphics[angle=-90,width=0.15\textwidth]{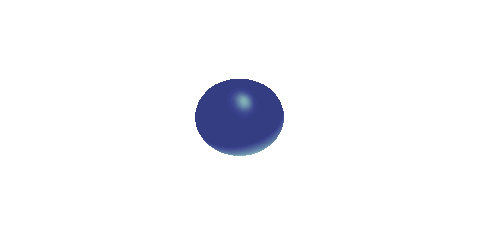} 
\caption{
$(\BGNmckappa_m^{F})^h$ for (\ref{eq:Gaussflow}). 
Gauss curvature flow for nonconvex initial data. Plots are at times
$t=0,0.05,\ldots,0.6$.
We also visualize the axisymmetric surface $\mathcal{S}^m$ generated by
$\Gamma^m$ at times $t=0$, $t=0.2$ and $t=0.6$.
} 
\label{fig:gaussnonconvex2}
\end{figure}%

\section*{Conclusions}
We have derived and analysed various numerical schemes for the parametric
approximation of axisymmetric mean curvature flow, its nonlinear and
volume conserving variants, as well as more general curvature flows.
The main fully discrete schemes to consider for standard mean curvature flow 
are $(\BGNmckappa_m)^h$ and $(\BGNmc_{m,\star})$. Here we have dismissed the
scheme $(\GDmckappa_m)^h$, as it can show oscillations in practice, 
recall Figure~\ref{fig:torusR1r05}, as well as
the scheme $(\BGNmc_{m,\star})^h$, as it can display very nonuniform meshes in
simulations where the discrete curves are attached to the $x_2$--axis,
recall Figure~\ref{fig:mcfTM} for its variant $(\BGNmc_{m,\star}^{f,V})^h$.
We also do not consider the schemes $(\GDmc_{m,\star})^{(h)}$, as they have no
advantage over $(\BGNmc_{m,\star})$ and as they can also exhibit very
nonuniform meshes.
Of the two schemes we consider, 
the scheme $(\BGNmckappa_m)^h$ is a linear scheme
that asymptotically leads to an equidistribution of
mesh points, recall Remark~\ref{rem:equid}.
In addition, even though there is no stability proof for 
$(\BGNmckappa_m)^h$, in practice the discrete energy is always monotonically
decreasing.
The scheme $(\BGNmc_{m,\star})$, on the other hand, 
is a nonlinear scheme that is unconditionally stable.
The nonlinearity is only very mild,
and so a Newton solver never takes more than 3 iterations in practice.
Moreover,
the distribution of vertices for $(\BGNmc_{m,\star})$ may be worse than for 
$(\BGNmckappa_m)^h$, but coalescence of vertices is not observed in practice.
Similar statements hold for the nonlinear variants
$(\BGNmckappa_m^f)^h$, $(\BGNmckappa_m^{f,V})^h$, $(\BGNmc_{m,\star}^f)$ and
$(\BGNmc_{m,\star}^{f,V})$, where the two conserving schemes show very good
volume conservation properties in practice.
Finally, for general curvature flows of the form (\ref{eq:Fmg}), we propose the
linear scheme $(\BGNmckappa_m^F)^h$, which asymptotically exhibits
equidistributed mesh points.

\begin{appendix}
\renewcommand{\theequation}{\Alph{section}.\arabic{equation}}
\setcounter{equation}{0}
\section{Derivation of (\ref{eq:bc})} \label{sec:A1}
Here we demonstrate that (\ref{eq:bgnnewb}) and (\ref{eq:Dziuknewb}) weakly
impose (\ref{eq:bc}). First we consider
(\ref{eq:bgnnewb}) and the case $\rho_0 = 0 \in \partial_0 I$.

We assume for almost all $t\in(0,T)$ that
$\vec x(t) \in [C^1(\overline I)]^2$ and $\varkappa_S(t) \in L^\infty(I)$.
These assumptions and (\ref{eq:xrho}) imply that 
\begin{equation} \label{eq:x1bound}
C_1\,\rho \leq |\vec x(\rho,t)\,.\,\vec\ek_1 | \leq C_2\,\rho
\qquad \forall\ \rho \in [0,\overline\rho]\,,
\end{equation}
for $\overline\rho$ sufficiently small, and for almost all $t\in(0,T)$.

Let $t \in (0,T)$.
For a fixed $\overline\rho > 0$ and $\epsilon \in (0,\overline\rho)$, we define
\begin{equation*} 
\vec\eta_\epsilon(\rho) = \begin{cases}
(\overline\rho)^{-1}\,\int_\epsilon^{\overline\rho} (\vec
x(z,t)\,.\,\vec\ek_1)^{-1}\,\vec\ek_2 \;{\rm d}z & 0 \leq \rho < \epsilon\,, \\
(\overline\rho)^{-1}\,\int_\rho^{\overline\rho} (\vec
x(z,t)\,.\,\vec\ek_1)^{-1}\,\vec\ek_2 \;{\rm d}z 
& \epsilon \leq \rho < \overline\rho\,, \\
\vec 0 & \overline\rho \leq \rho\,.
\end{cases}
\end{equation*}
It follows from (\ref{eq:x1bound}) that 
$(\vec x\,.\,\vec\ek_1)\,\vec\eta_\epsilon$ is integrable in the limit
$\epsilon \to 0$.
On choosing 
$\vec\eta = \vec\eta_\epsilon \in \Vpartial$ in (\ref{eq:bgnnewb}), we 
obtain in the limit $\epsilon \to 0$ that
\begin{equation} \label{eq:app0}
(\overline\rho)^{-1}\,
\int_0^{\overline\rho} \vec x\,.\,\vec\ek_1\,\varkappa_{\mathcal{S}}\,
\vec\ek_2\,.\,\vec\nu
\left(\int_\rho^{\overline\rho} (\vec x\,.\,\vec\ek_1)^{-1} \;{\rm d}z\right)
|\vec x_\rho| \drho
= (\overline\rho)^{-1}\,\int_0^{\overline\rho} 
\vec x_\rho\,.\,\vec\ek_2\,|\vec x_\rho|^{-1} \drho \,.
\end{equation}
Applying Fubini's theorem and noting (\ref{eq:x1bound}), as
well as the boundedness of $|\vec x_\rho|$ and $\varkappa_{\mathcal{S}}$, 
yields the existence of a constant $M$ such that
\begin{align}
& 
\left|
(\overline\rho)^{-1}\,
\int_0^{\overline\rho} \vec x\,.\,\vec\ek_1\,\varkappa_{\mathcal{S}}\,
\vec\ek_2\,.\,\vec\nu
\left(\int_\rho^{\overline\rho} (\vec x\,.\,\vec\ek_1)^{-1} \;{\rm d}z\right)
|\vec x_\rho| \drho
 \right| \nonumber \\ & \quad
= \left| (\overline\rho)^{-1}\,
\int_0^{\overline\rho} 
(\vec x\,.\,\vec\ek_1)^{-1}
\left(\int_0^z
\vec x\,.\,\vec\ek_1\,\varkappa_{\mathcal{S}}\,\vec\ek_2\,.\,\vec\nu
\,|\vec x_\rho| \drho \right) {\rm d}z\right| \nonumber \\ & \quad
\leq (\overline\rho)^{-1}\,M\,\int_0^{\overline\rho} 
z^{-1}\left(\int_0^z \rho \drho \right) {\rm d}z
= \tfrac12\,(\overline\rho)^{-1}\,M\,\int_0^{\overline\rho} z \;{\rm d}z
= \tfrac14\,M\,\overline\rho \to 0 \quad \text{as }\ \overline\rho\to0\,.
\label{eq:fubini0}
\end{align}
On the other hand, the right hand side in (\ref{eq:app0}) converges to
$(\vec x_\rho(0,t)\,.\,\vec\ek_2)\,|\vec x_\rho(0,t)|^{-1}$
as $\overline\rho\to0$, on recalling the smoothness assumptions on $\vec x$. 
Combining this with (\ref{eq:fubini0}) and 
(\ref{eq:xrho}) yields the boundary condition 
(\ref{eq:bc}) for $\rho = 0 \in \partial_0 I$. The proof for
$\rho = 1 \in \partial_0 I$ is analogous.
Finally, the proof for (\ref{eq:Dziuknewb}) is easily adapted from the above, 
on assuming that $\vec\varkappa_{\mathcal{S}}(t) \in [L^\infty(I)]^2$ for 
almost all $t\in(0,T)$.

\setcounter{equation}{0}
\section{Existence proof for $(\BGNmckappa_m^f)^h$ and 
$(\BGNmckappa_m^{f,V})^h$} \label{sec:B}
We adapt \cite[(2.12)--(2.14)]{gflows3d} to 
$(\BGNmckappa_m^f)^h$ and $(\BGNmckappa_m^{f,V})^h$. 
\begin{thm} \label{thm:B1}
Let $\vec X^m \in \Vhpartialzero$ satisfy the assumptions
$(\mathfrak A)$ and  $(\mathfrak B)^h$, and
assume that $f:(a,b)\to\bR$ with $-\infty\leq a<0<b\leq\infty$
is strictly monotonically increasing, continuous and
such that $f((a,b))=\mathbb{R}$. 
If $b = -a = \infty$,
then there exists a solution 
$(\delta\vec X^{m+1},\kappa^{m+1}) \in \Vhpartial\times V^h$
to $(\BGNmckappa_m^f)^h$ and $(\BGNmckappa_m^{f,V})^h$.
Moreover, for general $a<b$ there exists at most one solution.
\end{thm}
\begin{proof}
Let $f^{-1} : \bR \to (a,b)$ denote the inverse of $f$.
It follows from (\ref{eq:nlfda}) and (\ref{eq:nlVfda}), 
on recalling (\ref{eq:betam}), that
\begin{equation} \label{eq:kappafinv}
\lambda(q_j)\,\kappa^{m+1}(q_j) = 
f^{-1} \left[\left(\frac{\delta\vec X^{m+1}}{\ttau_m} \,.\,\vec\omega^m
\right) (q_j) + g_1^m \right] + g_0^m(q_j)
\qquad j = 0,\ldots,J\,,
\end{equation}
where $g_1^m \in\bR$ and $g_0^m \in W^h_{\partial_0} \subset V^h$ 
are independent of $\delta\vec X^{m+1}$ and $\kappa^{m+1}$. 
In particular, $g_1^m = 0$ for $(\BGNmckappa_m^f)^h$.
Substituting (\ref{eq:kappafinv}) into (\ref{eq:fdb}) yields 
\begin{align}
& \left(\lambda^{-1}\,
f^{-1} \left(\frac{\delta\vec X^{m+1}\,.\,\vec\omega^m}{\ttau_m} 
+ g_1^m \right) , \vec\omega^m\,.\,\vec\eta\,|\vec X^m_\rho|\right)^h
+ \left(\delta\vec X^{m+1}_\rho, \vec\eta_\rho\,|\vec X^m_\rho|^{-1}\right) 
= \ell^m(\vec\eta) 
\nonumber \\ & \hspace{11cm}
\qquad \forall\ \vec\eta \in \Vhpartial\,,
\label{eq:Xfinv}
\end{align}
where $\ell^m : \Vh \to \bR$ is a linear functional defined by
\begin{equation*} 
\ell^m(\vec\eta) = 
- \left(\vec X^{m}_\rho, \vec\eta_\rho\,|\vec X^m_\rho|^{-1}\right) 
- \left(g_0^m\,\vec\omega^m, \vec\eta\,|\vec X^m_\rho|\right)^h
- \sum_{i=1}^2 
 \sum_{p \in \partial_i I} \sliprho^{(p)}\,\vec\eta(p)\,.\,\vec\ek_{3-i}\,.
\end{equation*}
It follows that (\ref{eq:Xfinv}) is the Euler--Lagrange variation of the
minimization problem:
\begin{subequations}
\begin{align}
& \min_{\vec{\eta} \in \Vhpartial}
\mathcal{J}^h(\vec\eta)\,,\label{eq:minJ} \\ 
& \mathcal{J}^h(\vec\eta) :=
 \tfrac12 \left(|\vec\eta_\rho|^2, |\vec X^m_\rho|^{-1}\right) 
+ \ttau_m\,\left( \lambda^{-1}\, \Phi 
\left(\frac{\vec\eta\,.\,\vec\omega^m}{\ttau_m} + g_1^m \right) ,
|\vec X^m_\rho|\right)^h - \ell^m(\vec\eta)\,,
\label{eq:calJ}
\end{align}
\end{subequations}
where $\Phi \in C^1(\bR)$ 
denotes an antiderivative of $f^{-1}$. We note that $\Phi:
\bR\to\bR$ is strictly convex with $\Phi'(f(0))=f^{-1}(f(0))=0$ and 
hence we obtain that $\Phi$ is bounded from below and is coercive.

In the following we establish that the continuous functional 
$\mathcal{J}^h : \Vhpartial\to\bR$
is coercive, i.e.\ that $\mathcal{J}^h(\vec\eta) \to \infty$ as 
$\|\vec\eta\| \to \infty$,
where $\|\cdot\|$ is a fixed norm on $\Vh$.
The main task is to bound the growth of the
linear term $\ell^m(\vec\eta)$ in terms of the first two terms in 
(\ref{eq:calJ}). 

If $b = -a = \infty$, it is possible to show that for all $N \in \bN$ there
exists a positive constant $C_0(N)$ such that
\begin{equation} \label{eq:Phiassum}
\Phi(r) \geq N\,|r| - C_0(N) \qquad \forall\ r \in \bR\,.
\end{equation}
To see this, for $N\in\bN$ choose an $R\in\bR$ such that
$\min\{f(R),-f(-R)\} \geq N$, and define 
$\|f\|_{\infty,R} := \max_{q\in[-R,R]} |f(q)|$. On assuming
without loss of generality that $\Phi(r) = \int_0^r f(q)\; {\rm d}q$, it
holds for $r > R$ that
\begin{equation} \label{eq:HG1}
\Phi(r) = \int_0^{R} f(q) \; {\rm d}q + \int_{R}^r f(q) \; {\rm d}q
\geq - R\,\|f\|_{\infty,R} + (r - R)\,N 
= r\,N - (\|f\|_{\infty,R} + N)\,R\,.
\end{equation}
In addition, for $r \in [0,R]$ it trivially holds that
\begin{equation} \label{eq:HG2}
\Phi(r) \geq - R\,\|f\|_{\infty,R} + r\,N - r\,N
\geq r\,N - (\|f\|_{\infty,R} + N)\,R\,.
\end{equation}
Combining (\ref{eq:HG1}) and (\ref{eq:HG2}) yields (\ref{eq:Phiassum}) 
for $r \geq 0$. The case $r\leq0$ can be treated analogously.

Given $\vec\eta \in \Vhpartial$, we define 
$\vec\zeta = \vec\eta + \vec{\mathfrak f}^m \in \Vh$ with
$\vec{\mathfrak f}^m = \ttau_m\,g_1^m\,\vec\pi^h[
|\vec\omega^m|^{-2}\,\vec\omega^m] \in \Vh$.
Then it holds for every $N\in\bN$ that
\begin{align} 
\mathcal{J}^h(\vec\eta) & = \mathcal{J}^h(\vec\zeta - \vec{\mathfrak f}^m) 
\nonumber \\ & =
\tfrac12 \left(|\vec\zeta_\rho - \vec{\mathfrak f}^m_\rho|^2, 
|\vec X^m_\rho|^{-1}\right) 
+ \ttau_m\,\left( \lambda^{-1}\, \Phi 
\left(\frac{\vec\zeta\,.\,\vec\omega^m}{\ttau_m} \right) ,
|\vec X^m_\rho|\right)^h - \ell^m(\vec\zeta - \vec{\mathfrak f}^m)
\nonumber \\ & \geq
\tfrac14 \left(|\vec\zeta_\rho|^2, |\vec X^m_\rho|^{-1}\right) 
+  N \left( \lambda^{-1}\, |\vec\zeta\,.\,\vec\omega^m| ,
|\vec X^m_\rho|\right)^h  - \ell^m(\vec\zeta) 
- C_1(N) \,,
\label{eq:Jbound}
\end{align}
where, here and throughout, constants of the form $C_i$ are independent of
$\vec\zeta$, but may depend on the data $\vec X^m$, $\vec{\mathfrak f}^m$ etc.
Similarly, constants of the form $C_i(N)$ may also depend on $N$,
recall (\ref{eq:Phiassum}), but are independent of $\vec\zeta$. On defining 
$\mints\eta = \frac{(\eta , |\vec X^m_\rho|)}{(1, |\vec X^m_\rho|)}$, 
and extending the definition to vector valued functions, it
follows from (\ref{eq:Jbound}) that
\begin{align}
\mathcal{J}^h(\vec\eta) & \geq
\tfrac14 \left(|\vec\zeta_\rho|^2, |\vec X^m_\rho|^{-1}\right) 
+  N \left( \lambda^{-1}\, |\mint\vec\zeta\,.\,\vec\omega^m| ,
|\vec X^m_\rho|\right)^h  
\nonumber \\ & \qquad 
- N \left( \lambda^{-1}\, |(\vec\zeta - \mint\vec\zeta)\,.\,\vec\omega^m| ,
|\vec X^m_\rho|\right)^h 
- C_2\,\|\vec\zeta - \mint\vec\zeta\| - C_2\,\|\mint\vec\zeta\| - C_1(N)
\nonumber \\ & =
I + II - III - IV - V - C_1(N)\,.
\label{eq:Jbound2}
\end{align}
It remains to bound $-III-IV-V$ from below.
We have from the assumption $(\mathfrak B)^h$ that 
$II - V \geq N\,C_3\,|\mint\vec\zeta| -
C_4\,|\mint\vec\zeta|$. Choosing $N \geq 2\,C_4 / C_3$ implies that
\begin{equation} \label{eq:bound25}
II-V \geq C_4\,|\mint\vec\zeta|\,.
\end{equation}
In addition, it holds  that
\begin{align}
III + IV & \leq (N+1)\,C_5\,\|\vec\zeta - \mint\vec\zeta\| 
\leq (N+1)\,C_6\,\|\vec\zeta_\rho\| \nonumber \\ &
\leq (N+1)\,\delta\,\left(|\vec\zeta_\rho|^2, |\vec X^m_\rho|^{-1}\right) 
+ C_7(N,\delta) \qquad \forall\ \delta \in (0,\infty)\,.
\label{eq:bound34}
\end{align}
Choosing $(N+1)\,\delta \leq \tfrac18$ in (\ref{eq:bound34}) and combining with
(\ref{eq:Jbound2}) and (\ref{eq:bound25}) implies that
\begin{align*} 
\mathcal{J}^h(\vec\eta) & \geq \tfrac18 
\left(|\vec\zeta_\rho|^2, |\vec X^m_\rho|^{-1}\right) 
+ C_4\,|\mint\vec\zeta| - C_8(N) \nonumber \\ &
\geq C_9\,\|\vec\zeta\| - C_{10}(N) \geq C_9\,\|\vec\eta\| - C_{11}(N)
\,,
\end{align*}
which proves the coercivity of $\mathcal{J}^h(\vec\eta)$.

We now consider the uniqueness of a solution to (\ref{eq:Xfinv}). Let
$\delta\vec X^{(i)} \in \Vhpartial$, $i=1,2$ 
be two solutions to (\ref{eq:Xfinv}). Then they satisfy
\begin{align}
& \left( f^{-1} \left(\frac{\delta\vec X^{(1)}\,.\,\vec\omega^m}{\ttau_m} 
+ g_1^m \right) 
- 
f^{-1} \left(\frac{\delta\vec X^{(2)}\,.\,\vec\omega^m}{\ttau_m} 
+ g_1^m \right), \lambda^{-1}\,\vec\omega^m\,.\,[\delta\vec X^{(1)} -\delta\vec X^{(2)}]
\,|\vec X^m_\rho|\right)^h
\nonumber \\ & \quad
+ \left(|[\delta\vec X^{(1)} -\delta\vec X^{(2)}]_\rho|^2, 
|\vec X^m_\rho|^{-1}\right) = 0\,.
\label{eq:X12finv}
\end{align}
As $f^{-1}$ is strictly monotonically increasing it immediately follows from
(\ref{eq:X12finv}) that $\delta\vec X^{(1)} -\delta\vec X^{(2)} = \vec X^c \in
\bR^2$, and hence, on recalling (\ref{eq:ip0}), that
\begin{align*}
& \left(f^{-1} \left[\left(\frac{\delta\vec X^{(1)}\,.\,\vec\omega^m}{\ttau_m}
\right)(q_j) + g_1^m \right]
-  f^{-1} \left[ \left(\frac{\delta\vec X^{(2)}\,.\,\vec\omega^m}{\ttau_m} 
\right)(q_j) + g_1^m \right] \right)
\vec X^c\,.\,\vec\omega^m(q_j) = 0 \nonumber \\ &
\hspace{10cm} \qquad \forall\ j = 0,\ldots,J\,. 
\end{align*}
Now the strict monotonicity of $f^{-1}$ implies that 
$\vec X^c\,.\,\vec\omega^m(q_j) = 0$ for all $j=0,\ldots,J$,
and so the assumption $(\mathfrak B)^h$ yields that $\vec X^c=\vec0$.
This shows the uniqueness of a solution to 
$(\BGNmckappa_m^f)^h$ and $(\BGNmckappa_m^{f,V})^h$. 
\end{proof}

Theorem~\ref{thm:B1} yields existence of a unique solution for the schemes
$(\BGNmckappa_m^f)^h$ and $(\BGNmckappa_m^{f,V})^h$ in the case
(\ref{eq:fbeta}). For the case (\ref{eq:fimcf}) we only obtain uniqueness of a
solution.

\end{appendix}

\section*{Acknowledgements}
The authors gratefully acknowledge the support 
of the Regensburger Universit\"atsstiftung Hans Vielberth.

\begin{spacing}{0.94}

\end{spacing}

\end{document}